\documentclass[a4paper,12pt]{article}

\usepackage{amsthm,amsmath,stmaryrd,bbm,hyperref,geometry,color,authblk}
\usepackage[utf8]{inputenc}
\usepackage{amssymb}
\usepackage[english]{babel}
\usepackage{graphicx}
\usepackage{amsfonts,amssymb}
\usepackage{verbatim}
\usepackage{enumitem}
\usepackage{graphicx}
\usepackage{subcaption}
\usepackage{float}
\usepackage{xcolor}
\definecolor{purple}{rgb}{0.5, 0, 0.5}

\newcommand{\po}{\left(}
\newcommand{\pf}{\right)}

\newcommand{\cco}{\llbracket}
\newcommand{\ccf}{\rrbracket}
\newcommand{\R}{\mathbb R} 
\newcommand{\E}{\mathbb E}
\newcommand{\F}{\mathcal{F}} 
 
\newcommand{\T}{\mathbb T}
\newcommand{\mH}{\mathcal{H}}
\newcommand{\Z}{\mathbb Z} 
\newcommand{\N}{\mathbb N} 
\newcommand{\dd}{\text{d}}

\newcommand{\na}{\nabla}
 
\newcommand{\pierre}[1]{\textcolor{blue}{(P: #1)}}

\newtheorem{theorem}{Theorem}

\newtheorem*{assu*}{Assumption}
\newtheorem{lem}[theorem]{Lemma}
\newtheorem{defi}[theorem]{Definition}
\newtheorem{cor}[theorem]{Corollary}
\newtheorem{prop}[theorem]{Proposition}
\newtheorem{rem}{Remark}

\title{Convergence rates for an Adaptive Biasing Potential scheme from a Wasserstein optimization perspective}

\author[1]{Tony Lelièvre}
\author[2]{Xuyang Lin}
\author[3,4]{Pierre Monmarché}
\affil[1]{CERMICS, ENPC, Institut Polytechnique de Paris, and Inria, France}
\affil[2]{École Polytechnique}
\affil[3]{Sorbonne Université,  Laboratoire Jacques-Louis Lions \& Laboratoire de Chimie théorique, LJLL \& LCT, F-75005 Paris}
\affil[4]{Institut Universitaire de France}

\begin{document}

\maketitle

\begin{abstract}
Free-energy-based adaptive biasing methods, such as
Metadynamics, the Adaptive Biasing Force (ABF) and their variants, are enhanced sampling algorithms widely used in molecular simulations. Although their efficiency has been empirically acknowledged for decades, providing theoretical insights  via a quantitative convergence analysis is a difficult problem, in particular for the 
kinetic Langevin diffusion, which is non-reversible and hypocoercive.
We obtain the first exponential convergence result for such a process, in an idealized setting where the dynamics can be associated with a mean-field non-linear flow on the space of probability measures. A key of the analysis is the interpretation of the (idealized) algorithm as the gradient descent of a suitable functional over the space of probability distributions.
\end{abstract}

\section{Introduction}

\subsection{Sampling with Langevin diffusion}\label{sec:intro-motiv}

Given a temperature $\beta^{-1}>0$ and a potential energy function $U:\T^d \to \R$ (where $\T^d$ is the $d$-dimensional torus) such that $$U\in C^\infty(\T^d)$$ we consider  the Boltzmann-Gibbs probability measure with density\footnote{For simplicity, we choose a setting with periodic boundary conditions -which are actually used in many situations of practical interest in the context of molecular dynamics-, in order to avoid technical discussions about the behaviour at infinity of $U$. Generalizations to the state space $\R^d$ are possible, see Remark~\ref{re:T->R}}.
\begin{equation}\label{eq:def_gibbs}
\rho^*  = \frac{e^{-\beta U}}{\mathcal Z},    
\end{equation}
with $\mathcal Z:= \int_{\T^d} e^{-\beta U}$. We are interested in sampling $\rho^*$, namely in generating (approximately) independent random variables distributed according to~$\rho^*$. These samples can for example be used to estimate the expectation of some observables with respect to~$\rho^*$ by empirical averages. This question is ubiquitous in Molecular dynamics (MD)~\cite{FreeEnergy,ActaNumerica}, Bayesian statistics~\cite{robert-casella-2004} or Machine Learning~\cite{bach-2024}. Samples are typically obtained by Markov Chain Monte Carlo (MCMC) methods, that is the simulation of a Markov chain which is ergodic with respect to $\rho^*$. There are two types of numerical errors associated with such simulations: deterministic errors (biases) related to the fact that it takes a long time to reach stationarity, or that the measure actually sampled by the Markov chain slightly differs from $\rho^*$ (because of time discretization errors for example), and statistical errors related to the variances of the estimators.  In many cases, in particular in MD for physical motivations, the Markov process is a kinetic dynamics (with a position $x$ and a velocity $v$) ergodic with respect to the product measure in position and velocity
\begin{equation}\label{eq:def_gibbs_kin}
\nu^* = \rho^* \otimes \mathcal N(0, \beta^{-1} I_d) \propto e^{-\beta H}\,,
\end{equation}
where $\mathcal N(m,\Sigma)$ denotes the Gaussian distribution with mean $m\in\R^d$ and covariance matrix $\Sigma \in \R^{d\times d}$, and $H(x,v)=U(x)+|v|^2/2$ is called the Hamiltonian\footnote{The mass is here chosen to be the identity for simplicity. Generalizations to a general mass matrix are straightforward.}. Most MD simulations are based on the (kinetic, or underdamped) Langevin diffusion process
\begin{equation}
    \label{eq:kineticLangevin}
    \left\{\begin{aligned}
   \dd X_t  & =  V_t \, \dd t   \\
    \dd V_t & =  -\na U(X_t) \, \dd t -\gamma V_t \, \dd t + \sqrt{2\gamma/\beta} \, \dd B_t\,,
\end{aligned}\right.
\end{equation} 
 with $\gamma>0$ a friction parameter and $(B_t)_{t \ge 0}$ a $d$-dimensional Brownian motion. When $\gamma$ goes to infinity, $X_{\gamma t}$ converges in law to the so-called  overdamped Langevin diffusion
 \begin{equation}
     \label{eq:overdampedLangvein}
     \dd Y_t= - \na U(Y_t) \dd t + \sqrt{2/\beta} \dd B_t\,,
 \end{equation} 
 which is also a popular sampler for $\rho^*$. Under mild assumptions on $U$, the kinetic and overdamped processes are known to be ergodic with respect to their equilibria, respectively $\nu^*$ and~$\rho^*$. To assess the practical efficiency of these samplers, it is important to quantify their rates of convergence to equilibrium. When the rates are too small, the laws of $X_t$ or $Y_t$ are poor approximations of $\rho^*$ for all times $t$ reachable in a simulation. Moreover, the asymptotic variances of time averages over these dynamics are also related to these rates of convergence: the larger the rate, the smaller the asymptotic variance.

 Obtaining convergence rates for the overdamped process~\eqref{eq:overdampedLangvein} (which is an elliptic reversible diffusion process) is a very classical topic, for which many tools have been developed, depending on the assumptions on $U$ or on the (pseudo)distance used to quantify the convergence. In particular, denoting by $\rho_t$ the law of $Y_t$, the exponential convergence at some rate $\lambda>0$,
 \begin{equation}\label{eq:CV_OL}
     \forall t\geqslant 0\,,\qquad  \mathcal H(\rho_t|\rho^*) \leq e^{-2\lambda t} \,  \mathcal H(\rho_0|\rho^*)\,,
 \end{equation} 
for any initial condition $\rho_0$ with finite relative entropy with respect to $\rho^*$, is known to be equivalent to the so-called log-Sobolev inequality with constant $\lambda$,
\begin{equation}\label{eq:LSI}
\forall \rho \ll \rho^*\,,\qquad \mathcal H(\rho|\rho^*) \leq \frac{1}{2\lambda}\mathcal I(\rho|\rho^*)\,, 
\end{equation}
where 
\begin{equation}\label{eq:entropy}
\mathcal H(\nu_1|\nu_2) = \int_{\T^d} \ln\left(\frac{\dd\nu_1}{\dd\nu_2}\right)\nu_1
\end{equation} is the relative entropy of $\nu_1$ with respect to $\nu_2$  and 
\begin{equation}\label{eq:fisher}
\mathcal I(\nu_1|\nu_2) = \int_{\T^d} \left|\na \ln \left(\frac{\dd\nu_1}{\dd\nu_2}\right)\right|^2 \nu_1
\end{equation}
stands for the Fisher information  of $\nu_1$ with respect to $\nu_2$, see \cite{BakryGentilLedoux}. A log-Sobolev inequality is known to hold for a smooth potential $U$ defined on the torus. The largest $\lambda$ for which~\eqref{eq:LSI} holds is called the optimal log-Sobolev constant of $\rho^*$. When~\eqref{eq:LSI} holds, in order to achieve convergence, simulations over time horizons of order at least $1/\lambda$ are required, which is prohibitive when the log-Sobolev constant $\lambda$ is very small, a situation which is typically encountered in many practical cases of interest, in particular in MD, as will be discussed below.

In the kinetic case~\eqref{eq:kineticLangevin}, the stochastic process is non-reversible and non-elliptic which makes the analysis more difficult. Nevertheless, it is known since Villani's work on hypocoercivity~\cite{Villani2009} that, assuming that $\rho^*$ satisfies a log-Sobolev inequality and that $\na^2 U$ is bounded, an exponential entropy decay 
\begin{equation}\label{eq:CV_Lang}
   \forall t\geqslant 0\,,\qquad  \mathcal H(\nu_t|\nu^*) \leq C e^{-\lambda' t} \mathcal H(\nu_0|\nu^*)\,, 
\end{equation}
 holds for some $C,\lambda'>0$, where $\nu_t$ is the law of $(X_t,V_t)$ solution to~\eqref{eq:kineticLangevin}. The constants $C,\lambda'$ are explicit in terms of $\gamma,\|\na^2 U\|_\infty$ and the log-Sobolev constant $\lambda$. At fixed $\gamma$ and $\|\na^2 U\|_\infty$, the rate $\lambda'$ obtained with Villani's modified entropy method is proportional to $\lambda$ when $\lambda$ is small\footnote{In some specific cases, in particular when $U$ is convex -- which is not the focus of the present work --, the optimal convergence rate of the kinetic process scales like $\sqrt{\lambda}$, which is better than $\lambda$ in the small $\lambda$ regime, but this is false in the general non-convex case.}. Again, this means the sampling is  poor when $\lambda$ is small.

Very slow convergences in~\eqref{eq:CV_OL} or~\eqref{eq:CV_Lang} are particularly observed for multi-modal targets, namely when $U$ has several local minima, which is the typical case in MD. Both dynamics~\eqref{eq:kineticLangevin} and \eqref{eq:overdampedLangvein} are then metastable, as they stay for long times in the vicinity of each local minima and make rare transitions from one potential well to another. Such metastable behaviors are also observed when large flat potential areas are separated by narrow channels. The metastable regions correspond to modes of the target distributions (high probability regions surrounded by a low probability region).
We refer to \cite{lelievre2012two} for a more detailed discussion on metastability due to energetic or entropic barriers. Since the ergodic convergence requires to sample each mode, it only happens at a time-scale where such transitions are observed. In particular, for energetic barriers, the optimal log-Sobolev constant of $\rho^*$ is known to be of order $e^{-\beta c}$, up to polynomial terms in $\beta$, where $c>0$ is called the critical depth of the potential $U$ \cite{MenzSchlichting,Holley}. If $\beta$ is large (i.e. if temperature is small), this is prohibitive for naive MCMC estimators based on time discretizations of either~\eqref{eq:kineticLangevin} or \eqref{eq:overdampedLangvein}.

In such a metastable situation, a standard strategy is to use an importance sampling method, namely to target another probability measure $\bar{\rho}^* \propto e^{-\bar U}$,  easier to sample, and then to compute expectations with respect to $\rho^*$ by considering weighted averages as
\begin{equation}\label{eq:IS}
\int_{\T^d} f(x) \rho^*(x)\dd x = \frac{\int_{\T^d} f(x) \omega(x) \bar{\rho}^*(x)\dd x}{\int_{\T^d}  \omega(x) \bar{\rho}^*(x)\dd x}\,,\qquad \omega(x) =e^{-\beta U(x) + \bar U(x)} \propto  \frac{{\rho}^*(x)}{\bar{\rho}^*(x)}\,.
\end{equation}
There is a balance to find when choosing the probability measure $\bar{\rho}^*$: $\bar{\rho}^*$ should be close to $\rho^*$ in order for the weights not to vary too much (and thus to keep a large effective sample size), but $\bar{\rho}^*$ should not be too close to $\rho^*$ otherwise it is as difficult to sample $\bar{\rho}^*$ as to sample $\rho^*$ (see e.g.~\cite{chopin-lelievre-stoltz-12} for discussions about this in a specific example).

Since the log-Sobolev constant behaves badly at low temperature (i.e. when $\beta$ is large), a natural idea is to take $\bar{\rho}^* \propto e^{-\beta' U}$ where $\beta'<\beta$, i.e. to sample at a larger temperature (this is called tempering and can be declined in many variants, see e.g. \cite[Section 11.1]{henin-lelievre-shirts-valsson-delemotte-22}). However, especially in high dimension, $\bar{\rho}^*\propto e^{-\beta' U}$ may be very different from $\rho^*$, leading to a very high variance of the importance sampling estimator (due to a large amplitude of weights). This is notably the case in MD where, often, the dimension $d$ is very large (it could be of order $10^6$) but in fact most of the mass of $\rho^*$ is concentrated in the vicinity of a submanifold of much lower dimension. This structure is lost when sampling at high temperature. Indeed, at high temperatures, the samples are obtained from high-dimensional dynamics in a flat energy landscape and the resulting estimators are not efficient since only a few samples in the vicinity of the low-dimensional submanifold where $\rho^*$ is concentrated contribute to the trajectory average.


Another natural idea to overcome the difficulties induced by metastability is to define $\bar{\rho}^*$ as $\rho^*$ divided by the marginal of $\rho^*$ along well-chosen low-dimensional degrees of freedom, called collective variables. The biasing is only performed along these collective variables. This leads to the so-called free-energy importance sampling methods, that we will introduce in the next section. These methods can be seen as a variant of the tempering approach presented above: in some sense, the biasing measure $\bar{\rho}^*$ is then such that the temperature is set to $+ \infty$ but only for the collective variables, and the energy landscape is thus flattened only along these degrees of freedom. In the method that we will explore in this work, we will study a variant where the temperature is set to $(\alpha +1) \beta^{-1}$ on the collective variables, for some parameter $\alpha >0$ (see Section~\ref{sec:idealfreeenergy}).

\subsection{Free-energy adaptive biasing methods}\label{sec:PMFmethods}

One idea to build a good biasing function is to rely on the so-called free energy associated with a well-chosen collective variable. Let us make this precise.

A collective variable is a function\footnote{In this introductory section, we present these methods for a general $\xi$. In the remaining of this work, in order to avoid technicalities, we will however concentrate on the simple case $\xi(x_1, \ldots,x_d)=(x_1, \ldots,x_m)$.} 
$$\xi:\T^d \to \T^m$$
with $m$ much smaller than $d$, and such that knowing $\xi(x)$ is essentially sufficient to reconstruct~$x$. Mathematically, this can be formalized and quantified~\cite{lelievre2012two} by the following: the family of measures $\rho^*$ conditioned to a fixed value of $\xi$ satisfies a log-Sobolev inequality~\eqref{eq:LSI} with a large log-Sobolev constant~$\lambda$, uniformly in the value of $\xi$. Roughly speaking, this means that the conditional measures are unimodal and strongly peaked.

Let us now introduce the free energy $A: \T^m \to \R$ associated with the measure $\rho^*$ and the collective variable $\xi$ as follows:
\begin{equation}\label{eq:PMF}
    \exp(-\beta A(z)) dz = \xi \sharp \rho^*
\end{equation}
where $\xi \sharp \rho^*$ is the image (push forward) of the measure $\rho^*$ by $\xi$. We are here implicitly assuming that $\xi$ is such $\xi \sharp \rho^*$ admits a density with respect to the Lebesgue measure on $\T^m$. In the following, we will actually use the name ``free energy" for another object, and thus from now on, in order to avoid confusion, we will refer to $A$ as the potential of mean force (PMF), another name which is commonly used in practice~\cite{FreeEnergy}. The PMF $A$ can be seen as an effective potential in the $\xi$ variable, in the sense that the Boltzmann-Gibbs distribution on $\T^m$ associated with $A$ is exactly the image of $\rho^*$ by $\xi$.

The bottom line of PMF-biasing method is to modify the target measure $\rho^*$ to the biased measure 
$$\bar \rho^*=\frac{e^{-\beta (U-A\circ \xi)}}{\bar{\mathcal Z}^*}$$
where $A\circ \xi$ denotes the composition of $A$ with $\xi$.
The idea is that this biased measure $\bar \rho^*$ should be easier to sample than $\rho^*$, since the push-forward by $\xi$ of $\bar \rho^*$ is by definition the uniform measure on $\T^m$. In particular, for a good choice of the collective variable, the multi-modal features of the original measure $\rho^*$ should not be present anymore in the biased measure $\bar \rho^*$. This can be formalized and quantified in terms of log-Sobolev constants: if the conditional measures of $\rho^*$ given $\xi$ (which are actually the same as the conditional measures of $\bar \rho^*$ given~$\xi$) have large log-Sobolev inequality constants, then the measure $\bar \rho^*$ too, see~\cite{Lelievre_twoscale}, and thus the Langevin dynamics using the biased potential $U-A\circ \xi$ should quickly reach equilibrium (see~\eqref{eq:CV_OL} and~\eqref{eq:CV_Lang}).

\begin{rem}\label{re:T->R}
    When considering sampling problems in  $\R^d$ instead of $\T^d$ and $\xi$ takes values in $\R^m$, it is not possible to target the uniform measure (i.e. a flat histogram) for the collective variables. This issue is addressed by adding a nice (e.g. convex) additional confining potential $W:\R^m\rightarrow \R$ and then targeting the measure $\bar\rho^* \propto e^{-\beta(U+(W-A)\circ\xi )} $, as discussed in \cite[Equations~(10)-(11)]{LRS07} to which we refer for details. This modification can be in particular  applied to the schemes introduced in the next sections. 
\end{rem}

Of course, in practice, the PMF $A$ is  unknown, and actually computing $A$ is the objective of many works in the field of computational statistical physics~\cite{chipot-pohorille-07,FreeEnergy}. One idea is then to learn on the fly an approximation $A_t$ of the PMF $A$, while using the (time-dependent) biased potential $U-A_t \circ \xi$ in the dynamics. Such techniques are called free-energy adaptive biasing methods, and actually encompass many variants which differ in particular in the way the approximation $A_t$ of the PMF is updated~\cite{lelievre-rousset-stoltz-07-b}. Let us mention in particular the metadynamics \cite{Metadynamics} and its well-tempered variant \cite{welltempered}, together with the Adaptive Biasing Force method \cite{Darve-Pohorille,Henin-Chipot,Chipot2011,LRS07,Comer}, self-healing umbrella sampling \cite{marsili2006self,fort2017self} or the Wang-Landau algorithm \cite{wang2001efficient,chevallier2020wang} among others. One can distinguish in particular between adaptive biasing force methods which directly update the mean force $\nabla A_t$ and adaptive biasing potential methods which learn the potential of mean force $A_t$ rather than $\nabla A_t$. More precisely, on the one hand, adaptive biasing force methods rely on the fact that $$\forall z \in \T^m, \, \nabla A(z)=\E_{\rho^*}(f(X)|\xi(X)=z)=\E_{\bar \rho^*}(f(X)|\xi(X)=z)$$ for a so-called local mean force\footnote{In the simple case $\xi(x_1, \ldots, x_d)=(x_1,\ldots,x_m)$, one simply has for all $j \in \{1, \ldots, m\}$, $f_j=\partial_{x_j} U$. For a general function $\xi$, explicit formulas for $f$ in terms of $U$ and $\xi$ can be obtained using the co-area formula, see for example~\cite[Lemma 3.9]{FreeEnergy}.}
 $f:\T^d \to \R^m$  which can be written explicitly in terms of $U$ and $\xi$, so that an approximation of $\nabla A$ can be obtained by sampling the conditional measure of the marginal in time of the stochastic process, given $\xi$. On the other hand, adaptive biasing potential methods rely on the formula~\eqref{eq:PMF} to build an approximation of $A$ by using an estimator of the occupation measure in $\xi$. We refer to~\cite[Chapter 5]{FreeEnergy} for more details. As will become clearer below, the biasing techniques analyzed in the present work  falls into the family of the adaptive biasing potential methods.

In terms of the mathematical analysis of such adaptive biasing techniques, one can distinguish between three kinds of results\footnote{Notice that, although the mathematical analysis and results for algorithms based on mean-field interacting particles or on a single ergodic trajectory are quite different, in practice any algorithm can be declined in these two versions, and both approaches are often considered simultaneously, using ergodic averages of interacting particles.}, all obtained on the overdamped Langevin dynamics. In~\cite{fort-jourdain-kuhn-lelievre-stoltz-15,benaim-brehier-16,fort2017self,fort-jourdain-lelievre-stoltz-18}, results have been obtained on adaptive biasing potential algorithms, with an approximation of $A_t$ made using time averages along the path of a single trajectory, using essentially tools from the convergence analysis of stochastic approximation algorithms. 

In~\cite{LRS07,lelievre-minoukadeh-11,lelievre-maurin-monmarche-22}, convergence results are obtained for adaptive biasing force methods where $\nabla A_t(z)$ is approximated by empirical averages, conditionally on $\xi(x)=z$, over infinitely many interacting replicas.

Finally, in~\cite{benaim-brehier-monmarche-20,ehrlacher-lelievre-monmarche-22}, convergence results are proven for the adaptive biasing force process, where $\nabla A_t(z)$ is approximated by empirical averages, conditionally on $\xi(x)=z$, over the path of a single trajectory of the process.

To the best of our knowledge, there are no convergence results for such adaptive biasing methods applied to the underdamped Langevin dynamics. Besides, there are no convergence results on adaptive biasing potential techniques using 
infinitely many interacting replicas to estimate the current marginal distribution in $\xi$. The objective of this paper is to prove convergence results in those two directions, by proposing a new family of algorithms for which the associated Fokker-Planck dynamics can be interpreted as a gradient descent for some functional (a so-called free energy) defined over the space of probability distributions.

\subsection{Convergence rates for non-linear kinetic Langevin dynamics}

The first convergence rates for the kinetic Langevin diffusion~\eqref{eq:kineticLangevin}, established in~\cite{Talay,mattingly2002ergodicity}, were based on Harris theorem and thus were stated in terms of the total variation norm (or more general but similar $V$-norms). In view notably of the behaviour of this norm in high dimension, it was not clear how to extend these results to systems with mean-field interacting particles and then to their non-linear limit (although recently non-linear versions of Harris theorem have been developed \cite{JournelMonmarcheFV,Elementary}, using a particle-wise total variation norm and coupling arguments).

By contrast, due to the extensivity property of the entropy\footnote{For instance, $\mathcal H(\nu^{\otimes d}|\mu^{\otimes d})= d \mathcal H(\nu|\mu)$, to be compared with $\|\nu^{\otimes d} - \mu^{\otimes d}\|_{TV}/2 \leqslant 1 - (1-\|\nu-\mu\|_{TV}/2)^d$ (sharp in general) and $\| \nu^{\otimes d}/\mu^{\otimes d}-1\|_{L^2(\mu^{\otimes d})}^2 = (\|\nu/\mu-1\|_{L^2(\mu)}^2 +1)^d -1$.},
the  modified entropy method  introduced by Villani in 2009 in \cite{Villani2009}, initially developed for linear dynamics such as the kinetic Fokker-Planck equation associated to the Langevin diffusion~\eqref{eq:kineticLangevin}, was a good candidate for addressing the non-linear case, beyond the perturbation regime with respect to the linear case obtained by linearization \cite{herau2007short,herau2016global} or coupling arguments \cite{bolley2010trend,Schuh,M31}. The first results in that direction were obtained by applying Villani's method to the linear equation associated to a  system of $N$ interacting particles and then letting $N$ go to infinity to get a result for the non-linear equation \cite{GuillinMonmarche,M15}, relying on the fact
\begin{equation}
    \label{eq:DefFreeEnergy}
    \frac1N \mathcal H (\rho^{\otimes N}|\mu_N) \underset{N\rightarrow \infty}{\longrightarrow} \mathcal F(\rho)\,, 
\end{equation}
where $\mu_N$ is the equilibrium of the system of $N$ particles and $\mathcal F$ is the so-called free energy associated to the non-linear equation, see~\cite{GuillinWuZhang} and references therein. This approach requires the log-Sobolev constant of the interacting system to be uniform in $N$, which is in general difficult to establish. More recently, the modified entropy method has been applied directly to the non-linear flow \cite{chen-lin-ren-wang-2024,monmarche2023note,MonmarcheReygner}. Instead of the uniform log-Sobolev inequality, this requires an inequality between the free energy and its dissipation along the flow (called a non-linear log-Sobolev inequality in \cite{MonmarcheReygner}). Let us mention than this non-linear log-Sobolev inequality is actually implied by a uniform-in-$N$ log-Sobolev inequality for $\mu_N$ as shown in \cite{Delgadino} (the converse is conjectured; however, this is still an open question), but it can also be established directly, at least in some cases. For example, when $\mathcal F$ is convex along flat interpolation (i.e. $t\mapsto (1-t)\rho_0 + t\rho_1$) and under some regularity assumptions,
 the authors of \cite{chen-lin-ren-wang-2024} show an exponential convergence by using such an approach. Interestingly, it has recently been proven in \cite{Songbo} that, under the same convexity assumption, the log-Sobolev constant of the $N$ particle system is indeed uniform (see also \cite{2024arXiv240917901M} for a  generalization of~\cite{Songbo}).

In view of these recent progresses, it seems natural to try and apply this approach to PMF adaptive biasing methods. The objective of this work is actually to demonstrate that this is possible for an adaptive biasing potential technique.

\subsection{Contributions of this work}

In this work, we introduce a PMF-based adaptive biasing algorithm as the Wasserstein gradient flow of a suitable explicit free energy. As explained in Section~\ref{sec:PMFmethods}, it is an adaptive biasing potential method in the sense that it is based on the estimation of the PMF, contrary to ABF which estimates the gradient of the PMF (see Section~\ref{sec:ABF} for some perspectives on the application of our approach to ABF). It can be seen as one variant of metadynamics~\cite{Metadynamics,welltempered}, which is an adaptive biasing technique extensively used in MD. The first interest of this approach is that the associated free energy turns out to be a very natural quantity to minimize in view of our PMF-based enhanced sampling objective. Moreover, as we will see, this free energy turns out to be convex (along flat interpolations), which gives an insight from an optimization perspective on the efficiency of the algorithm. The second interest is that, thanks to the explicit free energy, we are able to obtain convergence rates for the kinetic process, which is more often used in practice in MD than the overdamped process. We also study the corresponding overdamped process, in which case we get convergence rates explicit enough to quantify an improvement with respect  to a naive (unbiased) MCMC sampler.

\subsection{Outline}

The remaining of this work is organized as follows. General notations are gathered in Section~\ref{sec:notation}. The settings and main results are stated in Section~\ref{sec:settings}.
In Section~\ref{sec:numerics}, we provide numerical experiments on a toy problem to illustrate the efficiency of the method.
Section~\ref{sec:proofFE} contains the proof of the first main result, Theorem~\ref{thm:cv_infty}, which concerns the stability with respect to a regularization parameter of the minimizers of the free energy. The other main results, namely Theorems~\ref{thm:CVoverdamped}, \ref{thm:CVkinetic} and \ref{th:CVoverdamped-sharp}, addressing the long time convergence of the algorithm towards these minimizers, are proven in Section~\ref{
sec:longtimeproof}. Technical results about Sobolev spaces and Gaussian kernels on the torus are  finally gathered in Appendices, as well as an analysis of the long-time behavior of a non-regularized version of the dynamics introduced in this work, which is however restricted to the overdamped dynamics and relies on a different entropy technique than the convergence results presented in the main text.

\subsection{Notations}\label{sec:notation}

Whenever a probability measure admits a Lebesgue density, we use the same letter for the measure and the density, namely $\mu(\dd x)=\mu(x)\dd x$. 


In the following, we suppose that $m,d\in \N^+$ (the set of positive integers; in this work $\N$ denotes the set of non-negative integers), $m\leq d.$ We decompose $x=(x_1,x_2)\in\T^d = \T^m\times\T^{d-m}$ and we denote by $\mu^1$ the first $m$-dimensional marginal law of a probability distribution $\mu$ on $\T^d$ (i.e.  $X_1\sim\mu^1$ when $(X_1,X_2)\sim \mu$). Given a function $f$ defined in~$\T^m$, we identify it with the function $\tilde f$ defined in $\T^d$ as $\tilde f(x)=\tilde{f}(x_1,x_2)=f(x_1)$. For example, given another function $g$ defined in $\T^d$, the notation $f\star g$ represents $\tilde{f}\star g$, that is, $f\star g(x_1,x_2)=\int_{\T^d}f(x_1-y_1)g(y_1,y_2)\dd y_1\dd y_2$. Give a measurable function $f$ on $\T^d$ or $\T^m$, we denote its essential supremum and essential infimum on $\T^d$ or $\T^m$ as $\inf f$ and $ \sup f$, i.e. $$\sup f=\inf\{M\in \R|M\geq f \text{ a.e.} \} \text{ and }\inf f=\sup\{M\in \R|M\leq f \text{ a.e.} \}.$$    
We will use the notation $\rho \propto g$, where $\rho: E \to \R_+$ is a probability density and $g: E \to \R_+$ is an integrable function, both defined  on a measurable space $E$, to state that $\rho=(\int_E g)^{-1} g$. Depending on the context, this may either state a property of the probability density $\rho$, or define the probability density $\rho$.

Let us now introduce Sobolev spaces on $\T^m$. Let us first clarify the use of notation related to multi-indices. We say that $\theta$ is a multi-index of dimension $m$, if $\theta=(\theta_1,\theta_2,...,\theta_m)\in \mathbb{N}^m$, and we define the length of $\theta$ as $|\theta|=\sum_{i=1}^m \theta_i.$ Given another multi-index $\eta=(\eta_1,\eta_2,...,\eta_m)\in \mathbb{N}^m,$ we define their sum  $\theta+\eta$ as $(\theta_1+\eta_1,\theta_2+\eta_2,...,\theta_m+\eta_m)\in\N^m$. Notice that $|\theta+\eta|=|\theta|+|\eta|.$ Given a function $g\in C^{|\theta|}(\T^m),$ we define the derivative
\begin{equation}\label{eq:deri}
D^\theta g = \frac{\partial^{|\theta|} g}{\partial x_1^{\theta_1} \partial x_2^{\theta_2} \dots \partial x_m^{\theta_m}}.    
\end{equation}
Given a measurable function $f$ defined in $\T^m$, we say that $f$ admits a weak derivative $D^\theta f$, if for any test function $\phi\in C^\infty(\T^m),$
\begin{equation}\label{eq:weak_deri}
\int_{\T^m} \phi D^{\theta} f=(-1)^{|\theta|}\int_{\T^m} f D^{\theta} \phi.    
\end{equation}
Notice that by Green's formula, when $g\in C^{|\theta|}(\T^m)$, the two definitions~\eqref{eq:deri} and~\eqref{eq:weak_deri} coincide. 

For a space $E=\T^m$ or $E=\Omega$ where $\Omega$ is a subdomain of $\R^m$, for $p\in [1,\infty]$ and $ k\in \N$, we denote by $W^{k,p}(E)$ the Sobolev space of measurable functions $f: E \to \R$ such that when $|\theta|\leq k$, $D^\theta f \in L^p(E)$. The norm on $W^{k,p}(E)$ is defined by $\|f\|_{W^{k,p}(E)}=\left(\sum_{|\theta|=0}^k \int_{E}|D^\theta f|^p\right)^\frac{1}{p}$ for $p\in[1,\infty)$ and $\|f\|_{W^{k,\infty}(E)}=\sup_{|\theta|\leq k} \|D^\theta f\|_{L^\infty(E)}$. In particular,  We denote $W^{k,2}(E)$ as $H^k(E)$. For $j \in \N$, we will use the notation $\|D^{j}f\|_{\infty}=\max_{|\theta|=j}\|D^\theta f\|_\infty=\max_{|\theta|=j}\sup |D^\theta  f|$. When $k=0$, $W^{k,p}(E)$ reduces to the $L^p(E)$ space, $p\in[1,\infty]$ and we denote the associated norm $\|f\|_{L^p(E)}$. When the domain $E$ is not made precise, then, implicitly, $E=\T^m$ in which case, for $f \in W^{k,p}(\T^m)$ (resp. $f \in H^k(\T^m)$), we simplify the notation for their norms as
$\|f\|_{W^{k,p}} = \|f\|_{W^{k,p}(\mathbb{T}^m)}$ (resp. $\|f\|_{H^k} = \|f\|_{H^k(\mathbb{T}^m)}$).

Finally, we denote by $\mathcal P(E)$ (resp. $\mathcal P_2(E)$) the set of probability measures on $E$ (resp. with finite second moments).

\section{Mathematical framework and results}\label{sec:settings}

In all this work, we assume that the collective variable is given by:
$$\xi:\left\{
\begin{aligned}
\T^m \times \T^{d-m} &\to \T^m\\
 (x_1,x_2) &\mapsto x_1.
\end{aligned}
\right.$$
In particular, the PMF defined by~\eqref{eq:PMF} is simply:
\begin{equation}\label{eq:FE_x1}
    \forall x_1 \in \T^m, \qquad A(x_1) = - \beta^{-1} \ln \left( \mathcal Z^{-1} \int_{\T^{d-m}} \exp(-\beta U(x_1,x_2)) \, \dd x_2 \right).
\end{equation}
Since $U \in C^\infty(\T^d)$, the PMF $A$ is also $C^\infty(\T^m)$. The objective of this section is to state the main results of this work without any proof, by first presenting the adaptive biasing potential dynamics of interest in an idealized setting in Section~\ref{sec:idealfreeenergy}, and then introducing a regularized version in Section~\ref{sec:regularized} together with the associated convergence results. Section~\ref{sec:ABF}  is finally devoted to a discussion of the difficulties raised by generalizing these convergence results to the Adaptive Biasing Force case.

\subsection{Ideal free energy and associated dynamics}\label{sec:idealfreeenergy}

 Let us introduce the free energy functional: for any probability density $\rho$ on $\T^d$, 
\begin{equation}\label{eq:FE_original}
\mathcal F_{\alpha}\po \rho\pf = \int_{\T^d} U \rho + \frac{1}{\beta}  \left( \int_{\T^d} \rho \ln \rho  + \alpha \int_{\T^m}  \rho^1 \ln \rho^1 \right)\,,
\end{equation}
parametrized by the inverse temperature $\beta$ and an additional parameter $\alpha>0$. By convention $\mathcal F_{\alpha}(\rho)=+\infty$ for probability measures $\rho$ that do not admit a density with respect to the Lebesgue measure on $\T^d$.

The free energy is the sum of the energy term $\int_{\T^d} U \rho$, the entropy of $\rho$ and the entropy of the marginal~$\rho^1$. When $\alpha=0$, up to an additive constant and a multiplicative factor $\beta^{-1}$, the free energy is exactly the relative entropy with respect to $\rho^*$ (defined by~\eqref{eq:def_gibbs}). The additional marginal entropy term is introduced to sample the collective variables at a higher temperature $(\alpha+1)\beta^{-1}$ than the rest of the system, as discussed in Section~\ref{sec:intro-motiv}, see Equation~\eqref{eq:marginal} below. 

Since the energy term and $\rho\mapsto \rho^1$ are linear in $\rho$, while the entropy is strictly convex (along flat interpolations), it is straightforward to see that $\mathcal F_\alpha$ is flat-convex, namely, for any $\rho_0,\rho_1$ with finite free energy,
\[\forall t\in[0,1],\qquad \mathcal F_{\alpha}\po (1-t)\rho_0 + t \rho_1\pf \leq (1-t)\mathcal F_\alpha(\rho_0) + t \mathcal F_\alpha(\rho_1)\,, \]
with a strict inequality if $\rho_0\neq \rho_1$ and $t\in(0,1)$. 

The minimizer of $\F_\alpha$ can be characterized as follows.

\begin{prop}\label{prop:rho_alpha}
For all $\alpha > 0$, $\mathcal F_\alpha$ admits a unique minimizer\footnote{According to the notations introduced in Section~\ref{sec:notation}, we do not indicate explicitly the dependency on $x_1$ and $x_2$, neither the normalizing constant. Equation~\eqref{eq:rho*alpha} means: for all $(x_1,x_2) \in \T^m \times \T^{d-m}$, $\rho^*_\alpha(x_1,x_2) = (\mathcal Z^*_\alpha)^{-1}   \exp \po -\beta\left(U(x_1,x_2) - \frac{\alpha}{\alpha+1} A(x_1)\right)\pf$ where  $\mathcal Z^*_\alpha=\int_{\T^m \times \T^{d-m}} \exp \po -\beta\left(U(x_1,x_2) - \frac{\alpha}{\alpha+1} A(x_1)\right)\pf dx_1 \, dx_2$.} 
\begin{equation}
\label{eq:rho*alpha}
\rho^*_\alpha \propto \exp \po -\beta\left(U - \frac{\alpha}{\alpha+1} A\right)\pf\,.
\end{equation}
Moreover, denoting by $D_\infty$ the optimal log-Sobolev constant of 
$\bar{\rho}^* \propto \exp\bigl(-\beta(U-A)\bigr)$ and setting
\begin{equation}
    \label{eq:Dalpha}
D_\alpha := D_\infty \exp\!\bigl(-\tfrac{2\beta}{\alpha+1}\|A\|_\infty\bigr),
\end{equation}
$\rho^*_\alpha$ satisfies a log-Sobolev inequality with constant $D_\alpha$.

\end{prop}
The proof of Proposition~\ref{prop:rho_alpha} is given in Section~\ref{sec:minimizer}. Recall from Section~\ref{sec:PMFmethods} that, under suitable conditions (particularily, a good choice of collective variables),  \(D_\infty\) is large. The fact that \(D_\alpha\) converges to \(D_\infty\) as \(\alpha\) goes to infinity motivates the introduction of \(\F_\alpha\).
More precisely, notice that the marginal of $\rho^*_\alpha$ in the collective variables satisfies:
\begin{equation}\label{eq:marginal}
\rho^{*,1}_\alpha \propto \exp\po - \frac{\beta}{\alpha+1} A\pf\,. \end{equation}
In other words, compared with $\rho^*$, $\rho^*_\alpha$ is the probability measure which has, on the one hand, the same conditional densities of $x_2$ given $x_1$, but, on the other hand, a marginal distribution in $x_1$ for which the temperature has been multiplied by $1+\alpha$. As a consequence, following~\cite{Lelievre_twoscale}, when the conditional distributions of $x_2$ given $x_1$ are not multimodal while the marginal law along $x_1$ is at a sufficiently large temperature,  the optimal log-Sobolev constant of $\rho_\alpha^*$ is much larger than the one of $\rho^*$, and thus much easier to sample. Indeed, for a non-convex potential, at low temperature, the log-Sobolev constant (namely the parameter $\lambda$ in~\eqref{eq:LSI}) decays exponentially fast to zero with the energy barriers  and the system size (besides, let us mention that for temperatures larger than a critical threshold, it can be shown that the log-Sobolev constant becomes independent from the dimension), see \cite{Delgadino,lelievre2012two,M40} and references within. 

In order to sample $\rho^*_\alpha$, one can use algorithms to minimize $\mathcal F_\alpha$, which can be done following the associated  Wasserstein gradient descent~\cite{ambrosio2005gradient}:
\[\partial_t \rho_t = \na\cdot \po \rho_t \na \frac{\delta\mathcal F_\alpha}{\delta \rho}(\rho_t) \pf \]
where
\[\frac{\delta\mathcal F_\alpha}{\delta \rho}(\rho,x) =  U(x)+ \frac{1}{\beta}\ln \rho (x)  +  \frac{\alpha}{\beta} \ln \rho^1(x_1) \]
is the linear functional derivative of the free energy. In other words,
\begin{equation}
    \label{eq:PDEoverdamped}
    \partial_t \rho_t = \na\cdot \po \rho_t \na \left(U+\frac\alpha\beta\ln \rho_t^1\right)   \pf + \frac{1}{\beta}\Delta \rho_t\,.
\end{equation}
This is a non-linear Fokker-Planck equation associated to the McKean-Vlasov diffusion process solution to
\begin{equation}
    \label{eq:McKean-Voverdamped}
    \left\{\begin{aligned}
     \dd X_t & =  -\na \po U + \frac\alpha\beta \ln \rho_t^1\pf(X_t) \dd t + \sqrt{\frac{2}{\beta}}\dd B_t  \\
     \rho_t &  = \mathcal Law(X_t)\,. 
\end{aligned}\right.
\end{equation}

Similarly, one can introduce a kinetic process which can be seen as a Nesterov second-order version~\cite{nesterov} of the  gradient descent dynamics~\eqref{eq:PDEoverdamped}--\eqref{eq:McKean-Voverdamped}. Denoting by $\nu(x,v)$ the probability density in the phase space, with $\nu^x$ the position marginal (and $\nu^{x,1}$ the first $m$-dimensional marginal of $\nu^x$), this corresponds to the Vlasov-Fokker-Planck-type equation
\begin{equation}
    \label{eq:PDEkinetic}
    \partial_t \nu_t + v\cdot\na_x \nu_t = \na_v\cdot \po \nu_t \po  v+ \na_x\left(U+\frac\alpha\beta\ln\nu_t^{x,1}\right)\pf \pf + \frac{1}{\beta}\Delta_v \nu_t \,.
\end{equation}
This is the Kolmogorov equation of a (self-interacting time-inhomogeneous) kinetic Langevin process solving 
\begin{equation}
    \label{eq:McKeanVkinetic}
    \left\{\begin{aligned}
     \dd X_t & =  V_t\dd t\\
     \dd V_t &=  -\na_x \po U + \frac\alpha\beta \ln \nu^{x,1}_t \pf(X_t) \dd t - V_t\dd t +  \sqrt{\frac{2}{\beta}}\dd B_t  \\
     \nu_t &  = \mathcal Law(X_t,V_t)\,. 
\end{aligned}\right.
\end{equation}

 To simulate the dynamics~\eqref{eq:McKean-Voverdamped} or \eqref{eq:McKeanVkinetic}, one typically introduces an approximation by a system of $N$ interacting particles, for which the marginal probability densities $\rho^1$ or $\nu^{x,1}$ are approximated by the convolution of the empirical distribution of the system with a smooth (typically Gaussian) kernel \cite{jourdain2010existence} (and, in addition, the stochastic differential equations are approximated by an appropriate time-discretization). The resulting practical algorithm based on~\eqref{eq:McKeanVkinetic} can then be seen as a particular version of the well-tempered Metadynamics~\cite{welltempered}, since it is an adaptive biasing potential method, where the potential is eventually modified by using only a fraction of the current estimate of the potential of mean force (see Proposition~\ref{prop:rho_alpha}). However, the algorithm based on~\eqref{eq:McKeanVkinetic} differs from the practical implementation of the well-tempered Metadynamics since we consider a Markovian mean-field system of interacting particles, while in classical presentation and use of such adaptive algorithms, the particles also interact with their past. Indeed, in well-tempered Metadynamics, the biasing potential is updated using the occupation measure, namely the trajectory in the collective variable space up to the current time. In this context, one needs to add a vanishing adaption mechanism in order for the method to converge, as in any stochastic approximation algorithm: as time goes, the last positions should less and less contribute to the biasing potential. We refer for example to~\cite{ehrlacher-lelievre-monmarche-22,benaim-brehier-16,benaim-brehier-monmarche-20} for a mathematical perspective on this non-Markovian setting.

\begin{rem}\label{rem:other-error}
     In this work, we focus only on the long-time convergence of non-linear dynamics, since it is specific to the particular processes considered here. We will not discuss the error induced by the particle approximation or the time-discretization in a practical implementation. However,  when proving our main results (specifically, Theorems~\ref{thm:CVoverdamped} and \ref{thm:CVkinetic} below for the regularized dynamics~\eqref{eq:McKean-Voverdamped_convolution} and \eqref{eq:McKeanVkinetic_convolution}), we check all the conditions of \cite{chen-lin-ren-wang-2024,chen2022uniform} where uniform-in-time propagation of chaos is proven. Under the same conditions,   \cite{Songbo}  directly gives long-time convergence rates for the particle system and \cite{suzuki2024mean,Schuh} additionally address the discretization error (respectively in the overdamped and kinetic case).
 \end{rem}

 The processes~\eqref{eq:McKean-Voverdamped} and \eqref{eq:McKeanVkinetic} are natural dynamics to study the efficiency of PMF-based adaptive algorithms in idealized mathematical settings. However, the pointwise non-linearity   $\na \ln \rho^1$ or $\na \ln \nu^{x,1}$, in the spirit of the Stein variational gradient descent introduced in \cite{liu2016stein}, raises some difficulties in the theoretical analysis, starting with the well-posedness of~\eqref{eq:PDEoverdamped} and~\eqref{eq:PDEkinetic}. It may be possible to work with a suitable notion of solutions, as in the theory of viscosity solutions~\cite{crandall1992user}, or more specifically in our situation by working with the JKO scheme as in \cite{burger2023porous}. However,  beyond the qualitative well-posedness of the equation, our analysis requires uniform-in-time bounds on the solutions, see Remarks~\ref{rem:adapt} and~\ref{rem:sharpkinetic} in particular.  For this reason, as explained in the next section, we will rather work with mollified versions of these processes. Moreover, this is actually more consistent with the algorithms used in  practice than the idealized dynamics~\eqref{eq:McKean-Voverdamped} and \eqref{eq:McKeanVkinetic}, since this mollification corresponds to the Gaussian kernel regularization mentioned above (see also Section~\ref{sec:numerics} for the practical implementation).

\subsection{Regularized free energy and processes}\label{sec:regularized}

As motivated above, we introduce the regularization of the free energy~\eqref{eq:FE_original}: for a probability density $\rho$ on $\T^d$,
\begin{align*}
\mathcal F_{\alpha,\epsilon}\po \rho\pf =& \int_{\T^d} U \rho + \frac{1}{\beta}\left(  \int_{\T^d} \rho \ln \rho  + \alpha \int_{\T^m}  (K^m_\epsilon\star\rho^1) \ln (K^m_\epsilon\star\rho^1)\right)\\=&  \int_{\T^d} U \rho + \frac{1}{\beta}\left(  \int_{\T^d} \rho \ln \rho  + \alpha \int_{\T^d}  (K^m_\epsilon\star\rho) \ln (K^m_\epsilon\star\rho)\right),  
\end{align*}
which is parametrized by $\beta,\alpha>0$ and an additional parameter $\epsilon\in(0,1]$, with $K^m_\epsilon$ the Gaussian kernel on $\T^m$, as introduced in the Appendix~\ref{sec:ap}. To make the notation more convenient in the following,
recall that, from Section~\ref{sec:notation},  $K^m_\epsilon\star\rho$ is defined by regarding $K^m_\epsilon$ as a function on $\T^d$, i.e. for $x=(x_1,x_2)\in \T^d=\T^m\times\T^{d-m}$ and $ y=(y_1,y_2)\in \T^d=\T^m\times\T^{d-m}$, one has
$K^m_\epsilon\star\rho(x)=\int_{\T^d}K^m_\epsilon(x_1-y_1)\rho(y)dy_1dy_2=\int_{\T^m}K^m_\epsilon(x_1-y_1)\rho^1(y_1)dy_1=K^m_\epsilon\star \rho^1(x_1).$
In particular $K^m_\epsilon\star\rho(x)$ is a probability density over $\T^d$. By convention $\mathcal F_{\alpha,\epsilon}(\rho)=+\infty$ for probability measures $\rho$ that do not admit a density with respect to the Lebesgue measure on $\T^d$.




Using similar arguments as for $\F_\alpha$, one can check that $\F_{\alpha,\epsilon}$ is convex and admits a unique minimizer $\rho^*_{\alpha,\epsilon}$. But unlike $\rho^*_\alpha$, $\rho^*_{\alpha,\epsilon}$ has no explicit representation in general. However, we will prove below in Lemma~\ref{lem:limit_entropy} that when $\F_\alpha(\rho)<\infty$, $\lim_{\epsilon\rightarrow 0}\F_{\alpha,\epsilon}(\rho)=\F_{\alpha}(\rho)$. It is thus natural to expect the convergence of the minimizer of $\F_{\alpha,\epsilon}$ to the minimizer of $\F_\alpha$ as $\epsilon\rightarrow 0$, as stated in Theorem~\ref{thm:cv_infty}. In view of the results of Theorem~\ref{thm:cv_infty}, we will sometimes use in the following the index $\epsilon=0$ to refer to the ideal case introduced in Section~\ref{sec:idealfreeenergy}.

\begin{theorem}
\label{thm:cv_infty}
For $\epsilon\in(0,1]$, the functional $\mathcal{F}_{\alpha,\epsilon}$  admits a unique minimizer $\rho^*_{\alpha,\epsilon}$, which has a density with respect to the Lebesgue measure on $\T^d$.

Moreover, there exist constants $C_1,C_2>0$, which are independent of $\alpha,\epsilon$, such that,
when $\epsilon\leq \frac{C_1}{(\alpha+1)}$,
\begin{equation}\label{eq:diffrhostar}
\|\rho_{\alpha,\epsilon}^*-\rho_{\alpha}^*\|_{\infty}\leq C_2{\epsilon},
\end{equation}
where $\rho_{\alpha}^*$ has been defined in Proposition~\ref{prop:rho_alpha}.

Finally, $\rho^*_{\alpha,\epsilon}$ satisfies a log-Sobolev inequality, with a constant $$D_{\alpha,\epsilon}=D_\alpha \exp(-C_3 \epsilon),$$ for some $C_3>0$ independent of $\alpha,\epsilon$, 
where $D_\alpha$ is given in~\eqref{eq:Dalpha}.
\end{theorem}

The first statement of Theorem~\ref{thm:cv_infty} is a consequence of Proposition~\ref{prop:rho*_uniqueness}, and the proofs of the other results are provided in Section~\ref{sec:CV_H1}. In view of the explicit form~\eqref{eq:rho*alpha} of $\rho_\alpha^*$ and the discussion which follows this equation, this result shows that sampling the minimizer of $\mathcal F_{\alpha,\epsilon}$ for a small $\epsilon$ leads to a relevant importance sampling scheme if the conditional distributions of $x_2$ given $x_1$ are not multimodal and $\alpha$ is chosen sufficiently large.

Following the same procedure as in the previous section (see~\eqref{eq:PDEoverdamped} and~\eqref{eq:McKean-Voverdamped}), let us introduce the  non-linear Fokker-Planck equation associated with $\F_{\alpha,\epsilon}$:
\begin{equation}
    \label{eq:PDEoverdamped_convolution}
    \partial_t \rho_t = \na\cdot \po \rho_t \na \left(U+
\frac\alpha\beta K^m_\epsilon\star\ln (K^m_\epsilon\star\rho_t)\right)   \pf + \frac{1}{\beta}\Delta \rho_t\,,
\end{equation}
and the associated  McKean-Vlasov diffusion process solution to:
\begin{equation}
    \label{eq:McKean-Voverdamped_convolution}
    \left\{\begin{aligned}
     \dd X_t & =  -\na \po U + \frac\alpha\beta K^m_\epsilon\star\ln (K^m_\epsilon\star\rho_t)\pf(X_t) \dd t + \sqrt{\frac{2}{\beta}}\dd B_t  \\
     \rho_t &  = \mathcal Law(X_t)\,. 
\end{aligned}\right.
\end{equation}

Likewise (see~\eqref{eq:PDEkinetic} and~\eqref{eq:McKeanVkinetic}), we also consider their kinetic counterparts, namely the Vlasov-Fokker-Planck equation
\begin{equation}
    \label{eq:PDEkinetic_convolution}
    \partial_t \nu_t + v\cdot\na_x \nu_t = \na_v\cdot \po \nu_t \po - v+ \na_x\left(U+\frac\alpha\beta K^m_\epsilon\star\ln(K^m_\epsilon\star\nu_t^{x})\right)\pf \pf + \frac{1}{\beta}\Delta_v \nu_t \,,
\end{equation}
and the associated non-linear Langevin diffusion
\begin{equation}
    \label{eq:McKeanVkinetic_convolution}
    \left\{\begin{aligned}
     \dd X_t & =  V_t\dd t\\
     \dd V_t &=  -\na \po U + \frac\alpha\beta K^m_\epsilon\star\ln (K^m_\epsilon\star\nu_t^{x})\pf(X_t) \dd t - V_t\dd t +  \sqrt{\frac{2}{\beta}}\dd B_t  \\
     \nu_t &  = \mathcal Law(X_t,V_t)\,. 
\end{aligned}\right.
\end{equation}

Thanks to the introduction of the regularized free energy, the drift terms in~\eqref{eq:PDEoverdamped_convolution} and~\eqref{eq:PDEkinetic_convolution} are smooth (compare with the non-regularized versions~\eqref{eq:PDEoverdamped}--\eqref{eq:PDEkinetic}). In particular, the well-posedness of \eqref{eq:PDEoverdamped_convolution} and \eqref{eq:PDEkinetic_convolution} in the set of probability distributions with finite second moments  follows respectively from \cite[Lemma 4.5]{chen2022uniform} and \cite[Proposition 4.6]{chen-lin-ren-wang-2024}, from which moreover for all times $t>0$ the solution admits a density, and has finite entropy and finite Fisher information.

Since we are interested in sampling the Boltzmann-Gibbs measure $\rho^*$ thanks to the adaptive biasing dynamics~\eqref{eq:PDEoverdamped_convolution}-\eqref{eq:PDEkinetic_convolution}, the convergence analysis as $\varepsilon\rightarrow 0$ in Theorem~\ref{thm:cv_infty} concentrates on the minimizer of the free energy $\mathcal F_{\alpha,\epsilon}$, namely the stationary solutions of their associated gradient flows. The analysis of the convergence of the adaptive biasing processes~\eqref{eq:PDEoverdamped_convolution} and \eqref{eq:PDEkinetic_convolution} towards solutions (in a suitable sense) to~\eqref{eq:PDEoverdamped} and~\eqref{eq:PDEkinetic} as $\varepsilon\rightarrow 0$ is beyond the scope of this paper. We refer readers interested in this question to \cite{burger2023porous} and references therein, where such an analysis is carried out for related equations. 

The main objective of this work is to study the long time convergence of the non-linear partial differential equations~\eqref{eq:PDEoverdamped_convolution} and \eqref{eq:PDEkinetic_convolution}.
Concerning~\eqref{eq:PDEoverdamped_convolution}, using the fact that $\rho^*\propto\exp(-\beta U)$ satisfies a log-Sobolev inequality, an exponential convergence for $\mH(\rho_t|\rho^*_{\alpha,\epsilon})$ to zero can be proven.

\begin{theorem}\label{thm:CVoverdamped}
Let $D>0$ be such that $\rho^*\propto\exp(-\beta U)$ satisfies a log-Sobolev inequality with constant $D$. There exists a constant $C>0$ independent of $\alpha,\epsilon$, such that for any $\epsilon \in (0,1]$, and for any initial condition $\rho_0$ in $\T^d$ with finite entropy, the solution $(\rho_t)_{t \ge 0}$ to~\eqref{eq:PDEoverdamped_convolution} satisfies:
$$\forall t\ge 0, \, \mathcal{H}(\rho_t \mid \rho^*_{\alpha,\epsilon}) \leq \beta\F_{\alpha,\epsilon}\left(\rho_t\right)-\beta\F_{\alpha,\epsilon}\left(\rho^{*}_{\alpha,\epsilon}\right)\leq \beta \left(\F_{\alpha,\epsilon}(\rho_0)-\F_{\alpha,\epsilon}\left(\rho^{*}_{\alpha,\epsilon}\right)\right)e^{-c_{\alpha,\epsilon} t},$$
where $c_{\alpha,\epsilon}=\frac{2D}{\beta}\exp(-C\alpha-\frac{m\alpha}{8\epsilon})$.    
\end{theorem}
This result is proven in Section~\ref{sec:longtime_overdamped_dependent}

We prove a similar result for the kinetic case~\eqref{eq:PDEkinetic_convolution}, for which the stationary solution is 
\[\nu^*_{\alpha,\epsilon} = \rho^*_{\alpha,\epsilon} \otimes \mathcal N(0, \beta^{-1} I_d).\]
Let us introduce the extended free energy defined for a probability density $\nu$ on $\T^d\times\R^d$ by
\[\mathcal{\tilde{F}}_{\alpha,\epsilon}\left(\nu\right) = \int_{\T^d\times \R^d} H \nu + \frac{1}{\beta}\left(  \int_{\T^d\times \R^d} \nu \ln \nu  + \alpha \int_{\T^d}  K^m_\epsilon\star\nu^x \ln (K^m_\epsilon\star\nu^x)\right),  \]
where $H(x,v)=U(x)+ \frac12|v|^2$, and as in the overdamped case, $K^m_\epsilon\star\nu^x(x)=K^m_\epsilon\star \nu^{x,1}(x^1)$. Again, by convention, $\mathcal{\tilde{F}}_{\alpha,\epsilon}\left(\nu\right)=+\infty$ if $\nu$ is a probability distribution which does not admit a density with respect to the Lebesgue measure on $\T^d \times \R^d$.

\begin{theorem}\label{thm:CVkinetic}
For any $\epsilon\in(0,1]$, there exist constants $C_{\alpha,\epsilon},\kappa_{\alpha,\epsilon}>0$ such that for any initial distribution $\nu_0$ on $\T^d\times\R^d$ with finite second moment and finite entropy, the solution $(\nu_t)_{t \ge 0}$ to~\eqref{eq:PDEkinetic_convolution} satisfies:
$$\forall t \ge 0, \, \mathcal{H}(\nu_t \mid \nu^*_{\alpha,\epsilon}) \leq \beta \mathcal{\tilde{F}}_{\alpha,\epsilon}\left(\nu_{t}\right)-\beta\mathcal{\tilde{F}}_{\alpha,\epsilon}\left(\nu^*_{\alpha,\epsilon}\right) \leq C_{\alpha,\epsilon} \left(\mathcal{\tilde{F}}_{\alpha,\epsilon}\left(\nu_0\right)-\mathcal{\tilde{F}}_{\alpha,\epsilon}\left(\nu^*_{\alpha,\epsilon}\right)  \right) e^{-\kappa_{\alpha,\epsilon} t}.  $$ 
\end{theorem}


This result is proven in Section~\ref{sec:proofCVkinetic}. To the best of our knowledge, this is the first convergence result for a PMF-based enhanced sampling algorithm applied to the kinetic Langevin diffusion. As mentioned in the introduction, the Langevin diffusion is the standard method used in MD. 

The convergence rate $\kappa_{\alpha,\epsilon}$ and the prefactor $C_{\alpha,\epsilon}$ in Theorem~\ref{thm:CVkinetic} could be made explicit by following the proof but they are not very informative. In fact, in both Theorems~\ref{thm:CVoverdamped} and~\ref{thm:CVkinetic}, the estimates of the convergence rates $c_{\alpha,\epsilon}$ and $\kappa_{\alpha,\epsilon}$ obtained from the proofs go exponentially fast to zero as $\epsilon \rightarrow 0$ and $\alpha\rightarrow+\infty$. As a consequence, these results are not sufficiently sharp to compare the efficiencies of these adaptive biasing schemes with the standard naive (overdamped) Langevin dynamics. 

This issue stems from a loose bound on the log-Sobolev constant of certain 
time-dependent probability distributions appearing in the proof. In the proof of Theorems~\ref{thm:CVoverdamped} and~\ref{thm:CVkinetic}, this bound is obtained from the non-explicit fixed-point characterization of 
$\rho^*_{\alpha,\epsilon}$, which will be shown in Proposition~\ref{prop:rho*}, together 
with uniform bounds on the 
Gaussian kernel $K^m_\epsilon$ and a crude dependence on $\alpha$. These bounds deteriorate 
as $\epsilon \to 0$ and $\alpha \to +\infty$.
Let us explain how we are able to circumvent this difficulty, at least for the overdamped  dynamics~\eqref{eq:PDEoverdamped_convolution}.

Thanks to Theorem~\ref{thm:CVoverdamped}, $ \rho_t $ converges to $\rho_{\alpha,\epsilon}^*$ as $t\rightarrow \infty$, and we thus may expect the convergence rate to be given by the log-Sobolev constant $D_{\alpha,\epsilon}$ of   $\rho_{\alpha,\epsilon}^*$. This is stated in the next theorem.

\begin{theorem}\label{th:CVoverdamped-sharp} Let $\rho^*_{\alpha,\epsilon},D_{\alpha,\epsilon}$ be as in Theorem~\ref{thm:cv_infty}.
There exists a constant $c>0$ independent of $\epsilon,\alpha$, such that when $\epsilon\in(0,\frac{c}{\alpha+1}]$, for any initial condition $\rho_0$ in $\T^d$ with finite entropy, there exists a constant $M>0$ which depends on $\epsilon, \alpha,\beta$ and $\mathcal F_{\alpha,\epsilon}(\rho_0)$, such that
  the solution $(\rho_t)_{t \ge 0}$ to~\eqref{eq:PDEoverdamped_convolution}, satisfies:
\begin{equation}\label{eq:CV_ind_eps} \forall t \ge0, \, \mathcal{H}(\rho_t \mid {\rho}^{*}_{\alpha,\epsilon})  \leq  \beta\F_{\alpha,\epsilon}(\rho_{t})-\beta\F_{\alpha,\epsilon}\left(\rho^{*}_{\alpha,\epsilon}\right)\leq  M e^{-\frac{2 D_{\alpha,\epsilon}}{\beta} t}. \end{equation}
\end{theorem}
The proof of Theorem~\ref{th:CVoverdamped-sharp} is given in Section~\ref{sec:longtime_overdamped}. In contrast with the proof of Theorem~\ref{thm:CVoverdamped}, it relies on a technical fixed-point argument (see Theorem~\ref{thm:implicit} below). According to Theorem~\ref{thm:cv_infty}, for any  $\epsilon_0 \in(0,\frac{c}{\alpha+1}]$, for any $\epsilon \in (0,\epsilon_0)$, 
$D_{\alpha,\epsilon} \ge D_\alpha \exp(-C_3 \epsilon_0)$: for sufficiently small $\epsilon$, the convergence rate in~\eqref{eq:CV_ind_eps} can thus be made independent of $\epsilon$.
We will thus refer to Theorem~\ref{th:CVoverdamped-sharp} as an exponential convergence result with an $\epsilon$-independent rate in the following, as opposed to the results of Theorems~\ref{thm:CVoverdamped} and~\ref{thm:CVkinetic} which provide exponential convergence results with $\epsilon$-dependent rates (in the sense that those rates are {\em a priori} not bounded away from $0$ when $\epsilon \to 0$).

Moreover, according again to Theorem~\ref{thm:cv_infty}, $D_{\alpha,\epsilon}$ can be made arbitrarily close to the optimal log-Sobolev constant of $\rho_{\alpha}^*$, for $\epsilon$ sufficiently small, which, thanks to \cite{Lelievre_twoscale}, is bigger (at least at low temperature and when the collective variable is well-chosen) than the one of the Boltzmann-Gibbs measure $\rho^*\propto \exp(-\beta U)$, see Section~\ref{sec:idealfreeenergy}. 
Theorem~\ref{th:CVoverdamped-sharp} thus shows that, in this setting, the dynamics~\eqref{eq:PDEoverdamped_convolution} converges to equilibrium at a much larger rate than the standard overdamped diffusion~\eqref{eq:overdampedLangvein}.

The same result could be expected to hold in the kinetic case~\eqref{eq:PDEkinetic_convolution}, i.e. it may be possible to improve~Theorem~\ref{thm:CVkinetic} to get a convergence rate independent of $\epsilon$ in the regime $\epsilon \to 0$. However, for the kinetic Langevin diffusion, even in the linear setting, the convergence rate obtained with the modified entropy approach, that we follow here, depends not only on the log-Sobolev constant of the stationary state but also on bounds on the Hessian of some log-densities.
This leads to additional difficulties when trying to establish the analogue of Theorem~\ref{th:CVoverdamped-sharp} in the kinetic case, see Remark~\ref{rem:sharpkinetic} at the end of Section~\ref{sec:longtime_overdamped} for details.



\begin{rem}
    In practice, when the dynamics~\eqref{eq:McKean-Voverdamped_convolution} (resp.~\eqref{eq:McKeanVkinetic_convolution}) reach their equilibria, they do not sample the target measure $\rho^*$ (resp. $\nu^*$) but biased measures. However, the ratio of the equilibrium measure over the target measure is proportional to $\omega(x)=\exp(\alpha K^m_\epsilon \star \ln (K^m_\epsilon \star \rho_t))$ (resp. $\omega(x)=\exp(\alpha K^m_\epsilon \star \ln (K^m_\epsilon \star \nu^x_t))$), namely a quantity  which depends on the biasing potential which is learnt on the fly: one can thus use weighted averages, as in~\eqref{eq:IS} to get unbiased  estimators. 
\end{rem}

\subsection{Discussion about the convergence of the Adaptive Biasing Force method}\label{sec:ABF}

Before presenting the proofs of the convergence results in the following sections, let us make a few comments on applying the mean-field free energy approach to the Adaptive Biasing Force (ABF) method (rather than to an adaptive biasing potential, as the one described above). A key point in the recent works  \cite{chen-lin-ren-wang-2024,monmarche2023note,MonmarcheReygner} on the mean-field underdamped Langevin diffusion is that the free energy, which can always be defined by~\eqref{eq:DefFreeEnergy} as discussed in \cite{monmarche2023note}, is explicit. This is not the case for ABF. Even for the overdamped dynamics, the free energy (in the sense of~\eqref{eq:DefFreeEnergy}) decays along the non-linear flow (cf.~\cite[Proposition 2]{monmarche2023note}), but neither this free energy nor its dissipation along the non-linear flow are explicit. It is actually unclear whether  the dynamics can be seen as a Wasserstein gradient flow  in the sense of~\cite{ambrosio2005gradient}.
That being said, since ABF relies on the estimation of the gradient of   a log density (which is related to what is referred to as score matching in machine learning \cite{hyvarinen2005estimation}), it may be possible to interpret it as a gradient flow of a Fisher Information, but then not with respect to the  Wasserstein distance (which would yield a fourth-order equation \cite{gianazza2009wasserstein}) but to another metric, such as the Fisher-Rao one \cite{carrillo2024fisher}. To our knowledge, modified entropy methods have not yet been developed for kinetic versions of Fisher-Rao gradient flows. This could be an interesting research direction for the future. 

Let us finally comment on the convergence of ABF for the overdamped dynamics. In~\cite{LRS07}, exponential rates of convergence in entropy are established for the overdamped ABF dynamics. The computations are based on the usual relative entropy with respect to the stationary solution (and not on the non-linear free energy given by~\eqref{eq:DefFreeEnergy}) and they rely on a very specific feature of the dynamics, namely that the marginal distribution in the collective variable solves an autonomous linear Fokker-Planck equation. This gives a convergence of this marginal distribution to its equilibrium, which can be used to control some error terms in the derivation of the full relative entropy, leading to an overall exponential decay (we exemplify in Appendix~\ref{sec:McKean-Voverdamped} how such techniques can be used on the example of the non-regularized dynamics~\eqref{eq:McKean-Voverdamped}).
Extending this analysis to the kinetic case is unclear. Indeed, with the standard definition of $\nabla A_t$ used in practice, it is not true that the marginal density of the collective variables (and possibly their velocities) solves an autonomous equation. We may modify the algorithm to enforce this, but then, with a given fixed biasing force, the process is a non-equilibrium kinetic Langevin process for which the equilibrium is not explicit~\cite{M35}. It is then unclear how to apply a modified entropy version of \cite{LRS07}, since this would require to have a bound on the Hessian of the logarithm of these non-explicit local equilibria.

The convergence of the ABF method for the (kinetic) Langevin dynamics thus remains an open problem.

\section{Numerical Experiments}\label{sec:numerics}
The objective of this section is to illustrate the efficiency of the Adaptive Biasing Potential method~\eqref{eq:McKean-Voverdamped_convolution} on a toy problem. We refer to the review papers~\cite{valsson2016enhancing,henin-lelievre-shirts-valsson-delemotte-22,Comer} for references to realistic molecular systems where similar adaptive biasing techniques are used.

\subsection{Numerical scheme}
In practice, the McKean-Vlasov process~\eqref{eq:McKean-Voverdamped_convolution} is first approximated by a system of $N$ interacting particles $(X_{t}^1,\dots,X_{t}^N)\in\T^{dN}$ evolving according to
\begin{equation}
     \dd X_{t}^i  =  -\na \po U + \frac\alpha{\beta } K_\epsilon^m \star \ln \po \frac1N \sum_{j=1}^N K_\epsilon^m(\cdot - X_{t}^j)  \pf  \pf(X_{t}^i) \dd t + \sqrt{\frac{2}{\beta}}\dd B_{t}^i
\end{equation}
with independent Brownian motions $B^{1},\dots,B^N$. Time is then discretized  using e.g. an Euler-Maruyama scheme with some step size $h>0$, and space is discretized on a grid $\mathcal G^m$ of $\T^m$ with step $\delta_x$ for the convolution terms. Notice that the number of points in the grid $\mathcal G^m$ remains manageable in pratice since $m$ is  small (typically $m\in\{1,2,3\}$). The resulting practical algorithm is thus the Markov chain $(Y_k^1,\dots,Y_k^N) \in \T^{dN}$ with transitions
\begin{equation}\label{Euler}
     Y_{k+1}^i  = Y_{k}^i  - h  \na \po  U + W_k\pf  (Y_{k}^i)   + \sqrt{\frac{2h}{\beta}} G_k^i
\end{equation}
with $G^i_{k} \sim \mathcal N(0,I_d)$ i.i.d. The biasing potential $W_k$ and its gradient, which depend only on the first $m$ coordinates, are evaluated at $x\in\T^d$ as follows:

\begin{align}
W_k(x)
& = \frac{\delta_x^m \alpha}{\beta}\sum_{y\in\mathcal G^m}
K_\epsilon^m(x_1-y)
\ln\!\Bigl(\frac{1}{N}\sum_{j=1}^N K_\epsilon^m(y-Y_{k,1}^j)\Bigr),
\label{loc:sim1}
\\
\nabla W_k(x)
& =\Biggl(
\frac{\delta_x^m \alpha}{\beta}\sum_{y\in\mathcal G^m}
\nabla K_\epsilon^m(x_1-y)
\ln\!\Bigl(\frac1N\sum_{j=1}^N K_\epsilon^m(y-Y_{k,1}^j)\Bigr),
\;0_{d-m}
\Biggr),\label{loc:sim2}
\end{align}
where $Y_{k,1}^i$ denotes the first $m$ coordinates of $Y_k^i$ and $0_{d-m}$ denotes the zero vector in $\T^{d-m}$.



These approximations for $W_k$ and $\nabla W_k$ crucially exploits the smoothness of the Gaussian kernel and the regularizing/convolution structure, which in turn provides additional motivation for introducing convolution-regularized dynamics.

The expectation of an observable $g$ with respect to the Gibbs measure $\rho_* \propto e^{-\beta U}$ is then approximated following the reweighting procedure~\eqref{eq:IS}. The average is done both over the particles $i\in\cco 1,N\ccf$  and time $k\geqslant 1$. The weight for the samples generated at time $k$ is $\omega_k(x) = e^{\beta W_k(x)}$. After $M$ time-steps,  the estimator is thus
\begin{equation}\label{eq:reweighted_estimator}
    \po \sum_{k=1}^M \sum_{i=1}^N  \omega_{k}(Y_k^i)\pf^{-1}\sum_{k=1}^M \sum_{i=1}^N g\po Y_{k}^i\pf \omega_{k}(Y_k^i)\,.
\end{equation}

As in practical use of MCMC methods, a burn-in time is introduced to reduce the influence of the initial transient behavior: specifically, we sometimes average only over the second half of the trajectory.
More precisely, we introduce a time window $\mathcal I\subset\{1,\dots,M\}$, and in practice take
$\mathcal I=\{\lceil M/2\rceil,\dots,M\}$.
The corresponding windowed reweighted estimator is
\begin{equation}\label{eq:reweighted_estimator_window}
    \Bigg(\sum_{k\in\mathcal I}\sum_{i=1}^N \omega_k(Y_k^i)\Bigg)^{-1}
    \sum_{k\in\mathcal I}\sum_{i=1}^N g\bigl(Y_k^i\bigr)\,\omega_k(Y_k^i).
\end{equation}

\subsection{Numerical results}
\textbf{General setting for the numerical experiments.} In the experiment below, the parameters are chosen as follows:
$m=1$, $d=2$, $\beta=1.5$, $h= 10^{-4}$, $\epsilon=10^{-3},\delta_x=10^{-3}$, $N=500$, $M=3000$. Thus, the time horizon is $T=M\cdot h=0.3$.
The state space is the two-dimensional torus $\mathbb T^2$, which we represent as $[-\tfrac12,\tfrac12)^2$
with periodic wrapping after each Euler step.
We consider the potential $U:\mathbb T^2 \to \R$ defined by:
\[
U(x_1,x_2)
= \bigl(\sin(4\pi x_1)-1\bigr)^2
+ \theta\,\sin(2\pi x_1)
+ 2\,\sin(2\pi x_2)\sin(2\pi x_1),
\qquad \theta=-0.5.
\]

The heatmap of $U$ is shown on Figure~\ref{fig:heatmap_minusU_torus}.We observe that the potential $U$ exhibits two wells, located near $(x_1,x_2)\approx(-0.32,\,0.25)$ and $(x_1,x_2)\approx(0.19,\,-0.25)$, which is also confirmed by the numerical search for local minima reported by the code.
This is consistent with a multimodal situation, typically associated with a small log-Sobolev constant for the associated Boltzmann-Gibbs probability measure $\rho^*$.

\begin{figure}[H]
  \centering
  \includegraphics[width=0.55\textwidth]{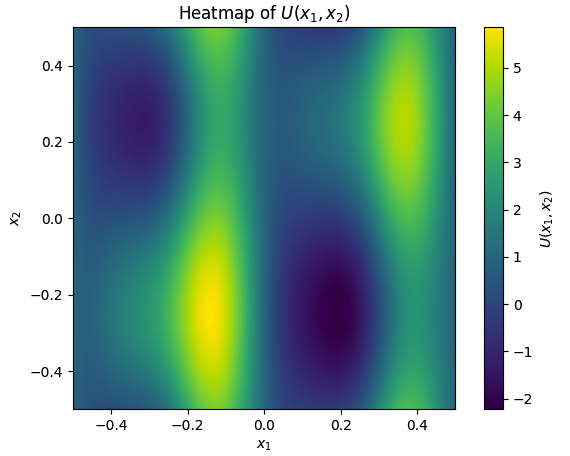}
  \caption{Heatmap of $-U$ on  $\mathbb T^2$, represented as $[-\tfrac12,\tfrac12)^2$.}
\label{fig:heatmap_minusU_torus}
\end{figure}

We compare two dynamics (see~\eqref{Euler}--\eqref{loc:sim2}): the original (unbiased) dynamics ($\alpha=0$) and the dynamics with an adaptive biasing potential ($\alpha>0$, here $\alpha=10$), keeping the same numerical scheme and parameters. We focus on three aspects: (i) the typical exit times from a metastable state, (ii) the sampling of the $x_1$-marginal of the Boltzmann-Gibbs measure $\rho^*$ (via reweighting when $\alpha>0$), and (iii) the convergence of empirical (weighted) averages for two observables, $g_1(x)=\mathbf{1}_{\{x_1\in[0,\frac{1}{2})\}}$ and $g_2(x)=U(x)$.
We simulate a system of $N$ interacting particles over a finite time horizon $[0,T]$, starting from an initial condition concentrated in the left well $(-0.32,0.25)$. 

Notice that we compare the two algorithms using the same number of iterations $M$, instead of the same computational time. On our toy example, the interacting particle system is obviously more expensive than the basic algorithm. However, in practical situations of interest for such free-energy adaptive biasing algorithm, the dimension $d$ is very large (typically $d\gg 10^4$), so that the cost for computing the biasing potential is negligible compared to the evaluation of $\nabla U$.

\medskip

\textbf{Quantities reported from simulations.} Let us now give some details on the numerical experiments.
Let $(Y_k^1,\dots,Y_k^N)$ denote the particle system (for $\alpha=0$ or $\alpha>0$) simulated by~\eqref{Euler} and~\eqref{loc:sim2}, with initial positions $(Y_1^1,\dots,Y_1^N)$ drawn independently as
\[
Y_1^i \sim \mathcal N\!\left(
\begin{pmatrix}
-0.32\\[2pt]
0.25
\end{pmatrix},
\begin{pmatrix}
0.02^2 & 0\\
0 & 0.02^2
\end{pmatrix}
\right),
\qquad i=1,\dots,N,
\]
and then wrapped onto $\T^2$ (identified with $[-\frac{1}{2},\frac{1}{2})^2$). In other words, the particles are initialized close to the local minimum of $U$ that is near $(-0.32,0.25)$. 

For the biased run, particles at iteration $k$ are reweighted with  weights
\[
\omega_k^i:=\omega_k(Y_k^i),
\qquad \omega_k(x)=e^{\beta W_k(x)}, \]
where $W_k$ is defined by~\eqref{loc:sim1} 
(whereas for the unbiased run $\omega_k^i\equiv 1$).

\emph{(i) Monitoring of the transitions.}
We plot the proportion of particles that have visited the right well at least once, for each $k=1,...,M$:
\[
p_k
=
\frac1N\sum_{i=1}^N \mathbf{1}_{\{\max_{0\le \ell\le k} Y_{\ell,1}^{i}\ge 0\}}.
\]

\emph{(ii) Unbiased $x_1$-marginal via reweighting.}
We estimate the unbiased $x_1$-marginal by a reweighted empirical measure over a time window
$\mathcal I\subset\{1,\dots,M\}$ (in practice, we take the last half of the run
$\mathcal I=\{\lceil M/2\rceil,\dots,M\}$):
\[
\widehat\rho_{x_1}(\dd x_1)
=
\frac{\sum_{k\in\mathcal I}\sum_{i=1}^N \omega_k^i\,\delta_{Y_{k,1}^{i}}(\dd x_1)}
{\sum_{k\in\mathcal I}\sum_{i=1}^N \omega_k^i}.
\]
This empirical measure is represented using histogram.

\emph{(iii) Convergence of (weighted) empirical averages for two observables.}
For $g_1(x)=\mathbf{1}_{\{x_1\in[0,\frac{1}{2})\}}$ and $g_2(x)=U(x)$, we monitor a windowed (recent-half) estimator:
for each $n=1,\dots,M$, let $\mathcal I_n=\{\lceil n/2\rceil,\dots,n\}$ and evaluate
\[
\widehat{\E}_{n}[g_j]
=
\frac{\sum_{k\in\mathcal I_n}\sum_{i=1}^N g_j(Y_k^i)\,\omega_k^i}
{\sum_{k\in\mathcal I_n}\sum_{i=1}^N \omega_k^i},
\qquad j\in\{1,2\}.
\]

In practice, when computing the convolution of $K_\epsilon^m$
on the torus, as in~\eqref{loc:sim1} and~\eqref{loc:sim2}, we approximate the wrapped Gaussian kernel
$K^m_\epsilon$ defined by~\eqref{eq:sum_gaussian_kernel} and~\eqref{eq:product_gaussian_kernel} by keeping only the leading term $n=0$.
Since $\epsilon$ is small in our simulations, the contributions of the terms $n\neq 0$ are exponentially small and can be safely neglected.
Moreover, when keeping only the $n=0$ term, the gradient of $K_\epsilon^m$ can be written as a product of $K_\epsilon^m$ and its argument (up to the factor $1/\epsilon$), so that at each time step we compute $K_\epsilon^m\!\bigl(y_i-Y_{k,1}^{j}\bigr)$ only once for all pairs $(i,j)$ (particles $j=1,\dots,N$ and grid points $y_i$), i.e.\ we form the kernel matrix with entries
\[
K_{j,i}=K_\epsilon^m\!\bigl(y_i-Y_{k,1}^j\bigr),
\]
and reuse it to evaluate both $W_k$ in~\eqref{loc:sim1} and $\nabla W_k$ in~\eqref{loc:sim2} at the particle positions $Y_{k,1}^j$.

Figure~\ref{fig:fourpanels} summarizes the results obtained for these three convergence indicators:
Transition times (panel~(a));
Convergence of the $x_1$-marginal (panel~(b)); Convergence of reweighted estimates for $g_1$ and $g_2$ (panels~(c)--(d)).

The codes are available for reproducibility (Google Colab notebook): 
\url{https://colab.research.google.com/drive/1-SbNpi1_DmQ-ffD_y230_R6OtJaWt_WC?usp=sharing}

\paragraph{Discussion of the results.}
Panel~(a) shows that the biased dynamics ($\alpha=10$) explores the two wells in $x_1$ more efficiently than the unbiased one ($\alpha=0$),
in the sense that a significantly larger fraction of particles visits the region $\{x_1\in[0,\frac{1}{2})\}$ over a fixed time horizon.
Panel~(b) illustrates that the reweighted empirical $x_1$-marginal (from the $\alpha=10$ run)
matches the theoretical marginal of the unbiased Gibbs measure more accurately than the direct empirical marginal obtained from the $\alpha=0$ run at the same simulation time.
Finally, panels~(c)--(d) confirm that (weighted) empirical averages converge to the theoretical averages for the observables
$g_1(x)=\mathbf 1_{\{x_1\in [0,\frac{1}{2})\}}$ and $g_2(x)=U(x)$:
the weighted windowed estimators based on the $\alpha=10$ trajectory converge faster to the corresponding theoretical values than their unweighted counterparts for the $\alpha=0$ simulation.

\begin{figure}[H]
\centering
\captionsetup{font=small}
\captionsetup[sub]{font=small}

\begin{subfigure}[t]{0.48\linewidth}
  \centering
  \includegraphics[width=\linewidth,height=0.22\textheight,keepaspectratio]{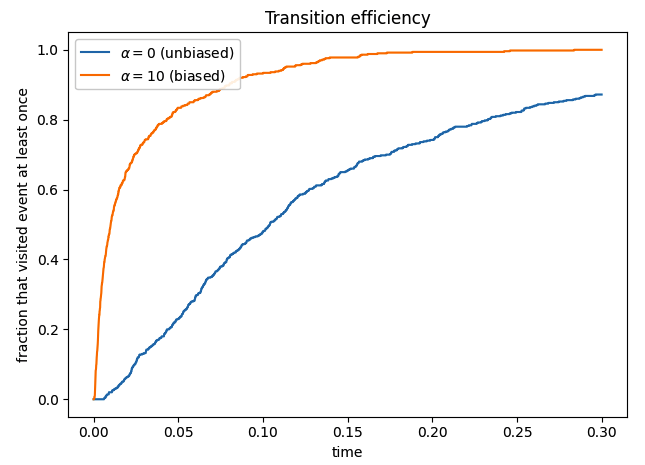}
  \caption{Exit from the initial metastable state}
\end{subfigure}\hfill
\begin{subfigure}[t]{0.48\linewidth}
  \centering
  \includegraphics[width=\linewidth,height=0.22\textheight,keepaspectratio]{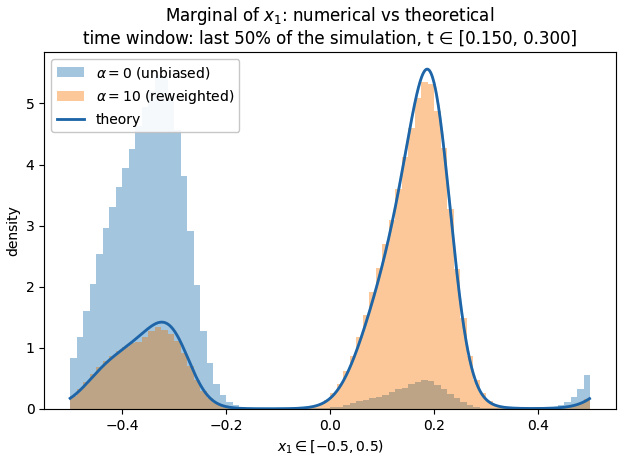}
  \caption{Convergence of the $x_1$-marginal}
\end{subfigure}

\begin{subfigure}[t]{0.48\linewidth}
  \centering
  \includegraphics[width=\linewidth,height=0.22\textheight,keepaspectratio]{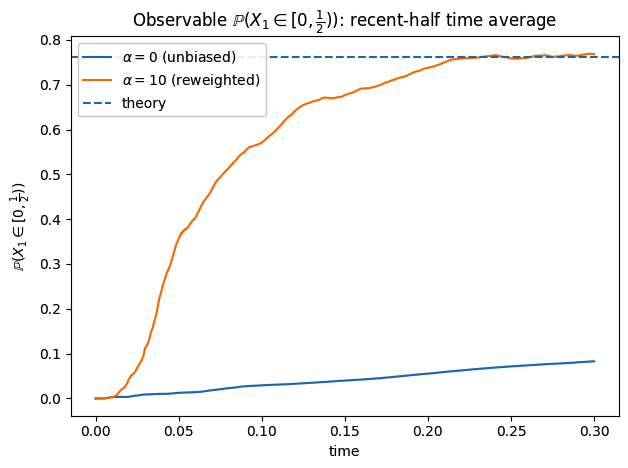}
  \caption{Approximation of ${\mathbb{P}}\!\big(X_1\in[0,\frac{1}{2}))$}
\end{subfigure}\hfill
\begin{subfigure}[t]{0.48\linewidth}
  \centering
  \includegraphics[width=\linewidth,height=0.22\textheight,keepaspectratio]{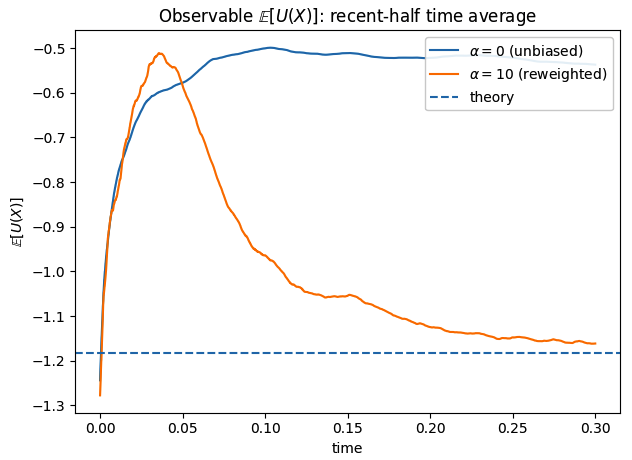}
  \caption{Approximation of ${\mathbb{E}}\!\left[U(X)\right]$}
\end{subfigure}

\caption{Numerical experiments on $\mathbb T^2$, for the unbiased and biased dynamics: (a) proportion of particles which leave the original metastable state as a function of time, (b) convergence of the empirical estimate of the $x_1$-marginal, (c) estimate of ${\mathbb{P}}\!\big(X_1\in[0,\frac{1}{2})\big)$ by a (weighted) empirical average, as a function of time, (d) estimate of ${\mathbb{E}}[U(X)]$ by a (weighted) empirical average, as a function of time.}
\label{fig:fourpanels_torus}

\end{figure}\label{fig:fourpanels}

\section{Proofs of the results on the free energies}\label{sec:proofFE}
In this section, we will prove some properties on the minimizers $\rho^*_{\alpha,\epsilon}$ (resp. $\rho^*_\alpha$) of $\F_{\alpha,\epsilon}$ (resp.~$\F_\alpha$), and the fact that when $\epsilon\rightarrow 0^+, \rho^*_{\alpha,\epsilon}\rightarrow \rho^*_{\alpha}$ in the sense of $L^\infty$ norm. To alleviate  notation, some proofs will be restricted to the case $\beta=1$,  which simply amounts to replace $\F_{\alpha,\epsilon}$ by $\beta \F_{\alpha,\epsilon}$, $U$ by $\beta U$, and $A$ by $\beta A$.

\subsection{Characterization of the free energy minimizers}\label{sec:minimizer}
In this subsection, we prove the existence and various properties of the minimizers of the free energies $\F_{\alpha,\epsilon}$ and $\F_\alpha$, following ideas of \cite{Hu2021MeanField}. In particular, we prove the results stated in Proposition~\ref{prop:rho_alpha}.

\begin{prop}\label{prop:rho*_uniqueness}
For $\epsilon\in(0,1]$ (resp. $\epsilon=0$), the functional $\mathcal{F}_{\alpha,\epsilon}$ (resp. $\mathcal{F}_{\alpha}$) admits a unique minimizer $\rho^*_{\alpha,\epsilon}$ (resp. $\rho^*_{\alpha}$), which has a density on $\T^d$.
\end{prop}

\begin{proof}We only prove this result for the case $\rho^*_{\alpha,\epsilon}$, the case for $\rho^*_\alpha$ follows by the same arguments.
We use standard properties of relative entropy, see~\cite[Lemma 1.4.3]{dupuis1997weak}. In the following, limit points and semicontinuity are meant in the sense of the weak topology on probability measures.

We recall that the Boltzmann-Gibbs measure $\rho^*$ is given by~\eqref{eq:def_gibbs}, where $\mathcal Z$ is the normalizing constant. Denoting by $1_{\T^m}$  the constant probability density over $\T^m$, then  for a probability density $\rho$ on $\T^d$, $\mathcal{F}_{\alpha,\epsilon}(\rho)=\frac{1}{\beta}(\mH(\rho|\rho^*)-\ln \mathcal Z)+ \alpha\mH(K^m_\epsilon\star\rho^1|1_{\T^m}))\geq-\frac{\ln \mathcal{Z}}{\beta}$. Hence, there exists a sequence of probability densities $(\rho_n)_{n\geq 0}$ such that $\mathcal{F}_{\alpha,\epsilon}(\rho_n)$ converges to $\inf_{\rho} \mathcal{F}_{\alpha,\epsilon}(\rho)\geq -\frac{\ln \mathcal{Z}}{\beta}$ where the infimum is taken over probability densities on~$\T^d$.
Since  $\F_{\alpha,\epsilon}(\rho_n)$ is bounded, thus so is
$\mH(\rho_n|\rho^*)_{n\geq 0}$, implying that $\rho_n$ admits a limit point   $\rho^*_{\alpha,\epsilon}$. Since $\mH(\cdot|\rho^*)$ is lower semicontinuous, $\mH(\rho^*_{\alpha,\epsilon}|\rho^*)\leq \sup_n \mH(\rho_n|\rho^*)<\infty$, thus $\rho^*_{\alpha,\epsilon}$  admits a density with respect to $\rho^*$ thus also with respect to  the Lebesgue measure, which we also denote  $\rho^*_{\alpha,\epsilon}$. 
It is easy to verify that $K^m_\epsilon\star\rho^{*,1}_{\alpha,\epsilon}$ is also a limit point of $K^m_\epsilon\star\rho^{1}_{n}$. Using again the lower semicontinuity of the relative entropy, we get that $\F_{\alpha,\epsilon}(\rho^*_{\alpha,\epsilon})\leq \inf_\rho \F_{\alpha,\epsilon}(\rho)$, so that $\rho^*_{\alpha,\epsilon}$ is a minimizer of $\mathcal{F} _{\alpha,\epsilon}$.

The uniqueness of the minimizer is a consequence of  the strict convexity of $\mathcal{F} _{\alpha,\epsilon}$, consequence of the  strict convexity of  $x \mapsto x\ln x$. Indeed, for any two probability densities $\rho_1,\rho_2$ and $\lambda\in[0,1]$, 
$$\lambda\mathcal{F} _{\alpha,\epsilon}(\rho_1)+(1-\lambda)\mathcal{F} _{\alpha,\epsilon}(\rho_2)\geq \mathcal{F} _{\alpha,\epsilon}(\lambda\rho_1+(1-\lambda)\rho_2)$$
with equality if and only if $\lambda\in\{0,1\}$ or $\rho_1=\rho_2$.

\end{proof}

From now on, $\rho^*_{\alpha,\epsilon}$ (resp. $\rho^*_{\alpha}$) always refers to the unique minimizer of $\F_{\alpha,\epsilon}$ (resp. $\F_{\alpha}$) introduced in Proposition~\ref{prop:rho*_uniqueness}. We will need the following technical lemma.

\begin{lem}\label{lem:limit_F}
For any probability density $\rho$ on $\T^d$ with $\F_{\alpha}(\rho)<\infty$, 
$$\lim_{\delta\rightarrow 0^+}\F_\alpha(K^d_\delta\star \rho)=\F_\alpha(\rho).$$  
\end{lem}
\begin{proof}
Since $\F_\alpha(\rho)<\infty$, $\int_{\T^d} \rho\ln \rho$ and $\int_{\T^m} \rho^1\ln\rho^1 $ are finite.
By Lemma~\ref{lem:limit_entropy}, 
$$\lim_{\delta\rightarrow 0^+}\int_{\T^d}K^d_\delta\star \rho\ln(K^d_\delta\star \rho)=\int_{\T^d}\rho\ln\rho.$$
Since both  $U$ and  $K^d_\delta\star U$ are bounded by $\|U\|_\infty$,  Lemma~\ref{lem:cv_Dirac} gives  $K^d_\delta\star U\to U$ a.e. Moreover, $\rho$ being integrable, by dominated convergence theorem,
$$\lim_{\delta\rightarrow 0^+}\int_{\T^d}U (K^d_\delta\star \rho)=\lim_{\delta\rightarrow 0^+}\int_{\T^d}\rho(K^d_\delta\star U)=\int_{\T^d}U\rho.$$
Noticing that $(K^d_\delta\star\rho)^1=K^m_\delta\star\rho^1$, using again  Lemma~\ref{lem:limit_entropy},
$$\lim_{\delta\rightarrow 0^+}\int_{\T^m}(K^d_\delta\star \rho)^1\ln((K^d_\delta\star \rho)^1)=\lim_{\delta\rightarrow 0^+}\int_{\T^m}(K^m_\delta\star \rho^1)\ln(K^m_\delta\star \rho^1)=\int_{\T^m}\rho^1\ln\rho^1.$$
This concludes the proof.
\end{proof}

To characterize the minimizer $\rho^*_{\alpha,\epsilon}$ of $\F_{\alpha,\epsilon}$ for $\epsilon\in(0,1]$, we introduce the map $\Gamma_{\alpha,\epsilon}$ defined as follows. For any probability density $\rho$ on $\T^d$,  $\Gamma_{\alpha,\epsilon}(\rho)$ is the probability density on $\T^d$ satisfying
\begin{equation}\label{eq:Gamma}
    \Gamma_{\alpha,\epsilon}(\rho)\propto\exp(-\beta U - \alpha K^m_\epsilon \star \ln(K^m_\epsilon \star \rho))\,,
\end{equation}
which is well-defined since for any  $\epsilon\in(0,1]$,  $\exp(-\beta U - \alpha K^m_\epsilon \star \ln(K^m_\epsilon \star \rho))$ is continuous and positive, hence in $ L^1(\T^d)$.

\begin{prop}\label{prop:rho*}
For $\epsilon\in(0,1]$,
$\rho^*_{\alpha,\epsilon}$ is the only probability density on $\T^d$ which satisfies the relation
$\Gamma_{\alpha,\epsilon}(\rho)=\rho$.
In addition, for $\epsilon=0$, it holds $\rho^*_\alpha\propto\exp(-\beta(U-\frac{\alpha}{\alpha+1}A))$, and thus $\rho^{*,1}_\alpha\propto\exp(-\frac{\beta}{\alpha+1}A).$    
\end{prop}
\begin{proof}
We only prove this result for $\beta=1$. The result for a general $\beta$ can then be easily deduced, by changing $\F_{\alpha,\epsilon}$ to $\beta \F_{\alpha,\epsilon}$, $U$ to $\beta U$, and $A$ to $\beta A$.

\medskip\noindent
{\bf Step 1: Characterization of $\rho^*_{\alpha,\epsilon}$.} Let us first consider the case $\epsilon \in(0,1]$, and prove that $\Gamma_{\alpha,\epsilon}(\rho^*_{\alpha,\epsilon})=\rho^*_{\alpha,\epsilon}$. The idea is essentially to write the Euler-Lagrange equations associated with the minimization problem defining $\rho^*_{\alpha,\epsilon}$ (see Proposition~\ref{prop:rho*_uniqueness}).
Consider a probability density $\rho$ on $\T^d$ with $\F_{\alpha,\epsilon}(\rho)<\infty$. Writing $\rho_\lambda=\lambda\rho+(1-\lambda)\rho^*_{\alpha,\epsilon},\lambda\in (0,1]$, then by the convexity of $\F_{\alpha,\epsilon}$, $\F_{\alpha,\epsilon}(\rho_\lambda)<\infty$, and 
 since $\rho^*_{\alpha,\epsilon}$ is the minimizer of $\F_{\alpha,\epsilon}$, $\F_{\alpha,\epsilon}(\rho_\lambda)\geq \F_{\alpha,\epsilon}(\rho^*_{\alpha,\epsilon})$. 

Writing $\phi(x)=x\ln x$, let us introduce the scalar valued functions on $\T^d$:
\begin{align*}
g_\lambda&=\frac{1}{\lambda}\left( \phi(\rho_\lambda)-\phi(\rho^*_{\alpha,\epsilon})+U\rho_\lambda-U\rho^*_{\alpha,\epsilon}+\alpha\phi(K^m_\epsilon\star\rho_\lambda)-\alpha\phi(K^m_\epsilon\star\rho^*_{\alpha,\epsilon})\right),\lambda \in(0,1]\\g_0&=(\rho-\rho^*_{\alpha,\epsilon})(1+\ln\rho^*_{\alpha,\epsilon}+U)+\alpha K^m_\epsilon\star(\rho-\rho^*_{\alpha,\epsilon})(1+\ln(K^m_\epsilon\star \rho^*_{\alpha,\epsilon})).
\end{align*}
 Then $\lim_{\lambda\rightarrow 0^+}g_\lambda=g_0.$
 By convexity of $\phi$, $\frac{1}{\lambda}(\phi((1-\lambda)x_0+\lambda x_1 )-\phi(x_0))\leq \phi(x_1)-\phi(x_0)$ for all $x_1,x_0\in \R,\lambda\in(0,1].$ Thus $g_\lambda\leq g_1,a.e.$ for $ \lambda\in[0,1]$. Moreover, from the definition of $\F_{\alpha,\epsilon}$, for $\lambda\in(0,1]$, $g_\lambda \in L^1(\T^d)$, and
$$\int_{\T^d}g_\lambda=\frac{\F_{\alpha,\epsilon}(\rho_\lambda)-\F_{\alpha,\epsilon}(\rho^*_{\alpha,\epsilon})}{\lambda}\geq 0.$$
Thus by Fatou's lemma, 
$$\int_{\T^d}g_1-g_0=\int_{\T^d}\liminf_{\lambda\to 0^+}g_1-g_{\lambda}\leq \liminf_{\lambda\rightarrow 0^+} \int_{\T^d}g_1-g_{\lambda}\leq \int_{\T^d} g_1<+\infty.$$
Notice that $g_1-g_0\geq 0, a.e.$, so that the first integral is well defined. As a consequence, $g_1-g_0\in L^1(\T^d)$ and so is $g_0$, and $\int_{\T^d} g_0=\int_{\T^d} g_1-\int_{\T^d} g_1-g_0\geq 0$. Since $\F_{\alpha,\epsilon}(\rho^*_{\alpha,\epsilon})$ is finite, so that $\rho^*_{\alpha,\epsilon}(\ln\rho^*_{\alpha,\epsilon}+U)+\alpha K^m_\epsilon\star\rho^*_{\alpha,\epsilon}\ln(K^m_\epsilon\star \rho^*_{\alpha,\epsilon})\in L^1(\T^d)$, we also get that
$\rho(\ln\rho^*_{\alpha,\epsilon}+U)+\alpha K^m_\epsilon\star\rho\ln(K^m_\epsilon\star \rho^*_{\alpha,\epsilon})\in L^1(\T^d)$.

In particular, if we choose $\rho=1_{\T^d}$ as the probability density on $\T^d$ that takes the constant value $1$, then we have that $\ln\rho^*_{\alpha,\epsilon}+U+\alpha\ln(K^m_\epsilon\star \rho^*_{\alpha,\epsilon})\in L^1(\T^d).$ By Lemma~\ref{lem:estim_K}, $U$ and $\ln(K^m_\epsilon\star\rho^*_{\alpha,\epsilon})$ are bounded, thus $\ln \rho^*_{\alpha,\epsilon}\in L^1(\T^d).$ This implies that $\rho^*_{\alpha,\epsilon}>0 $ a.e. and $\ln\rho^*_{\alpha,\epsilon}+U+\alpha K^m_\epsilon\star\ln(K^m_\epsilon\star \rho^*_{\alpha,\epsilon})\in L^1(\T^d)$.

For general $\rho$ such that $\F_{\alpha,\epsilon}(\rho)<\infty,$ using that $\int_{\T^d} g_0\geq 0,$
\begin{align*}
0\leq&\int_{\T^d} (\rho-\rho^*_{\alpha,\epsilon})(1+\ln\rho^*_{\alpha,\epsilon}+U)+\alpha K^m_\epsilon\star(\rho-\rho^*_{\alpha,\epsilon})(1+\ln(K^m_\epsilon\star \rho^*_{\alpha,\epsilon}))\\=&\int_{\T^d}(\rho-\rho^*_{\alpha,\epsilon})\left(\ln\rho^*_{\alpha,\epsilon}+U+\alpha K^m_\epsilon\star\ln(K^m_\epsilon\star\rho^*_{\alpha,\epsilon})\right).  
\end{align*}
By Lemma~\ref{lem:estim_K}, when $\delta\in (0,1],x\in \T^d$, $\F_{\alpha,\epsilon}(K^d_\delta(x-\cdot))<\infty,$ thus we can take $\rho(y)=K^d_\delta(x-y),y\in\T^d,\delta\in(0,1]$ in the inequality above. By passing to the limit $\delta\rightarrow 0^+$, by Lemma~\ref{lem:cv_Dirac}, it holds, a.e.
$$\ln\rho^*_{\alpha,\epsilon}+U+\alpha K^m_\epsilon\star\ln(K^m_\epsilon\star\rho^*_{\alpha,\epsilon})\geq \int_{\T^d}\rho^*_{\alpha,\epsilon}(\ln\rho^*_{\alpha,\epsilon}+U+\alpha K^m_\epsilon\star\ln(K^m_\epsilon\star\rho^*_{\alpha,\epsilon}))=\F_{\alpha,\epsilon}(\rho^*_{\alpha,\epsilon}).$$

Let us define   $f=\ln\rho^*_{\alpha,\epsilon}+U+\alpha K^m_\epsilon\star\ln(K^m_\epsilon\star\rho^*_{\alpha,\epsilon})-\F_{\alpha,\epsilon}(\rho^*_{\alpha,\epsilon}).$ Then $f\geq 0$ a.e., thus $f\rho^*_{\alpha,\epsilon}\geq 0$ a.e. Since $\int_{\T^d}f \rho^*_{\alpha,\epsilon}=0$, one gets $f\rho^*_{\alpha,\epsilon}=0,$ a.e., and thus $f=0$ a.e. since $\rho^*_{\alpha,\epsilon}>0$ a.e.  Thus $\rho^*_{\alpha,\epsilon}=\exp(-U-\alpha K^m_\epsilon\star \ln (K^m_\epsilon\star \rho^*_{\alpha,\epsilon})+\F_{\alpha,\epsilon}(\rho^*_{\alpha,\epsilon}))$ a.e. which yields $\rho^*_{\alpha,\epsilon}=\Gamma_{\alpha,\epsilon}(\rho^*_{\alpha,\epsilon}).$

\medskip\noindent
{\bf Step 2: Uniqueness of the fixed point of $\Gamma_{\alpha,\epsilon}$.} Let us now prove that $\rho^*_{\alpha,\epsilon}$ is the only fixed point of  $\Gamma_{\alpha,\epsilon}$. Suppose that $\hat\rho^*_{\alpha,\epsilon}$ is a probability density satisfying
$\Gamma_{\alpha,\epsilon}(\hat\rho^*_{\alpha,\epsilon})=\hat\rho^*_{\alpha,\epsilon}.$ Let us define $F(\lambda)=\F_{\alpha,\epsilon}((1-\lambda)\hat\rho^*_{\alpha,\epsilon}+\lambda\rho^*_{\alpha,\epsilon}), \lambda\in [0,1].$ Notice that $\rho^*_{\alpha,\epsilon},\hat\rho^*_{\alpha,\epsilon}$ both belonging to the image of $\Gamma_{\alpha,\epsilon}$, they are bounded from below and above by positive constants, and this justifies the use of dominated convergence theorem which leads to the right derivative of $F$ at $0$, as
\begin{align*}
F'_+(0)=\lim_{\lambda\rightarrow 0^+}\frac{F(\lambda)-F(0)}{\lambda}=\int_{\T^d}(\rho^*_{\alpha,\epsilon}-\hat\rho^*_{\alpha,\epsilon})(\ln \hat\rho^*_{\alpha,\epsilon}+U+\alpha K^m_\epsilon\star\ln(K^m_\epsilon\star \hat\rho^*_{\alpha,\epsilon}))=0.
\end{align*}
The last equality follows from $\Gamma_{\alpha,\epsilon}(\hat \rho^*_{\alpha,\epsilon})=\hat \rho^*_{\alpha,\epsilon}$, which implies  that $\ln \hat\rho^*_{\alpha,\epsilon}+U+\alpha K^m_\epsilon\star\ln(K^m_\epsilon\star \hat\rho^*_{\alpha,\epsilon})$
is constant.

Then by the convexity of $F$ inherited from the convexity of $\F_{\alpha,\epsilon}$ we have that $\F_{\alpha,\epsilon}(\hat\rho^*_{\alpha,\epsilon})=F(0)\leq F(1)=\F_{\alpha,\epsilon}(\rho^*_{\alpha,\epsilon}).$ Since $\rho^*_{\alpha,\epsilon}$ is the unique minimizer of $\F_{\alpha,\epsilon}$, we have that $\rho^*_{\alpha,\epsilon}=\hat \rho^*_{\alpha,\epsilon}.$ This concludes the proof of the properties stated in the case $\epsilon\in(0,1]$.

\medskip\noindent
{\bf Step 3: Results on $\rho^*_\alpha$.} Let us now consider the case $\epsilon=0$. Let us prove  that the probability density  $\hat\rho^*_\alpha\propto\exp(-U+\frac{\alpha}{\alpha+1}A)$ is the minimizer of $\F_{\alpha}.$ For a probability density $\rho$ such that $\F_{\alpha}(\rho)<\infty,\delta\in(0,1]$, we define $F(\lambda)=\F_{\alpha}((1-\lambda)\hat\rho^*_\alpha+\lambda K^d_\delta\star \rho),\lambda \in [0,1]$. Then, since $\hat\rho^*_\alpha$ and $K^d_\delta\star\rho$ are bounded from below and above by positive constants, it holds
\begin{align*}
F'_+(0)=\lim_{\lambda\rightarrow 0^+}\frac{F(\lambda)-F(0)}{\lambda}=\int_{\T^d}(K^d_\delta\star\rho-\hat\rho^*_{\alpha})(\ln \hat\rho^*_{\alpha}+U+\alpha \ln\hat\rho^{*,1}_{\alpha})=0.
\end{align*}
The last equality is again a consequence of the fact that $\ln \hat\rho^*_{\alpha}+U+\alpha \ln\hat\rho^{*,1}_{\alpha}$ is a constant, by definition of $\hat\rho^{*,1}_{\alpha}$.

By the convexity of $F$, $\F_{\alpha}(\hat\rho^*_{\alpha})=F(0)\leq F(1)=\F_{\alpha}(K^d_\delta\star\rho)$. Thanks to Lemma~\ref{lem:limit_F}, one thus gets
$\F_{\alpha}(\hat\rho^*_{\alpha})\leq \F_{\alpha}(\rho).$ This implies that $\hat\rho^*_\alpha$ is the minimizer of $\F_{\alpha}.$ Thus, by the uniqueness of the minimizer of $\F_\alpha$ stated in Proposition~\ref{prop:rho*_uniqueness}, $\rho^*_\alpha=\hat \rho^*_\alpha\propto \exp(-U+\frac{\alpha}{\alpha+1}A)$, which concludes the proof in the case $\epsilon=0$.
\end{proof}

\begin{cor}\label{cor:inf_rho}
The function $\rho^*_\alpha$ satisfies the following: $\inf_{\alpha>0} \inf_{x\in \T^d} \rho^*_\alpha(x)>0$.
\end{cor}
\begin{proof}
By Proposition~\ref{prop:rho*}, $\rho^*_\alpha\propto\exp(-\beta(U-\frac{\alpha}{\alpha+1}A)).$ Conclusion follows by noticing that $$\inf U-\max\{0,\sup A\}\leq U-\frac{\alpha}{\alpha+1}A\leq \sup U-\min\{0,\inf A\}.$$
\end{proof}

In order to prove Proposition~\ref{prop:rho_alpha}, we will also need 
the standard Holley--Stroock perturbation criterion for log-Sobolev 
inequalities. For the reader's convenience, we recall a standard form of this 
criterion (see, for example,~\cite{ane-2000,deuschel-stroock-1990}).
\begin{prop}\label{prop:HS}
(Holley–Stroock) If a probability density \(\rho\) on $E=\T^d$ or $\T^d\times \R^d$  satisfies the log-Sobolev inequality with constant $D$ and if \(\tilde{\rho} = e^{-\psi} \rho\) is also a probability density on $E$ for some \(\psi \in L^{\infty}(E)\), then \(\tilde{\rho}\) satisfies a log-Sobolev inequality with constant \(\tilde{D} = D e^{\inf \psi - \sup \psi}\).    
\end{prop}

We are now in position to prove Proposition~\ref{prop:rho_alpha}.
\begin{proof}[Proof of Proposition~\ref{prop:rho_alpha}]
By Propositions~\ref{prop:rho*_uniqueness} and~\ref{prop:rho*}, for all $\alpha > 0$, $\mathcal F_\alpha$ admits a unique minimizer 
\begin{equation*}
\rho^*_\alpha \propto \exp \po -\beta\left(U - \frac{\alpha}{\alpha+1} A\right)\pf\,.
\end{equation*}
The log-Sobolev inequality for $\rho^*_\alpha$  is then a direct consequence of  Proposition~\ref{prop:HS}, since $A \in L^\infty(\T^m)$.
\end{proof}

\begin{rem}\label{rem:Gamma} Similarly to $\Gamma_{\alpha,\epsilon}$, it is possible to introduce a map $\Gamma_\alpha$ such that, for any probability density $\rho$ on~$\T^d$,
$\Gamma_\alpha(\rho) \propto \exp(-\beta U-\alpha\ln \rho^1)$. The probability density $\Gamma_\alpha(\rho)$ is, of course, not well-defined for all $\rho$. A sufficient condition for $\Gamma_\alpha(\rho)$ to be well defined is for example $\inf_{\T^d}\rho>0$. However,  $\exp(-\beta U-\alpha\ln \rho^1)\notin L^1(\T^d)$ (and consequently $\Gamma_\alpha(\rho)$ is not well-defined) if we consider for example $\rho$ such that $\rho^1$ is zero over an open interval. 

By the explicit formula for $\rho^*_\alpha$ obtained in Proposition~\ref{prop:rho*}, we deduce that $\Gamma_\alpha(\rho^*_\alpha)=\rho^*_\alpha$. Besides, by Lemma~\ref{lem:cv_Dirac}, for any regular probability density $\rho$ on $\T^d$, we know that $\Gamma_{\alpha,\epsilon}(\rho)\rightarrow\Gamma_{\alpha}(\rho)$ as $\epsilon \to 0$. The characterization of $\rho^*_{\alpha,\epsilon}$ as the fixed point of the map $\Gamma_{\alpha,\epsilon}$ thus passes to the limit $\epsilon \to 0$ in order to provide a similar characterization for $\rho^*_\alpha$. In particular, $\rho^*_\alpha$ is the unique fixed point of $\Gamma_\alpha$ on the set of probability densities $\rho$ on $\T^d$ such that $\inf_{\T^d} \rho > 0$. 
\end{rem}

\subsection{Convergence of the free energy minimizer $\rho^*_{\alpha,\epsilon}$ when $\epsilon \to 0$}\label{sec:CV_H1}

As seen in Proposition~\ref{prop:rho*} and Remark~\ref{rem:Gamma}, the minimizers of $\mathcal F_{\alpha,\epsilon}$ and $\mathcal F_{\alpha}$ are characterized as fixed points of maps $\Gamma_{\alpha,\epsilon}$ and $\Gamma_\alpha$. The question of the convergence of $\rho_{\alpha,\epsilon}^*$ towards $\rho_\alpha^*$ thus requires to quantify how the fixed point of $\Gamma_{\alpha,\epsilon}$ depends on $\epsilon$, which will be done using estimates obtained from the Banach fixed point theorem, see Lemma~\ref{lem:banach_fixedpoint} below, which can be seen as a variant of the implicit function theorem.


Let us start with a technical lemma. Let us recall that the function $A \in C^\infty (\T^m) $ is defined by~\eqref{eq:FE_x1}, and is such that $\rho^{*,1}=\exp(-\beta A)$.

\begin{lem}\label{prop:F_endo_Hk} Suppose that $k\in \N^+$ and $2k>m$.
For $\epsilon\geq 0$ and $f\in H^k(\T^m)$ with $\inf f>0$, let us define 
\begin{equation}\label{eq:F}
F(\epsilon,f)=\ln f+\beta A+\alpha K^m_\epsilon\star\ln(K^m_\epsilon\star f),
\end{equation}
where by convention $K^m_0\star\ln(K^m_0\star f)=\ln f$. Then $F(\epsilon,f)\in H^k(\T^m).$
\end{lem}
\begin{proof}
Notice that by Lemma~\ref{lem:convolution}, $K^m_\epsilon\star f\in H^k(\T^m)$. Furthermore, $\inf K^m_\epsilon\star f\geq \inf f>0$, and by Lemma~\ref{lem:embedding}, $\sup f<\infty$, thus $\sup K^m_\epsilon\star f\leq \sup f<\infty$. Moreover,  $\ln x$ and $(\ln x)^{(j)}=(-1)^{j-1} \frac{(j-1)!}{x^j}$ for $j\geq 1$ are bounded over $[\inf f,\sup f]\supseteq[\inf K^m_\epsilon\star f,\sup K^m_\epsilon\star f]$. Thus by Proposition~\ref{prop:phi(g)}, $\ln(f),\ln(K^m_\epsilon\star f)\in H^k(\T^m)$. Using again Lemma~\ref{lem:convolution}, $K^m_\epsilon\star\ln(K^m_\epsilon\star f)\in H^k(\T^m).$ Since $A$ is smooth, hence in $H^k(\T^m)$, we get that $F(\epsilon,f)\in H^k(\T^m).$ 
\end{proof}

In Theorem~\ref{thm:implicit} below, 
for sufficiently small $\epsilon \ge 0$, we prove the existence of a zero $f_\epsilon$ of the map $f\mapsto F(\epsilon,f)$, such that $f_\epsilon\rightarrow f_0$ when $\epsilon\rightarrow 0^+$, in some Sobolev norm. Then, as we will see in the proof of Theorem~\ref{thm:cv_infty} below, we will express $\rho^*_{\alpha,\epsilon}$ in terms of $f_\epsilon$,  and thus deduce that $\rho^*_{\alpha,\epsilon}\rightarrow \rho^*_\alpha$ in a suitable sense, as $\epsilon\rightarrow 0^+$. 
The proof of Theorem~\ref{thm:implicit} relies on a contraction property and Lemma~\ref{lem:banach_fixedpoint} taken from~\cite[Corollaries 1 and 3 (page 229-230)]{loomis-sternberg-1968} with a slight modification of the notation.
\begin{lem}\label{lem:banach_fixedpoint}
Let $(X,d)$ be a complete metric space.
\begin{enumerate}
    \item Let ${\mathcal B}(x_0,r)=\{x \in X, d(x,x_0) \le r\}$ be the closed ball centered at $x_0$ with radius $r>0$, and let $\Phi: {\mathcal B}(x_0,r) \to X$ be a contraction with constant $\|\Phi\|_{Lip}\in [0,1)$ such that:
$$d(x_0,\Phi(x_0)) \le (1 - \|\Phi\|_{Lip})r.$$
 Then $\Phi({\mathcal B}(x_0,r))\subseteq {\mathcal B}(x_0,r)$, and $\Phi$ has a unique fixed point in ${\mathcal B}(x_0,r)$.
 \item  Let $\Phi:X \to X$ be a function, let $x_0 \in X$ be a point and let us set $d=d(x_0,\Phi(x_0))$. Assume that $\Phi$ is a contraction with constant $\|\Phi\|_{Lip} \in [0,1)$ on ${\mathcal B}\left(x_0,\frac{d}{1-\|\Phi\|_{Lip}}\right)$. Then, $\Phi$ admits a unique fixed point in ${\mathcal B}\left(x_0,\frac{d}{1-\|\Phi\|_{Lip}}\right)$.
\end{enumerate}
\end{lem}

\begin{proof}
    Let us provide the proof for completeness. The first result is a simple consequence of the Banach fixed point theorem. Indeed for any $x \in {\mathcal B}(x_0,r)$, 
\begin{align*}
    d(\Phi(x),x_0)&\le d(\Phi(x),\Phi(x_0))+ d(\Phi(x_0),x_0)\\
    &\le \|\Phi\|_{Lip} d(x,x_0)+ (1 - \|\Phi\|_{Lip})r\leq r
\end{align*}
and thus $\Phi({\mathcal B}(x_0,r))\subseteq {\mathcal B}(x_0,r)$ and $\Phi:{\mathcal B}(x_0,r) \to {\mathcal B}(x_0,r)$ is a contraction over the complete metric space ${\mathcal B}(x_0,r)$, and therefore admits a unique fixed point in ${\mathcal B}(x_0,r)$.

The second result is a direct consequence of the first one with $r=\frac{d}{1-\|\Phi\|_{Lip}}$. 
\end{proof}

We are now in position to state and prove Theorem~\ref{thm:implicit}.

\begin{theorem}\label{thm:implicit}
Suppose that $k\in \mathbb{N^+}$ and $2k>m$. There exist constants $C_1,C_2>0$ independent of $\alpha,\epsilon$, such that for all $\epsilon \in [0, \frac{C_1}{\alpha+1}]$, 
there exists  $f_\epsilon\in H^k(\T^m)$ with $\inf_{\T^m} f_\epsilon>0$ such that $F(\epsilon,f_\epsilon)=0$ (where $F$ is defined by~\eqref{eq:F}),  with additionaly $\|f_\epsilon-f_0\|_{H^k}\leq C_2{\epsilon}$. Moreover,  $f_0=\exp(-\frac{\beta A}{\alpha+1}).$  
\end{theorem}  
\begin{proof}
In all this proof, the only property we use about $A$ is that $A\in C^\infty(\T^m).$ Besides, we can assume without loss of generality that $\beta=1$, the general case being simply obtained by replacing $A$ by~$\beta A$.

First of all, when $\epsilon=0,$ $f_0=\exp(-\frac{A}{\alpha+1})$ already satisfies the requirements that $F(0,f_0)=0,f_0\in H^k(\T^m), \inf f_0>0$. In the rest of the proof we assume that $\epsilon \in (0,   1]$ (since in the statement of Theorem~\ref{thm:implicit} we can always chose $C_1$ to be less than $1$, so that the condition $\epsilon\leq \frac{C_1}{\alpha+1}$ implies $\epsilon\leq   1$). Thus all the estimates on $K_\epsilon^m$ requiring $\epsilon\leq1$ proven in Appendix~\ref{sec:ap} can be used.

We define, for $f\in H^k(\T^m)$ such that $\inf_{\T^m} f>0$, \begin{equation}\label{eq:K}
K(\epsilon,f)=f-\frac{\exp{(-\frac{A}{\alpha+1})}}{\alpha+1}F(\epsilon,f)=f-\frac{f_0}{\alpha+1}F(\epsilon,f).  
\end{equation}
By Lemma~\ref{prop:F_endo_Hk}, $F(\epsilon,f)\in H^k(\T^m).$ Since $f_0\in H^k(\T^m)$, then by Lemma~\ref{lem:product},  $f_0 F(\epsilon,f) \in H^k(\T^m)$, and thus $K(\epsilon,f)\in H^k(\T^m)$. Let us denote $H^k_+(\T^m)$ the subset of functions $f \in H^k(\T^m)$  such that $\inf_{\T^m} f>0$. By Equation~\eqref{eq:embedding_infty} from Lemma~\ref{lem:embedding}, $H^k_+(\T^m)$ is an open set of $H^k(\T^m)$. Thus for fixed $\epsilon$, $K(\epsilon,\cdot)$ can be viewed as a map from  $H^k_+(\T^m)$    to $H^k(\T^m).$

Notice that if $f_\epsilon$ exists, then it should be a fixed point in $H^k(\T^m)$ of the map $f\mapsto K(\epsilon,f)$. The map $K(\epsilon,f)$ is built such that its Fréchet derivative $D_f K(0,f_0)$ at $(0,f_0)$ is zero, and the hope is thus to build a contraction for small enough $\epsilon$.


Let us first give the main ideas of the proof before presenting the details of the arguments. The proof is decomposed into four steps:
\begin{enumerate}
\item The (Fréchet) derivative of $K(\epsilon,\cdot)$ is well defined and can be explicitly written. More precisely, there exists a constant $c>0$, such that for all $h_1\in H^k(\T^m)$ with $\|h_1-f_0\|_{H^k}\leq c\inf f_0$, and $\inf h_1>0$, for all $\epsilon \in (0,1]$, the map $f\mapsto K(\epsilon,f)$ is Fréchet differentiable at $f=h_1$. Moreover, the Fréchet derivative is given by:
for all $g\in H^k(\T^m)$,
\begin{equation}\label{eq:DfK}
D_fK(\epsilon,h_1)(g)=g-\frac{f_0}{\alpha+1}\frac{g}{h_1}-\frac{\alpha f_0}{\alpha+1} \left( K^m_{\epsilon}\star\frac{K^m_{\epsilon}\star g}{K^m_{\epsilon}\star h_1}\right).
\end{equation}
\item The linear operator $g \in H^k(\T^m) \mapsto D_fK(\epsilon,f_0)(g) \in H^k(\T^m) $  has an operator norm smaller than $\frac{\alpha}{\alpha+1}\left(1+C\frac{\sqrt{\epsilon}}{1+\alpha}\right)$, for a constant $C$ independent of $\alpha$ and $\epsilon \in (0,1]$.
\item For any $h_1 \in H^k(\T^m)$ such that
$\|h_1-f_0\|_{H^k}\le c\inf f_0$, where $c>0$ is the constant introduced in step 1, the linear operator $g\in H^k(\T^m)\mapsto D_fK(\epsilon,h_1)(g)-D_fK(\epsilon,f_0)(g)\in H^k(\T^m)$ has an operator norm smaller than $C\|f_0-h_1\|_{H^k}$, for a constant $C$ independent of $\alpha,\epsilon$ and $h_1$.
\item From the two previous steps, by the mean value theorem, we obtain that for any $h_1,h_2 \in H^k(\T^m)$ such that $\max( \|h_1-f_0\|_{H^k},\|h_2-f_0\|_{H^k}) \le c\inf f_0$, where $c>0$ is the constant introduced in step 1, 
\begin{align}
    &\|K(\epsilon,h_1)-K(\epsilon,h_2)\|_{H^k}\nonumber\\
    &\leq \left(M_1\max(\|h_1-f_0\|_{H^k},\|h_2-f_0\|_{H^k})+\frac{\alpha}{\alpha+1}\left(1+M_1\frac{\sqrt{\epsilon}}{1+\alpha}\right)\right)\|h_1-h_2\|_{H^k},\label{eq:MVT}
\end{align}
for a constant $M_1$ independent of $\alpha,\epsilon$.  The conclusion will then follow by applying~Lemma \ref{lem:banach_fixedpoint} to the function $K(\epsilon,\cdot)$.
\end{enumerate}
Let us now provide the details of each step. Each time an $H^k$ norm is used, all the terms included in the $H^k$ norm belong to $H^k(\T^m)$. We check this only in Step 1, and we do not repeat the arguments in the other steps in order to avoid repetitive discussions. Readers can check it easily by using the same arguments. In the following proof, $C$ is a positive constant independent of $\alpha,\epsilon$, which may vary from one occurrence to another.

\medskip\noindent
\textbf{Step 1.} 
For any $g \in H^k(\T^m)$, let us denote by
$$\delta(g)=g-\frac{f_0}{\alpha+1}\frac{g}{h_1}-\frac{\alpha f_0}{\alpha+1} \left( K^m_{\epsilon}\star\frac{K^m_{\epsilon}\star g}{K^m_{\epsilon}\star h_1}\right)$$
the function in the right-hand side of~\eqref{eq:DfK}.

Let
\begin{equation}\label{eq:c}
  c=\frac{1}{2C_\infty},  
\end{equation}
 where $C_\infty$ is introduced in Lemma~\ref{lem:embedding}. We suppose that  $h_1,g \in H^k(\T^m)$, and  $\|f_0-h_1\|_{H^k}\leq c\inf f_0$, $\|g\|_{H^k}\leq \frac{c}{2} \inf f_0$. We have $\inf h_1\geq \inf f_0-\|f_0-h_1\|_\infty\geq \frac{1}{2} \inf f_0>0$, 
 and similarly, $\inf (h_1+g)\geq\frac{1}{4} \inf  f_0.$ Thus by Lemma~\ref{prop:F_endo_Hk} and Lemma~\ref{lem:product}, $K(\epsilon,h_1),K(\epsilon,h_1+g)\in H^k(\T^m).$ 
 
Moreover, since $h_1\in H^k(\T^m)$ and $\inf h_1>0$, by Lemma~\ref{lem:inverse_Kf_1}, $\frac{1}{h_1}\in H^k(\T^m)$, and then by Lemma~\ref{lem:product}, $\frac{f_0g}{h_1}\in H^k(\T^m).$  Using Lemma~\ref{lem:convolution}, one has $f_0( K^m_{\epsilon}\star\frac{K^m_{\epsilon}\star g}{K^m_{\epsilon}\star h_1})\in H^k(\T^m)$. Thus, it holds $$\forall g \in H^k(\T^m), \, \delta(g) \in H^k(\T^m).$$

Proving~\eqref{eq:DfK}, namely that the Fréchet derivative $D_fK(\epsilon,h_1)(g)$ of $f\mapsto K(\epsilon,f)$ at  $f=h_1$ is the function $\delta(g)$, amounts to proving that (notice that from the discussion above, $K(\epsilon,h_1),\, K(\epsilon,h_1+g), \,\delta(g) \in H^k(\T^m)$):
\[
\lim_{ \|g\|_{H^k} \to 0} \frac{\| K(\epsilon, h_1 + g) - K(\epsilon, h_1) - \delta(g) \|_{H^k}}{\| g \|_{H^k}} = 0.
\]
This is equivalent to proving that
$$\frac{1}{\alpha +1}\left\|f_0\left(F(\epsilon,h_1+g) - F(\epsilon,h_1) - \left( \frac{g}{h_1}+\alpha K^m_{\epsilon}\star\frac{K^m_{\epsilon}\star g}{K^m_{\epsilon}\star h_1}\right) \right) \right\|_{H^k} = o(\|g\|_{H^k}).$$
Using Lemma~\ref{lem:product}, since $f_0 \in H^k$, to get~\eqref{eq:DfK}, it thus suffices to prove that 
\begin{align}
&\left\|\ln(h_1+g)-\ln(h_1)-\frac{g}{h_1}\right\|_{H^k}
= o(\|g\|_{H^k})\label{eq:DfF1} \\
&\left\|\alpha K^m_\epsilon\star\left(\ln(K^m_\epsilon\star(h_1+g)) - \ln(K^m_\epsilon\star h_1)-\frac{K^m_{\epsilon}\star g}{K^m_{\epsilon}\star h_1}\right) \right\|_{H^k}= o(\|g\|_{H^k}).    \label{eq:DfF2}
\end{align}

This is a consequence of the the following estimates. On the one hand,  we apply Proposition~\ref{prop:estimate_h} to $\phi(x)=\ln (1+x)-x$
$$\left\|\ln(h_1+g)-\ln(h_1)-\frac{g}{h_1}\right\|_{H^k}=o\po \left\|\frac{g}{h_1}\right\|_{H^k}\pf=o(\|g\|_{H^k})$$
where we used Lemma~\ref{lem:product} to ensure that $\|\frac{g}{h_1}\|_{H^k}\leq C\|g\|_{H^k}\|\frac{1}{h_1}\|_{H^k}$ for some constant $C$.
This yields~\eqref{eq:DfF1}.

On the other hand, to get~\eqref{eq:DfF2}, notice first that by Lemma \ref{lem:convolution}, the convolution by $K^m_\epsilon$ decreases the $H^k$ norm. The result is thus obtained following the same reasoning as for~\eqref{eq:DfF1}, replacing $g$ by $K^m_\epsilon\star g$ and $h_1$ by $K^m_\epsilon\star h_1$, using the linearity of the convolution and Lemma~\ref{lem:inverse_Kf_1}. 

This concludes the proof of~\eqref{eq:DfK}. 

\medskip\noindent
\textbf{Step 2.} For a given $g \in H^k(\T^m)$, we would like to estimate the $H^k$ norm of 
$$D_fK(\epsilon,f_0)(g)=\frac{\alpha}{\alpha+1}g-\frac{\alpha}{\alpha+1}f_0 \left( K^m_{\epsilon}\star\frac{K^m_{\epsilon}\star g}{K^m_{\epsilon}\star f_0}\right) \in  H^k(\T^m)$$
in terms of $\|g\|_{H^k}$. First, by Lemma \ref{lem:convolution}, one has
\begin{equation}\label{eq:step2.1}
\frac{\alpha}{\alpha+1}\|g-K^m_{2\epsilon}\star g\|_{H^k}\leq \frac{\alpha}{\alpha+1}\|g\|_{H^k}.
\end{equation}
Let us now estimate 
$$\left\|K^m_{2\epsilon}\star g-f_0 \left( K^m_{\epsilon}\star\frac{K^m_{\epsilon}\star g}{K^m_{\epsilon}\star f_0}\right)\right\|_{H^k}.$$
We will use that by Lemma~\ref{lem:e_var}, $K^m_{2\epsilon}=K^m_{\epsilon}\star K^m_{\epsilon}.$


By the triangle inequality, it is sufficient to estimate separately, 
\begin{equation}\label{eq:step1_triangle} \left\|K^m_{2\epsilon}\star g-f_0 \left( K^m_{\epsilon}\star\frac{K^m_{\epsilon}\star g}{f_0}\right)\right\|_{H^k},\left\|f_0 \left( K^m_{\epsilon}\star\frac{K^m_{\epsilon}\star g}{f_0}\right)-f_0 \left( K^m_{\epsilon}\star\frac{K^m_{\epsilon}\star g}{K^m_{\epsilon}\star f_0}\right)\right\|_{H^k}.
\end{equation}

 By Lemma~\ref{lem:convolution} and Lemma~\ref{lem:product}, the second term is bounded by 
$$C\|g\|_{H^k} \|f_0\|_{H^k}\left\|\frac{1}{f_0}\right\|_{H^k}\left\|\frac{1}{K^m_\epsilon\star f_0}\right\|_{H^k}\|f_0-K^m_\epsilon\star f_0\|_{H^k}.$$
By Lemma~\ref{lem:inverse_Kf_1}, $\|\frac{1}{ f_0}\|_{H^k}$ and $\|\frac{1}{K^m_\epsilon\star f_0}\|_{H^k}$
are bounded by a constant independent of $\alpha,\epsilon$. By Lemma~\ref{lem:Kf-f} and the fact that $C_1 \le 1$ so that $\epsilon\leq 1$, it holds  
$\|f_0-K^m_\epsilon\star f_0\|_{H^k}\leq \frac{C{\epsilon}}{\alpha+1}\leq \frac{C{\sqrt{\epsilon}}}{\alpha+1}$. 

For the first term in~\eqref{eq:step1_triangle}, again by Lemma~\ref{lem:product}, $$\left\|K^m_{2\epsilon}\star g-f_0 \left( K^m_{\epsilon}\star\frac{K^m_{\epsilon}\star g}{f_0}\right)\right\|_{H^k}\leq C\|f_0\|_{H^k}\left\|\frac{1}{f_0}K^m_{2\epsilon}\star g- \left( K^m_{\epsilon}\star\frac{K^m_{\epsilon}\star g}{f_0}\right)\right\|_{H^k}.$$ 
Let us define $g_1=\frac{1}{f_0}$ and $g_2=K^m_\epsilon\star g$, so that we need to estimate 
$$\left\|g_1K^m_{\epsilon}\star g_2- K^m_{\epsilon}\star(g_1g_2)\right\|_{H^k}.$$
Therefore, it suffices to estimate, for $\theta_1,\theta_2\in \mathbb{N}^m$ such that $|\theta_1|+|\theta_2|\leq k$, 
$$\left\|D^{\theta_1}g_1K^m_{\epsilon}\star D^{\theta_2}g_2- K^m_{\epsilon}\star(D^{\theta_1}g_1D^{\theta_2}g_2)\right\|_{2}.$$
Since $D^{\theta_1}g_1\in C^1(\T^m),D^{\theta_2}g_2\in L^2(\T^m)$, we apply the Lemma~\ref{lem:h_1h_2}, and notice that $\|D^{\theta_2} g_2\|_2\leq \|g_2\|_{H^k},$ 
$$\left\|D^{\theta_1}g_1K^m_{\epsilon}\star D^{\theta_2}g_2- K^m_{\epsilon}\star(D^{\theta_1}g_1D^{\theta_2}g_2)\right\|_{2}\leq C\sqrt{\epsilon m} \|D^1 (D^{\theta_1}g_1)\|_\infty\|g_2\|_{H^k},$$
where we recall that, for any function $f$, $\|D^1 f\|_\infty =\max_{|\theta|=1}\|D^\theta f\|_\infty$. Notice that $\|D^1 (D^{\theta_1}g_1)\|_\infty=\|D^1 (D^{\theta_1}\frac{1}{f_0})\|_\infty\leq \frac{C}{\alpha+1}$ for a constant $C$ independent of $\alpha,\epsilon$, since $\frac{1}{f_0}=\exp(\frac{A}{\alpha+1}).$
Thus $$\left\|K^m_{2\epsilon}\star g-f_0 \left( K^m_{\epsilon}\star\frac{K^m_{\epsilon}\star g}{K^m_{\epsilon}\star f_0}\right)\right\|_{H^k}\leq \frac{C\sqrt{\epsilon}}{\alpha+1}\|g\|_{H^k}.$$

We thus have proved that for all $g \in H^k(\T^m)$,
$$\|D_fK(\epsilon,f_0)(g)\|_{H^k}\leq 
\frac{\alpha}{\alpha+1}\left(1+C\frac{\sqrt{\epsilon}}{1+\alpha}\right)\|g\|_{H^k}.$$
This concludes the proof of Step 2.

\medskip\noindent
\textbf{Step 3.} Let $h_1 \in H^k(\T^m)$ be such that
$\|f_0-h_1\|_{H^k} \le c \inf f_0$, where $c$ is given by~\eqref{eq:c} in Step 1, and let $g \in H^k(\T^m)$. Notice that from Step 1, 
\begin{align}
&\|D_fK(\epsilon,h_1)(g)-D_fK(\epsilon,f_0)(g)\|_{H^k}\nonumber\\
&\leq\frac{1}{\alpha+1} \left\|\frac{f_0g}{h_1}-g\right\|_{H^k}+\frac{\alpha}{\alpha+1} \left\|f_0 \left( K^m_{\epsilon}\star \left(\frac{K^m_{\epsilon}\star g}{K^m_{\epsilon}\star h_1}-\frac{K^m_{\epsilon}\star g}{K^m_{\epsilon}\star f_0}\right)\right)\right\|_{H^k}.\label{eq:step3.1}
\end{align}
By Lemma \ref{lem:product},
\begin{equation}\label{eq:step3.2}
\left\|\frac{f_0g}{h_1}-g\right\|_{H^k}\leq  C\|f_0-h_1\|_{H^k} \left\|\frac{1}{h_1}\right\|_{H^k} \|g\|_{H^k}.
\end{equation}
Notice that again  
$\|\frac{1}{h_1}\|_{H^k}$ is bounded independently of $\alpha,\epsilon$, by Lemma~\ref{lem:inverse_Kf_1}.

For the second term, we can deduce similarly by Lemmas \ref{lem:convolution} and \ref{lem:product} that 
\begin{equation}\label{eq:step3.3}
\left\|f_0\left( K^m_{\epsilon}\star\left(\frac{K^m_{\epsilon}\star g}{K^m_{\epsilon}\star h_1}-\frac{K^m_{\epsilon}\star g}{K^m_{\epsilon}\star f_0}\right)\right)\right\|_{H^k}\leq C\|f_0\|_{H^k}\|g\|_{H^k}\|f_0-h_1\|_{H^k} \left\|\frac{1}{K^m_\epsilon\star f_0}\right\|_{H^k}\left\|\frac{1}{K^m_\epsilon\star h_1}\right\|_{H^k}.
\end{equation}
Notice again that by Lemma~\ref{lem:inverse_Kf_1}, $\left\|\frac{1}{K^m_\epsilon\star f_0}\right\|_{H^k},\left\|\frac{1}{K^m_\epsilon\star h_1}\right\|_{H^k}$ are bounded independently of $\alpha,\epsilon$, and so is $\|f_0\|_{H^k}$.

By combining~\eqref{eq:step3.1}-\eqref{eq:step3.2}-\eqref{eq:step3.3}, one obtains that there exist constants $c,C>0$ independent of $\alpha,\epsilon$, such that, for all $h_1 \in H^k(\T^m)$, such that $\|f_0-h_1\|_{H^k}\le c \inf f_0$,
$$\|D_fK(\epsilon,h_1)(g)-D_fK(\epsilon,f_0)(g)\|_{H^k}\leq C\|f_0-h_1\|_{H^k}\|g\|_{H^k}.$$

\medskip\noindent
\textbf{Step 4.} The estimate~\eqref{eq:MVT} is a straightforward consequence of the mean value theorem. Let us give some details. Let us introduce, for $t \in [0,1]$, $k(t)=K(\epsilon,t h_1 + (1-t) h_2)$. By the mean value theorem, $K(\epsilon,h_1)-K(\epsilon,h_2)=k(1)-k(0)=k'(t^*)=D_f K(\epsilon, t^* h_1 + (1-t^*) h_2) (h_1-h_2)$ for some $t^* \in [0,1]$. To get~\eqref{eq:MVT}, we now have to estimate the operator norm of $D_f K(\epsilon, t^* h_1 + (1-t^*) h_2)$ over $H^k$, that we denote as $\|\cdot\|_{\mathcal{L}(H^k)}$. Using Steps 2 and 3, and the fact that $\|f_0-t^* h_1- (1-t^*) h_2\|_{H^k} \le  c \inf f_0$,
\begin{align*}
    \|&D_f K(\epsilon, t^* h_1 + (1-t^*) h_2)\|_{\mathcal{L}(H^k)}\\
    &\le \|D_f K(\epsilon, t^* h_1 + (1-t^*) h_2) - D_f K(\epsilon, f_0)\|_{\mathcal{L}(H^k)} + \|D_f K(\epsilon, f_0)\|_{\mathcal{L}(H^k)}\\
    &\le C\|t^* h_1 + (1-t^*) h_2-f_0\|_{H^k} + \frac{\alpha}{\alpha+1}\left(1+C\frac{\sqrt{\epsilon}}{1+\alpha}\right)\\
    &\le C (t^* \|h_1 - f_0\|_{H^k} + (1-t^*)\|h_2 - f_0\|_{H^k}) + \frac{\alpha}{\alpha+1}\left(1+C\frac{\sqrt{\epsilon}}{1+\alpha}\right)
\end{align*}
and this yields the estimate~\eqref{eq:MVT}.

From~\eqref{eq:MVT}, for $\epsilon>0$ and $\delta>0$ small enough, the functional $f \in H^k(\T^m) \mapsto K(\epsilon,f) \in H^k(\T^m)$ is a contraction on the $H^k$ ball $${\mathcal B}(f_0, \delta)=\{f \in H^k(\T^m), \|f-f_0\|_{H^k} \le \delta\},$$ with a Lipschitz constant uniform in $\epsilon$. Indeed, let us define $\delta >0$ by:
\begin{equation}\label{eq:assump_delta}
\delta = \min \left(\frac{1}{3M_1(\alpha+1)} , c \inf f_0 \right),
\end{equation} 
where, we recall, $M_1$ is the constant appearing in~\eqref{eq:MVT}.
Let us assume that $\epsilon >0$ satisfies \begin{equation}\label{eq:assump_eps1}
\sqrt{\epsilon} \leq \min \left( 1, \frac{\alpha+1}{3 \alpha M_1}\right). 
\end{equation}
Then $K(\epsilon,\cdot)$ restricted to ${\mathcal B}(f_0, \delta)$ is a contraction with Lipschitz constant which satisfies:
\begin{align}
    \|K(\epsilon,\cdot)\|_{Lip} &\le M_1 \delta + \frac{\alpha}{\alpha+1}\left(1+M_1\frac{\sqrt{\epsilon}}{1+\alpha}\right)\nonumber\\
  &  \leq \frac{1}{3(\alpha+1)} + \frac{\alpha}{\alpha+1}+\frac{1}{3(\alpha+1)} =1-\frac{1}{3(\alpha
+1)}. \label{eq:estim_lipK}
\end{align}

Besides, let us introduce the constant $M_2$ such that (see Lemma~\ref{lem:product} and Lemma~\ref{lem:KlnKf-lnf}),
\begin{equation}
\label{eq:ecart_PF}
\|K(\epsilon,f_0)-f_0\|_{H^k}=\frac{\alpha}{\alpha+1}\left\|f_0 \Big(K^m_\epsilon\star\ln(K^m_\epsilon\star f_0)-\ln f_0 \Big)\right\|_{H^k}\leq \frac{M_2 {\epsilon}}{\alpha+1}.
\end{equation}
By choosing $\epsilon > 0$ which satisfies
\begin{equation}\label{eq:assump_eps2}
\epsilon \le \frac{\delta}{3 M_2}
\end{equation}
one has (from~\eqref{eq:estim_lipK}) 
\begin{align*}
 \|K(\epsilon,f_0)-f_0\|_{H^k}
  \le \frac{M_2{\epsilon}}{\alpha+1}
  \le \frac{\delta}{3(\alpha+1)} \le \delta \left( 1- \|K(\epsilon,\cdot)\|_{Lip} \right).\nonumber
\end{align*}

Therefore, by the first result of Lemma~\ref{lem:banach_fixedpoint}, the map $f\mapsto K(\epsilon,f)$ is a contraction from ${\mathcal B}(f_0, \delta)$ into ${\mathcal B}(f_0, \delta)$, with constant $1-\frac{1}{3(\alpha+1)}.$
Then by applying the second result of Lemma \ref{lem:banach_fixedpoint} with $X={\mathcal B}(f_0, \delta)$ with $\delta$ satisfying~\eqref{eq:assump_delta}, one gets that for any $\epsilon>0$ such that
\begin{equation}\label{eq:assump_eps}
\epsilon \le \min\left(1,\left(\frac{\alpha +1}{3\alpha M_1}\right)^2, \frac{1}{9 M_1M_2 (\alpha+1)} , \frac{c\inf f_0}{3 M_2} \right)
\end{equation}
(so that~\eqref{eq:assump_eps1} and~\eqref{eq:assump_eps2} are satisfied) one has the existence of a function $f_\epsilon$ which is a fixed point of $K(\epsilon,\cdot)$ and such that (using~\eqref{eq:estim_lipK} and~\eqref{eq:ecart_PF})
\begin{align*}
    \|f_\epsilon-f_0\|_{H^k} &\le  \frac{\|K(\epsilon,f_0)-f_0\|_{H^k}}{1-\|K(\epsilon,\cdot)\|_{Lip}}\le \frac{M_2{\epsilon}}{\alpha+1} 3(\alpha+1) = 3 M_2 \epsilon.
\end{align*}

To conclude, we take $C_2=3M_2$ and
 observe  that there exists $C_1\in[0,1]$ independent of $\alpha,\epsilon$ such that for all $\alpha \ge 0$,
$$\epsilon \in \left[0,\frac{C_1}{1+\alpha}\right] \Rightarrow \eqref{eq:assump_eps}.$$

\end{proof}

We are now in position to complete the proof of Theorem~\ref{thm:cv_infty}.
\begin{proof}[Proof of Theorem~\ref{thm:cv_infty}]
Again, we assume without loss of generality that $\beta=1$.

Let us recall that the first statement of Theorem~\ref{thm:cv_infty} about the existence and uniqueness of the minimizer $\rho^*_{\alpha,\epsilon}$ of $\mathcal F_{\alpha,\epsilon}$ has already been proven in Proposition~\ref{prop:rho*_uniqueness}. It thus only remains to prove the two other results.

 At first, we fix a $k\in\N$ such that $2k>m$. By Theorem~\ref{thm:implicit} (an $H^k$ estimate) and Lemma~\ref{lem:embedding} (Sobolev embedding on $\T^m$, from $H^k$  to $L^\infty$), there exist constants $C_1,C_2>0$ which are independent of $\alpha,\epsilon$, such that when $\epsilon\leq \frac{C_1}{\alpha+1}$, 
there exists  $f_\epsilon$, such that $F(\epsilon,f_\epsilon)=0$ (where $F$ is defined by~\eqref{eq:F}), and  $\|f_\epsilon-f_0\|_{\infty}\leq C_2{\epsilon}$.  Obviously $f_0=\exp(-\frac{A}{\alpha+1})$ is bounded from below by a positive constant $C_3$ independent of $\alpha$. Since replacing \(C_1\) with a smaller positive constant does not affect the conclusion of Theorem~\ref{thm:implicit}, we can assume that $C_1 \in (0, \frac{C_3}{2C_2}]$.

When  $\epsilon\leq \frac{C_1}{\alpha+1}\leq \frac{\inf f_0}{2C_2}$, it holds that $\|f_\epsilon-f_0\|_{\infty}\leq C_2{\epsilon}\leq \frac{\inf f_0}{2}\leq\frac{\sup f_0}{2}$.  
Therefore, $$\inf f_\epsilon\geq\frac{\inf f_0}{2}\geq\frac{\exp(-\max\{\sup A,0\})}{2} \text{ and } \sup f_{\epsilon}\leq\frac{3}{2}\sup f_0\leq\frac{3\exp(-\min\{\inf A,0\})}{2}.$$ 

Let us define  $g_\epsilon=\exp(-U+A)f_\epsilon$, and
\begin{align*}
M_1&=\frac{\exp(-\max\{\sup A,0\})}{2}\inf\exp(-U+A)\\
M_2& =\frac{3\exp(-\min\{\inf A,0\})}{2}\sup\exp(-U+A).    
\end{align*}
Then $g_\epsilon \in[M_1,M_2]$, and $\|g_\epsilon-g_0\|_\infty\leq C_2\|\exp(-U+A)\|_\infty\epsilon.$

We now express $\rho^*_{\alpha,\epsilon}$ in terms of $f_\epsilon$, first in the case $\epsilon \in (0, C_1/(\alpha+1)]$, and then in the case $\epsilon = 0$.
When $\epsilon\in(0,\frac{C_1}{\alpha+1}]$, we define the probability density $\hat\rho^*_{\alpha,\epsilon}\propto \exp(-U-\alpha K^m_\epsilon\star\ln(K^m_\epsilon\star f_\epsilon)).$ Then $\hat\rho^{*,1}_{\alpha,\epsilon}\propto \exp(-A-\alpha K^m_\epsilon\star\ln(K^m_\epsilon\star f_\epsilon))=f_\epsilon.$ Thus $\hat\rho^*_{\alpha,\epsilon}\propto \exp(-U-\alpha K^m_\epsilon\star\ln(K^m_\epsilon\star \hat\rho^{*,1}_{\alpha,\epsilon})).$ By Proposition~\ref{prop:rho*}, $\hat\rho^*_{\alpha,\epsilon}=\rho^*_{\alpha,\epsilon}.$ Therefore, $$\rho^*_{\alpha,\epsilon}\propto\exp(-U-\alpha K^m_\epsilon\star\ln(K^m_\epsilon\star \rho^{*,1}_{\alpha,\epsilon} ))\propto \exp(-U+A)\rho^{*,1}_{\alpha,\epsilon}\propto \exp(-U+A) f_\epsilon=g_\epsilon.$$
When $\epsilon=0$, by Proposition~\ref{prop:rho*},
$$\rho^*_\alpha\propto\exp\left(-U+\frac{\alpha}{\alpha+1}A\right)=g_0.$$

From the two former equations, it holds 
\begin{align*}
\|\rho^*_{\alpha,\epsilon}-\rho^*_{\alpha}\|_\infty=& \left\|\frac{g_\epsilon}{\int_{\T^d}g_\epsilon}-\frac{g_0}{\int_{\T^d}g_0}\right \|_\infty \\
\leq&\left\|\frac{g_\epsilon}{\int_{\T^d}g_\epsilon}-\frac{g_0}{\int_{\T^d}g_\epsilon}\right\|_\infty+\left\|\frac{g_0}{\int_{\T^d}g_\epsilon}-\frac{g_0}{\int_{\T^d}g_0}\right\|_\infty 
\\ \leq&\|g_\epsilon- g_0\|_\infty\frac{1}{M_1}+\|g_0\|_\infty\left|\frac{\int_{\T^d}g_\epsilon-\int_{\T^d}g_0}{\int_{\T^d}g_\epsilon\int_{\T^d}g_0}\right|\\\leq& \left(\frac{1}{M_1}+\frac{M_2}{M_1^2}\right)\|g_\epsilon-g_0\|_\infty\\ 
\leq&C_2 \left(\frac{1}{M_1}+\frac{M_2}{M_1^2}\right)\|\exp(-U+A)\|_\infty\epsilon. 
\end{align*}
This concludes the proof of~\eqref{eq:diffrhostar}.

The last statement of Theorem~\ref{thm:cv_infty} is then a simple consequence of the Proposition~\ref{prop:HS} (Holley-Stroock criterion), see Corollary~\ref{cor:ULSI} below.
\end{proof}

\begin{cor}\label{cor:ULSI}
Let us recall that $\rho^*_\alpha$ satisfies a log-Sobolev inequality with constant $D_\alpha$ (see Proposition~\ref{prop:rho_alpha}). There exist constants $c \in (0,1]$ and $C>0$ independent of $\alpha,\epsilon$, such that when $\epsilon\leq \frac{c}{1+\alpha}$, $\rho^*_{\alpha,\epsilon}$ satisfies the log-Sobolev inequality with constant $D_\alpha \exp(-C\epsilon)$. 
\end{cor}

\begin{proof}
We take $C_1,C_2$ as in the Theorem~\ref{thm:cv_infty} and define $C_3=\inf_{\alpha>0} \inf_{x\in \T^d} \rho^*_\alpha(x)$. By Corollary~\ref{cor:inf_rho}, $C_3>0.$ Let $c = \min\{1,C_1,C_3/(2C_2)\}$. Assume that $\epsilon\leq \frac{c}{\alpha+1}$, thus $\epsilon\in[0,1]$. Then, by~\eqref{eq:diffrhostar}, it holds $\|\rho^*_{\alpha,\epsilon}-\rho^*_{\alpha}\|_{\infty}\leq C_2{\epsilon}\leq\frac{cC_2}{\alpha+1}\leq \frac{\inf \rho^*_\alpha}{2}$. Therefore, $\inf \rho^*_{\alpha,\epsilon}\geq\frac{\inf \rho^*_\alpha}{2}\geq \frac{C_3}{2}$. Thus $\|\ln\rho^*_{\alpha,\epsilon}-\ln\rho^*_\alpha\|_\infty\leq \frac{2C_2\epsilon}{C_3}.$ Let us  define $C=\frac{4C_2}{C_3}.$ 
The result is then a consequence of Proposition~\ref{prop:HS}.
\end{proof}

\section{Proofs of the long time convergence results}\label{
sec:longtimeproof}

This section is organized as follows. Section~\ref{sec:longtime_overdamped_dependent} is devoted to the proof of Theorem~\ref{thm:CVoverdamped}. Section~\ref{sec:proofCVkinetic} is devoted to the proof of Theorem~\ref{thm:CVkinetic}. Finally, Section~\ref{sec:longtime_overdamped} is devoted to the proof of Theorem~\ref{th:CVoverdamped-sharp}.

\subsection{The overdamped case, with an $\epsilon$-dependent convergence rate}\label{sec:longtime_overdamped_dependent}

In this section, we prove Theorem~\ref{thm:CVoverdamped}, which will be a simple consequence of the following.

\begin{prop}\label{prop:chizat}
Suppose that $\epsilon\in(0,1]$. For any probability density $\rho$ on $\T^d$,
\begin{equation}\label{eq:sandwich}
\mathcal{H}(\rho \mid {\rho}^{*}_{\alpha,\epsilon}) \leq \beta\F_{\alpha,\epsilon}(\rho)-\beta\F_{\alpha,\epsilon}\left(\rho^{*}_{\alpha,\epsilon}\right) \leq \mathcal{H}(\rho \mid \Gamma_{\alpha,\epsilon}(\rho)).
\end{equation}

Let $\rho_t$ be the probability density of the adaptive biasing potential process~\eqref{eq:McKean-Voverdamped_convolution}. It holds: for all $t\ge 0$,
\begin{equation}\label{eq:ent_decay}
\frac{d}{dt} \mathcal F_{\alpha,\epsilon}(\rho_t) = - \frac{1}{\beta^2}\int_{\T^d} \left|\nabla \ln \frac{\rho_t}{\Gamma_{\alpha,\epsilon}(\rho_t)}\right|^2 \rho_t.
\end{equation}
\end{prop}
\begin{proof}
The sandwich inequality~\eqref{eq:sandwich} is a consequence of~\cite[Lemma 3.4]{chizat-2022}. Let us provide the proof of this inequality for the sake of completeness. Let us denote by$$G(\rho)=\int_{\T^d} \beta U \rho + \alpha  (K^m_\epsilon\star\rho)\ln (K^m_\epsilon\star\rho),$$
so that $\mathcal F_{\alpha,\epsilon}(\rho)=\frac{1}{\beta}\left(\int_{\T^d} \rho \ln \rho + G(\rho)\right)$. By standard convexity property, one has, for $\delta \in [0,1]$ and two probability densities $\mu$ and $\nu$
$$G((1-\delta)\mu+\delta \nu) \le (1-\delta) G(\mu) + \delta G(\nu)$$
so that for $\delta \in (0,1]$,
$$\frac{G(\mu +\delta (\nu-\mu))-G(\mu)}{\delta} \le  G(\nu) -  G(\mu).$$
By sending $\delta$ to $0$, and using the boundedness of $U$ and $K^m_\epsilon$ from Lemma~\ref{lem:estim_K}  to apply the dominated convergence theorem, one gets
$$\int_{\T^d} (\nu-\mu) (\beta U+ \alpha K^m_\epsilon\star\ln(K^m_\epsilon\star \mu))  \le  G(\nu) -  G(\mu).$$
By applying this inequality twice with $(\mu,\nu)=(\rho,\rho^*_{\alpha,\epsilon})$ and then $(\mu,\nu)=(\rho^*_{\alpha,\epsilon},\rho)$, one gets:
$$
\int_{\T^d} (\rho-\rho^*_{\alpha,\epsilon})(\beta U + \alpha  K^m_\epsilon\star\ln(K^m_\epsilon\star\rho^{*}_{\alpha,\epsilon})) \le G(\rho)-G(\rho^*_{\alpha,\epsilon}) \le \int_{\T^d} (\rho-\rho^*_{\alpha,\epsilon})(\beta U + \alpha  K^m_\epsilon\star\ln(K^m_\epsilon\star\rho))
$$
which can be rewritten as:
$$
- \int_{\T^d} \ln(\Gamma_{\alpha,\epsilon}(\rho^*_{\alpha,\epsilon})) (\rho-\rho^*_{\alpha,\epsilon})  \le G(\rho)-G(\rho^*_{\alpha,\epsilon}) \le - \int_{\T^d} \ln(\Gamma_{\alpha,\epsilon}(\rho) )(\rho-\rho^*_{\alpha,\epsilon}).
$$
By adding $\int_{\T^d} \rho \ln \rho - \int_{\T^d} \rho^*_{\alpha,\epsilon} \ln \rho^*_{\alpha,\epsilon}$, and using the fact that $\Gamma_{\alpha,\epsilon}(\rho^*_{\alpha,\epsilon})=\rho^*_{\alpha,\epsilon}$, one thus gets
$$ 
\mathcal{H}(\rho \mid {\rho}^{*}_{\alpha,\epsilon})  \leq \beta\F_{\alpha,\epsilon}(\rho)-\beta\F_{\alpha,\epsilon}\left(\rho^{*}_{\alpha,\epsilon}\right) \leq \mathcal{H}(\rho \mid \Gamma_{\alpha,\epsilon}(\rho)) - \mathcal{H}(\rho^{*}_{\alpha,\epsilon} \mid \Gamma_{\alpha,\epsilon}(\rho)).
$$
This concludes the proof of~\eqref{eq:sandwich} by using the fact that $\mathcal{H}(\rho^{*}_{\alpha,\epsilon} \mid \Gamma_{\alpha,\epsilon}(\rho))\ge 0$.

The second result~\eqref{eq:ent_decay} is obtained by a classical calculation. Let us provide some details, for the sake of completeness. Let us introduce
$$\gamma_{\alpha,\epsilon}(\rho)=\exp(-\beta U- \alpha K^m_\epsilon\star\ln (K^m_\epsilon\star\rho)),$$
so that $\Gamma_{\alpha,\epsilon}(\rho)=\gamma_{\alpha,\epsilon}(\rho)/\int_{\T^d}\gamma_{\alpha,\epsilon}$ and
$$\mathcal F_{\alpha,\epsilon}(\rho) =\frac{1}{\beta}\int_{\T^d} \rho \ln \left( \frac{\rho}{\gamma_{\alpha,\epsilon}(\rho)}\right).$$
Notice that $\rho_t$ satisfies the non-linear Fokker-Planck equation~\eqref{eq:PDEoverdamped_convolution}, which can be rewritten as
$$\partial_t \rho_t = \frac{1}{\beta}{\rm div} \left(\gamma_{\alpha,\epsilon}(\rho_t) \nabla \left(\frac{\rho_t}{\gamma_{\alpha,\epsilon}(\rho_t)}\right) \right).$$
Thus by noticing that $\int_{\T^d}\rho_t\partial_t \ln \rho_t=\int_{\T^d} \partial_t\rho_t=\frac{d}{dt}\int_{\T^d} \rho_t=0$,
\begin{align*}
    \frac{d}{dt} \mathcal F_{\alpha,\epsilon}(\rho_t)&=\frac{1}{\beta}\int_{\T^d} \partial_t \rho_t\ln \left( \frac{\rho_t}{\gamma_{\alpha,\epsilon}(\rho_t)}\right)-\frac{1}{\beta} \int_{\T^d}  \rho_t \partial_t \ln \gamma_{\alpha,\epsilon}(\rho_t) \\
    &= \frac{1}{\beta^2}\int_{\T^d} {\rm div} \left(\gamma_{\alpha,\epsilon}(\rho_t) \nabla \left(\frac{\rho_t}{\gamma_{\alpha,\epsilon}(\rho_t)}\right) \right) \ln \left( \frac{\rho_t}{\gamma_{\alpha,\epsilon}(\rho_t)}\right)+\frac\alpha\beta \int_{\T^d}  \rho_t 
    \partial_t (K^m_{\epsilon}\star\ln(K^m_\epsilon\star\rho_t))\\
    &= -\frac{1}{\beta^2}\int_{\T^d} \left|\nabla \ln \frac{\rho_t}{\gamma_{\alpha,\epsilon}(\rho_t)}\right|^2 \rho_t + \frac\alpha\beta\int_{\T^d}  K^m_\epsilon\star\rho_t \frac{\partial_t (K^m_\epsilon\star\rho_t)}{K^m_\epsilon\star\rho_t}\\
    &= -\frac{1}{\beta^2}\int_{\T^d} \left|\nabla \ln \frac{\rho_t}{\Gamma_{\alpha,\epsilon}(\rho_t)}\right|^2 \rho_t.
\end{align*}
\end{proof}

\begin{proof}[Proof of Theorem~\ref{thm:CVoverdamped}]
Recall that
$$\Gamma_{\alpha,\epsilon}(\rho)\propto \exp(-\beta U-\alpha K^m_\epsilon\star\ln (K^m_\epsilon\star \rho)).$$

Since $\inf K^m_\epsilon \le K^m_\epsilon\star \rho\le \sup K^m_\epsilon$ (since $\rho \ge0$ and $\int_{\T^d} \rho =1$), by~\eqref{eq:Linf_K} of Lemma~\ref{lem:estim_K}, 

$$\max \ln(K^m_\epsilon\star \rho)- \min \ln(K^m_\epsilon\star \rho)\leq C+\frac{m}{8\epsilon},$$
for some constant $C$. Then, for any probability density $\rho$, by Proposition~\ref{prop:HS}, $\Gamma_{\alpha,\epsilon}(\rho)$ satisfies a log-Sobolev inequality with constant $D \exp(-C\alpha-\frac{m\alpha}{8\epsilon})=\frac{c_{\alpha,\epsilon}\beta}{2}$, where $c_{\alpha,\epsilon}$ is defined in the statement of Theorem~\ref{thm:CVoverdamped}.

Writing $f(t)=\F_{\alpha,\epsilon}(\rho_t)-\F_{\alpha,\epsilon}\left(\rho^{*}_{\alpha,\epsilon}\right)$, we obtain the result by  using~\eqref{eq:ent_decay} and~\eqref{eq:sandwich} from Proposition~\ref{prop:chizat}
\[f'(t) = -\frac{1}{\beta^2} \int_{\T^d} \left|\nabla \ln \frac{\rho_t}{\Gamma_{\alpha,\epsilon}(\rho_t)}\right|^2 \rho_t \leq - \frac{c_{\alpha,\epsilon}}{\beta} \mathcal{H}(\rho_t \mid \Gamma_{\alpha,\epsilon}(\rho_t)) \leq  - c_{\alpha,\epsilon} f(t)\,. \]    
The conclusion follows by using again~\eqref{eq:sandwich}.
\end{proof}

\begin{rem}\label{rem:adapt}
Let us comment on the possible adaptation of the latter proof for the non-regularized dynamics~\eqref{eq:McKean-Voverdamped}. For the free energy $\F_\alpha$, the sandwich inequality stated in Proposition~\ref{prop:chizat} also holds when replacing the map $\Gamma_{\alpha,\epsilon}$ by the function $\Gamma_{\alpha}$ introduced in Remark~\ref{rem:Gamma}, at least for probability densities $\rho$ such that $\Gamma_\alpha(\rho)$ is well defined.
In view of the proof of Theorem~\ref{thm:CVoverdamped}, if we could prove a uniform log-Sobolev inequality for $\Gamma_\alpha(\rho_t)$ (for all positive time), we could get an exponential convergence to zero for $\F_\alpha(\rho_t)-\F_\alpha\left(\rho^{*}_\alpha\right)$, hence also for $\mathcal{H}(\rho_t \mid {\rho}^{*}_\alpha)$.

However, since $\Gamma_\alpha(\rho_t)\propto\exp(-U-\alpha \ln\rho^1_t)$, a uniform log-Sobolev inequality for $\Gamma_\alpha(\rho_t)$ would require for example $-\ln\rho_t^1$ to be bounded uniformly in time, a property which does not seem obvious to obtain (such bounds are established in \cite{10.1214/24-EJP1217,M33} for some non-linear McKean-Vlasov equations but the methods in these works do not apply to the present non-linearity).
\end{rem}

\subsection{The underdamped case, with an $\epsilon$-dependent convergence rate}\label{sec:proofCVkinetic}
In this section, we consider the long time behaviour of the regularized adaptive biasing potential method for the underdamped Langevin dynamics introduced in~\eqref{eq:PDEkinetic_convolution}.




In a similar way as we defined $\Gamma_{\alpha,\epsilon}$ by \eqref{eq:Gamma} in the overdamped case, 
let us introduce in  the kinetic case the map $\tilde{\Gamma}_{\alpha,\epsilon}$ defined as follows  (for $\epsilon\in(0,1]$ and $\alpha>0$): for a given probability density $\nu(x,v)$ on $\T^d\times\R^d$,  $\tilde\Gamma_{\alpha,\epsilon}(\nu)$ is the probability density on $\T^d\times\R^d$ such that 
$$\tilde{\Gamma}_{\alpha,\epsilon}(\nu)\propto \exp\left(-\beta  H - \alpha K^m_\epsilon \star \ln(K^m_\epsilon \star \nu^{x,1})\right)$$
where, we recall $H(x,v)=U+\frac{|v|^2}{2}$.
Equivalently, $\tilde{\Gamma}_{\alpha,\epsilon}(\nu) =\Gamma_{\alpha,\epsilon}(\nu^x) \otimes \mathcal N(0, \beta^{-1} I_d)$.

The next proposition is obtained following the same reasoning as for the proofs of Proposition~\ref{prop:rho*_uniqueness} and Proposition~\ref{prop:rho*}.
\begin{prop}\label{prop:minimizer_underdamped}For $\epsilon\in(0,1]$, the functional $\tilde{\mathcal{F}}_{\alpha,\epsilon}$
admits a unique minimizer $\tilde{\rho}^*_{\alpha,\epsilon}$, which is a probability density on $\T^d\times \R^d$. Moreover $\nu^*_{\alpha,\epsilon} = \rho^*_{\alpha,\epsilon} \otimes \mathcal N(0, \beta^{-1} I_d)$, so that $\tilde \Gamma_{\alpha,\epsilon}( \nu^*_{\alpha,\epsilon})= \nu^*_{\alpha,\epsilon}$.
\end{prop}
\begin{proof}
The existence and uniqueness of a minimizer $\nu^*_{\alpha,\epsilon}$ to $\tilde{\mathcal{F}}_{\alpha,\epsilon}$ is proven following the same reasoning as in the proof of Proposition~\ref{prop:rho*_uniqueness}. 
Let us now prove that $\nu^*_{\alpha,\epsilon} = \rho^*_{\alpha,\epsilon} \otimes \mathcal N(0, \beta^{-1} I_d)$.

Let us recall (see~\eqref{eq:def_gibbs_kin}) that $\nu^*(x,v)=\nu^{*,x}(x)\nu^{*,v}(v)$ where
$$\nu^{*,x}(x)=\frac{\exp(-\beta U(x))}{\int_{\T^d}\exp(-\beta U(x))dx},\qquad \nu^{*,v}(v)=\frac{\exp(-\beta\frac{|v|^2}{2})}{\int_{\R^d}\exp(-\beta\frac{|v|^2}{2})dv}.$$

Let  $\nu$ be a probability density on $\T^d\times\R^d$. Let us define the probability density on $\T^d\times\R^d$: $\tilde{\nu}(x,v)=\nu^{x}(x)\nu^{*,v}(v)$. Simple calculations lead to
\begin{align}
\mH(\tilde{\nu}|\nu^{*})=&\int_{\T^d\times\R^d}\nu^x\nu^{*,v}\ln(\nu^x\nu^{*,v})-\int_{\T^d\times\R^d}\nu^x\nu^{*,v}\ln(\nu^{*,x}\nu^{*,v})\nonumber\\
=&\int_{\T^d}\nu^x\ln\nu^x+\int_{\R^d}\nu^{*,v}\ln\nu^{*,v}-\int_{\T^d}\nu^x\ln\nu^{*,x}-\int_{\R^d}\nu^{*,v}\ln\nu^{*,v}\nonumber \\
=&\int_{\T^d}\nu^x\ln\nu^x-\int_{\T^d}\nu^x\ln\nu^{*,x}\nonumber\\
=&\mH(\nu^x|\nu^{*,x}). \label{eq:res1}
\end{align}
By the Donsker-Varadhan variational formula of relative entropy (e.g.~\cite[Lemma 1.4.3]{dupuis1997weak}),  denoting by $C_b(E)$   the set of bounded continuous functions on $E=\T^d\times\R^d$ or $\T^d$, we have that
\begin{align}
\mH(\nu|\nu^*)=&\sup_{g\in C_b(\T^d\times\R^d)}\left\{\int_{\T^d\times\R^d}g(x,v)\nu(x,v)dxdv-\ln\left( \int_{\T^d\times\R^d}\exp(g(x,v))\nu^*(x,v)dxdv\right)\right\} \nonumber \\ 
\geq& \sup_{g\in C_b(\T^d)}\left\{\int_{\T^d\times\R^d}g(x)\nu(x,v)dxdv-\ln\left( \int_{\T^d\times\R^d}\exp(g(x))\nu^*(x,v)dxdv\right)\right\}\nonumber \\
=& \sup_{g\in C_b(\T^d)}\left\{\int_{\T^d}g(x)\nu^x(x)dx-\ln\left( \int_{\T^d}\exp(g(x))\nu^*(x)dx\right)\right\}\nonumber \\
=&\mH(\nu^x|\nu^{*,x}). \label{eq:res2}
\end{align}
This inequality is also a direct consequence of the extensivity property of the entropy, see for example~\cite[Lemma 2.1]{Lelievre_twoscale}.
By combining~\eqref{eq:res1} and~\eqref{eq:res2}, we thus have proven that
$$\mH(\nu|\nu^{*}) \ge  \mH(\tilde{\nu}|\nu^{*}).$$
This leads to  $\tilde{\F}_{\alpha,\epsilon}(\nu)\geq \tilde{\F}_{\alpha,\epsilon}(\tilde{\nu}).$
Besides, by direct calculation, $\tilde{\F}_{\alpha,\epsilon}(\tilde{\nu})={\F}_{\alpha,\epsilon}({\nu}^x)-\frac{1}{\beta}\ln \int_{\R^d}\exp(-\beta\frac{|v|^2}{2})dv$. 
Since $\rho^*_{\alpha,\epsilon}$ is the minimizer of $\F_{\alpha,\epsilon}$, one thus gets
$$\tilde{\F}_{\alpha,\epsilon}(\nu)\geq {\F}_{\alpha,\epsilon}(\rho^*_{\alpha,\epsilon})-\frac{1}{\beta}\ln \int_{\R^d}\exp\left(-\beta\frac{|v|^2}{2}\right)dv.$$
It is easy to check that the lower bound is obtained by choosing $\nu(x,v) = \rho^*_{\alpha,\epsilon}(x) \nu^{*,v}(v)$. This shows that $\nu^*_{\alpha,\epsilon}$ is the (unique) minimizer of $\tilde{\F}_{\alpha,\epsilon}$. 

The fact that $\tilde \Gamma_{\alpha,\epsilon}( \nu^*_{\alpha,\epsilon})= \nu^*_{\alpha,\epsilon}$ easily follows by using the analytical expression of $\nu^*_{\alpha,\epsilon}$, and the fact that $\rho^*_{\alpha,\epsilon}$ is a fixed point of $\Gamma_{\alpha,\epsilon}$ (see Proposition~\ref{eq:Gamma}):
$$\tilde \Gamma_{\alpha,\epsilon}( \nu^*_{\alpha,\epsilon}) \propto \Gamma_{\alpha,\epsilon}(\rho^*_{\alpha,\epsilon}) \nu^{*,v} = \rho^*_{\alpha,\epsilon}\nu^{*,v}.$$
\end{proof}

Our objective is now to prove Theorem~\ref{thm:CVkinetic}, namely that $\nu_t$ solution to~\eqref{eq:PDEkinetic_convolution} converges to $\nu^*_{\alpha,\epsilon}$ in the sense of relative entropy when $t\rightarrow\infty$.
Compared to the overdamped dynamics~\eqref{eq:PDEoverdamped_convolution}, there is an additional difficulty related to the fact that time derivative of the entropy only gives a partial Fisher information. We circumvent this difficulty by applying  \cite[Theorem 2.1]{chen-lin-ren-wang-2024}, which relies on a modified free energy.

Notice that \cite[Theorem 2.1]{chen-lin-ren-wang-2024} is written in the case when the position variable $x$ is in $\R^d$. It is easy to check that the proof applies verbatim for $x \in \T^d$, and this is what we use in the following.

\begin{proof}
[Proof of Theorem~\ref{thm:CVkinetic}]
First of all, by \cite[Lemma 4.2]{chen-lin-ren-wang-2024}, any probability measure $\nu\in\mathcal{P}_2(\T^d \times \R^d)$  satisfies a sandwich inequality similar to the inequality~\eqref{eq:sandwich} stated in Proposition~\ref{prop:chizat}:
$$ 
\mathcal{H}(\nu \mid \nu^*_{\alpha,\epsilon}) \leq \beta\mathcal{\tilde{F}}_{\alpha,\epsilon}\left(\nu\right)-\beta\mathcal{\tilde{F}}_{\alpha,\epsilon}\left(\nu^*_{\alpha,\epsilon}\right) \leq \mathcal{H}(\nu \mid \tilde{\Gamma}_{\alpha,\epsilon}(\nu)).
$$
 To prove the exponential convergence, we will use \cite[Theorem 2.1 and Proposition 5.5]{chen-lin-ren-wang-2024}. These two results rely on three conditions, that we will check in a second step.

\medskip
\noindent
{\bf Step 1}: Let us first show that \cite[Theorem 2.1 and Proposition 5.5]{chen-lin-ren-wang-2024} yield the stated results.

By \cite[Proposition 5.5]{chen-lin-ren-wang-2024}, for $t_0>0$,
$$\mathcal{I}\left(\nu_{t_0} \mid \tilde{\Gamma}_{\alpha,\epsilon}(\nu_{t_0})\right)\leq \frac{C''_{\alpha,\epsilon}\beta^2}{t_0^3}\left(\mathcal{\tilde{F}}_{\alpha,\epsilon}\left(\nu_{0}\right)-\mathcal{\tilde{F}}_{\alpha,\epsilon}\left(\nu^*_{\alpha,\epsilon}\right)\right),$$
where $C''_{\alpha,\epsilon}$ is a constant\footnote{The constant $C''_{\alpha,\epsilon}$ actually depends on $\beta$, which is a fixed constant in all this work. We introduce somewhat artificially the factor $\beta^2$ in the inequality so that, at the end of Step 1, $C_{\alpha,\epsilon}$ is easily expressed in terms of $C''_{\alpha,\epsilon}$.}. In particular, $\mathcal{I}\left(\nu_{t_0} \mid \tilde{\Gamma}_{\alpha,\epsilon}(\nu_{t_0})\right)< \infty$.

Thus, \cite[Theorem 2.1]{chen-lin-ren-wang-2024} shows that for all $0< t_0\leq t$, (applying the cited result with an initial condition at time $t_0>0$ instead of $0$),
$$
\beta\mathcal{\tilde{F}}_{\alpha,\epsilon}\left(\nu_{t}\right)-\beta\mathcal{\tilde{F}}_{\alpha,\epsilon}\left(\nu^*_{\alpha,\epsilon}\right) \leq\left(\beta\mathcal{\tilde{F}}_{\alpha,\epsilon}\left(\nu_{t_0}\right)-\beta\mathcal{\tilde{F}}_{\alpha,\epsilon}\left(\nu^*_{\alpha,\epsilon}\right)+C'_{\alpha,\epsilon}\frac{1}{\beta} \mathcal{I}\left(\nu_{t_0} \mid \tilde{\Gamma}_{\alpha,\epsilon}(\nu_{t_0})\right)\right) e^{-\kappa_{\alpha,\epsilon} (t-t_0)},
$$where $C'_{\alpha,\epsilon},\kappa_{\alpha,\epsilon}$ are constants.
Besides, by a similar calculation as the one done to prove Proposition~\ref{prop:chizat},
$$\frac{d}{dt} \tilde{\mathcal F}_{\alpha,\epsilon}(\nu_t) = -\frac{1}{\beta^2} \int_{\T^d\times \R^d} \left|\nabla_v \ln \frac{\nu_t}{\tilde{\Gamma}_{\alpha,\epsilon}(\nu_t)}\right|^2 \nu_t\leq 0,$$thus $\mathcal{\tilde{F}}_{\alpha,\epsilon}\left(\nu_{t_0}\right)-\mathcal{\tilde{F}}_{\alpha,\epsilon}\left(\nu^*_{\alpha,\epsilon}\right)\leq \mathcal{\tilde{F}}_{\alpha,\epsilon}\left(\nu_{0}\right)-\mathcal{\tilde{F}}_{\alpha,\epsilon}\left(\nu^*_{\alpha,\epsilon}\right).$

The conclusions thus follows by choosing $t_0=1$ and $C_{\alpha,\epsilon}=(1+C'_{\alpha,\epsilon}C''_{\alpha,\epsilon})\exp(\kappa_{\alpha,\epsilon}).$

\medskip
\noindent
{\bf Step 2}: Let us now check that the conditions required to apply the results of \cite{chen-lin-ren-wang-2024} are satisfied.

First, notice that $\nu_{t_0}$ has finite second moments, finite entropy and finite Fisher information. Besides, thanks to Proposition~\ref{prop:minimizer_underdamped}, $\nu^*_{\alpha,\epsilon} = \rho^*_{\alpha,\epsilon} \otimes \mathcal N(0, \beta^{-1} I_d)$ has finite second moments, finite
exponential moments, finite entropy and finite Fisher information, and satisfies $\tilde{\Gamma}_{\alpha,\epsilon}(\nu^*_{\alpha,\epsilon})=\nu^*_{\alpha,\epsilon}$, which are conditions required by \cite[Theorem 2.1 and Proposition 5.5]{chen-lin-ren-wang-2024}.  

Let us now check that the three conditions (conditions (2.1), (2.2) and (2.3)) of \cite[Theorem 2.1, Lemma 4.2, Proposition 5.5]{chen-lin-ren-wang-2024} are satisfied in our context. These conditions are expressed in terms of a functional $F_{\alpha,\epsilon}$ defined as follows: for any probability density $\nu^x$ on~$\T^d$, $$F_{\alpha,\epsilon}(\nu^x)=\beta \int_{\T^d} U\nu^x+\alpha \int_{\T^m} (K^m_\epsilon\star\nu^{x,1}) \ln(K^m_\epsilon\star\nu^{x,1}).$$
To make a link with our notation, let us notice that
$\mathcal F_{\alpha,\epsilon}(\nu^x)=\frac 1 \beta \left( F_{\alpha,\epsilon}(\nu^x) + \int_{\T^d} \nu^x \ln \nu^x \right)$ and $\Gamma_{\alpha,\epsilon}(\nu^x) \varpropto \exp \left(- \frac{\delta F_{\alpha,\epsilon}}{\delta \nu^x}(\nu^x, \cdot)\right)$, where here and in the following, $\frac{\delta F_{\alpha,\epsilon}}{\delta \nu^x}$ denotes the functional derivative of $F_{\alpha,\epsilon}$ with respect to $\nu^x$. 
The three conditions of \cite{chen-lin-ren-wang-2024} are the following: 
\begin{enumerate}
    \item The map $\nu^x \mapsto  F_{\alpha,\epsilon}(\nu^x)$ is convex, i.e. for any probability densities $\nu^x_1,\nu^x_2$ on $\T^d$, for any $\lambda\in [0,1]$, $$\lambda F_{\alpha,\epsilon}(\nu^x_1)+(1-\lambda)F_{\alpha,\epsilon}(\nu^x_2)\geq F_{\alpha,\epsilon}(\lambda\nu^x_1+(1-\lambda)\nu^x_2).$$ 
    \item The family of probability densities $ ({\Gamma}_{\alpha,\epsilon}(\nu^x) )_{\nu^x\in \mathcal{P}(\T^d)}$ satisfies a uniform log-Sobolev inequality. 
    \item The function $\nabla \frac{\delta F_{\alpha,\epsilon}}{\delta\nu^x}=\beta \nabla  U+\alpha \nabla K^m_\epsilon\star\ln (K^m_\epsilon\star\nu^{x,1})$ (where here and in the following, the derivatives $\nabla$ are with respect to the position variable $x$) is Lipschitz in the two variables  $(\nu^x,x)$, in the following sense: there exists constants $L_x,L_\nu$, such that for any pairs $(\nu^x,x),(\nu'^x,x')$,     
    $$\left|\nabla \frac{\delta F_{\alpha,\epsilon}}{\delta\nu^x}(\nu^x,x)-\nabla \frac{\delta F_{\alpha,\epsilon}}{\delta\nu^x}(\nu'^x,x')\right|\leq L_x|x-x'|+L_\nu W_1(\nu^x,\nu'^x),$$
    where $W_1$ is the Wasserstein 1-distance defined as follows: for $\mu,\nu\in\mathcal{P}(\T^d)$,
$$W_1(\mu,\nu)=\sup_{f\in C(\T^d),\|f\|_{Lip}\leq 1}\int_{\T^d} fd\mu-fd\nu.$$
\end{enumerate}

The first condition follows from the convexity of $x\mapsto x\ln x$.

To verify the second condition, by Proposition~\ref{prop:HS}, it is sufficient to verify that $\alpha K^m_\epsilon\star\ln(K^m_\epsilon\star \nu^{x,1})$ is  bounded independently of $\nu$. This is ensured by the fact that $K^m_\epsilon \in [c_{\epsilon},C_\epsilon]$ for some constants $C_\epsilon,c_\epsilon>0$, see~\eqref{eq:Linf_K}.

For the third condition, let us first note that $\nabla U(x)$ is Lipschitz continuous in $x$, since $\nabla^2 U$ is continuous in $\T^d$, thus bounded. Besides, for any $\nu$,
\[
\| \nabla \left( \nabla K^m_\epsilon\star\ln (K^m_\epsilon\star\nu^{x,1})\right)\|_\infty  = \|  \nabla^2 K^m_\epsilon\star\ln (K^m_\epsilon\star\nu^{x,1}) \|_\infty  
\leq   \|\nabla^2 K^m_\epsilon\|_\infty \ln\max\left( C_\epsilon,\frac{1}{c_\epsilon}\right) \,,
\]
which concludes the proof for the Lipschitz continuity in $x$ (notice that the upper bound is independent of $\nu$). For the Lipschitz continuity in $\nu$, we bound
\begin{align*}
    \left\|\nabla \frac{\delta F_{\alpha,\epsilon}}{\delta\nu^x}(\nu^x, \cdot)-\nabla \frac{\delta F_{\alpha,\epsilon}}{\delta\nu^x}(\nu'^x,\cdot)\right\|_\infty &\leq \alpha \|\nabla K^m_\epsilon\|_\infty \|\ln (K^m_\epsilon\star\nu^{x,1})-\ln (K^m_\epsilon\star\nu'^{x,1})\|_\infty \\
    &\leq \alpha \|\nabla K^m_\epsilon\|_\infty \frac{\|K^m_\epsilon\star\nu^{x,1}-K^m_\epsilon\star\nu'^{x,1}\|_\infty}{c_\epsilon}\\
    &  \leq \alpha \|\nabla K^m_\epsilon\|_\infty \frac{\|\nabla K^m_\epsilon\|_\infty W_1(\nu^{x,1},\nu'^{x,1})}{c_\epsilon}\\
    & \leq \alpha \|\nabla K^m_\epsilon\|_\infty \frac{\|\nabla K^m_\epsilon\|_\infty W_1(\nu^x,\nu'^x)}{c_\epsilon},
\end{align*} which concludes the proof.
\end{proof}

\subsection{The overdamped case, with an $\epsilon$-independent convergence rate}\label{sec:longtime_overdamped}

This section is devoted to the proof of Theorem~\ref{th:CVoverdamped-sharp}. Let us start with a technical lemma.

\begin{lem}\label{lem:LSI_gamma} Suppose that $\rho^*_{\alpha,\epsilon}$ satisfies the log-Sobolev inequality with constant $D_{\alpha,\epsilon}$. Then there exist  constants $C,c>0$ independent of $\alpha$ and $\epsilon$ such that when $\epsilon\in(0,\frac{c}{\alpha+1}]$, for any probability density $\rho$ on $\T^d$ such that $\mathcal{H}(\rho|\rho^*_{\alpha,\epsilon})\leq c\epsilon^m,$ $\Gamma_{\alpha,\epsilon}(\rho)$ satisfies the log-Sobolev inequality with constant $D_{\alpha,\epsilon}\exp\left(-{C\alpha}\epsilon^{-m/2}\sqrt{\mathcal{H}(\rho|\rho^*_{\alpha,\epsilon})}\right)$.
\end{lem}
\begin{proof}
First of all, referring to the proof of Corollary~\ref{cor:ULSI}, there exists a constant $c_1 \in (0,1]$ independent of $\alpha,\epsilon$ such that, for all $\epsilon> 0$ such that $$\epsilon\leq \frac{c_1}{\alpha+1}$$
it holds
$$\inf_{\T^m} \rho^{*,1}_{\alpha,\epsilon}\geq\inf_{\T^d} \rho^*_{\alpha,\epsilon}\geq \frac{\inf_{\T^d} \rho^*_\alpha}{2}\geq\frac{C_3}{2}$$
where $C_3=\inf_{\alpha>0} \inf_{x\in \T^d} \rho^*_\alpha(x)>0$.

Let now $\rho$ be a probability density on $\T^d$.
We would like to apply Proposition~\ref{prop:HS} to $\Gamma_{\alpha,\epsilon}(\rho)$. Since $\Gamma_{\alpha,\epsilon}(\rho)\propto \exp(-\beta U-\alpha K^m_\epsilon\star\ln(K^m_\epsilon\star \rho))\propto\rho^*_{\alpha,\epsilon}\exp(-\alpha K^m_\epsilon\star\ln(K^m_\epsilon\star\rho)+\alpha K^m_\epsilon\star\ln(K^m_\epsilon\star\rho^{*}_{\alpha,\epsilon}))$, and $\|K^m_\epsilon\star (\ln(K^m_\epsilon\star\rho^1)-\ln(K^m_\epsilon\star\rho^{*,1}_{\alpha,\epsilon}))\|_\infty\leq \|\ln(K^m_\epsilon\star\rho^1)-\ln(K^m_\epsilon\star\rho^{*,1}_{\alpha,\epsilon})\|_\infty,$ we only need to estimate the $L^\infty$ norm of $\ln(K^m_\epsilon\star\rho^1)-\ln(K^m_\epsilon\star\rho^{*,1}_{\alpha,\epsilon}).$

Using Pinsker’s inequality (see for example~\cite{ane-2000}) and the fact that, thanks to \eqref{eq:Linf_K}, $\|K^m_\epsilon\|_\infty\leq C'\epsilon^{-m/2}$ for some positive constant $C'$ independent of $\alpha,\epsilon$, 
$$ \|K^m_\epsilon\star (\rho^1-\rho^{*,1}_{\alpha,\epsilon})\|_\infty \leq C'\epsilon^{-m/2}\|\rho-\rho^*_{\alpha,\epsilon}\|_{TV}\leq C'\epsilon^{-m/2}\sqrt{2\mathcal{H}(\rho|\rho^*_{\alpha,\epsilon})},
 $$
where $\|\cdot\|_{TV}$ is the total variation between probability measure defined as
$$\|\mu-\nu\|_{TV}=\sup_{\|f\|_\infty\leq 1}\int_E fd\mu-fd\nu.$$
As a consequence, when $\rho$ is such that $\frac{C_3}{2}\geq 2C'\epsilon^{-m/2}\sqrt{2\mathcal{H}(\rho|\rho^*_{\alpha,\epsilon})}$, or equivalently $\mathcal{H}(\rho|\rho^*_{\alpha,\epsilon})\leq \epsilon^m\frac{C^2_3}{32C'^2}$, we have that
$$K^m_\epsilon\star\rho^1\geq K^m_\epsilon\star\rho^{*,1}_{\alpha,\epsilon}-C'\epsilon^{-m/2}\sqrt{2\mathcal{H}(\rho|\rho^*_{\alpha,\epsilon})}\geq \inf \rho^{*,1}_{\alpha,\epsilon}-\frac{C_3}{4}\geq\frac12\inf \rho^{*,1}_{\alpha,\epsilon}.$$
Hence, using the fact that $\inf \rho_{\alpha,\epsilon}^{*,1} \geq \frac{C_3}{2}$,
$$|\ln(K^m_\epsilon\star\rho^1)-\ln(K^m_\epsilon\star\rho^{*,1}_{\alpha,\epsilon})|\leq \frac{C'\epsilon^{-m/2}\sqrt{2\mathcal{H}(\rho|\rho^*_{\alpha,\epsilon})}}{\frac12 \inf (\rho^{*,1}_{\alpha,\epsilon})}\leq  \frac{4\sqrt{2}C'}{C_3}\epsilon^{-m/2}\sqrt{\mathcal{H}(\rho|\rho^*_{\alpha,\epsilon})}.$$

As a conclusion,  when $\rho$ is such that $\mathcal{H}(\rho|\rho^*_{\alpha,\epsilon})\leq \epsilon^m\frac{C^2_3}{32C'^2},$
\begin{align*}
   \|\alpha K^m_\epsilon\star\ln(K^m_\epsilon\star \rho^1)-\alpha K^m_\epsilon\star\ln(K^m_\epsilon\star \rho^{*,1}_{\alpha,\epsilon})\|_\infty & \leq \alpha\|\ln(K^m_\epsilon\star\rho^1)-\ln(K^m_\epsilon\star\rho^{*,1}_{\alpha,\epsilon})\|_\infty \\
   & \leq \frac{4\sqrt{2}\alpha C'}{C_3}\epsilon^{-m/2}\sqrt{\mathcal{H}(\rho|\rho^*_{\alpha,\epsilon})}. 
\end{align*}
We take $c=\min\{c_1,\frac{C^2_2}{32C'^2}\}$  and $C=\frac{8\sqrt{2}C'}{C_3}.$
Then, we conclude by Proposition~\ref{prop:HS}.
\end{proof}

 Notice that when $\epsilon \le \frac{c}{\alpha+1}$, we have that $\epsilon\leq 1$, and thus all the estimates on $K_\epsilon^m$ requiring $\epsilon\leq 1$ proven in Appendix~\ref{sec:ap}  can be used.

\begin{proof}[Proof of Theorem~\ref{th:CVoverdamped-sharp}] For $\epsilon\in(0,1]$,
we define $f(t)=\F_{\alpha,\epsilon}(\rho_t)-\F_{\alpha,\epsilon}\left(\rho^{*}_{\alpha,\epsilon}\right)$. Denote by $D_t$ a log-Sobolev inequality  constant of~$\Gamma_{\alpha,\epsilon}(\rho_t)$. From Proposition~\ref{prop:chizat},
$$f'(t)=\frac{d}{dt} \mathcal F_{\alpha,\epsilon}(\rho_t) = -\frac{1}{\beta^2} \int_{\T^d} \left| \nabla \ln \frac{\rho_t}{\Gamma_{\alpha,\epsilon}(\rho_t)}\right|^2 \rho_t\leq-\frac{2}{\beta^2}D_t\mathcal{H}(\rho_t|\Gamma_{\alpha,\epsilon}(\rho_t))\leq -\frac{2}{\beta}D_t f(t).$$
As a consequence, using that we can suppose that $2D_t \geq \beta c_{\alpha,\epsilon}$ for some $c_{\alpha,\epsilon} >0$ as we saw in the proof of Theorem~\ref{thm:CVoverdamped}, for all $t\geqslant 0$,
\[\mathcal{H}(\rho_t \mid \rho^*_{\alpha,\epsilon})\leq\beta f(t) \leq \beta e^{-\int_0^t \frac{2D_s}{\beta} d s } f(0) \leq \beta e^{-t c_{\alpha,\epsilon}} f(0)\,.\]

Now, let $C,c$ be as in Lemma~\ref{lem:LSI_gamma}, and let $t_0 = \frac{1}{c_{\alpha,\epsilon}}\ln\left(\frac{\beta f(0)}{c \epsilon^m} \right)$, so that for all $t \ge t_0$, $\mathcal{H}(\rho_{t}|\rho^*_{\alpha,\epsilon})\leq c\epsilon^m$. Applying Lemma~\ref{lem:LSI_gamma}, we get that,  for all $t\geq t_0$, we can suppose
\[D_t \geq D_{\alpha,\epsilon}\exp\left(-C\alpha\epsilon^{-m/2}\sqrt{\mathcal{H}(\rho_t|\rho^*_{\alpha,\epsilon})}\right) \geq D_{\alpha,\epsilon}\exp\left(-C\alpha\epsilon^{-m/2}\sqrt{\beta e^{-t c_{\alpha,\epsilon}}f(0)}\right). \]
Writing $C'_\epsilon=C\alpha \epsilon^{-m/2} \sqrt{\beta f(0)}$ and using that $e^{-C'_\epsilon x} \geq 1 - C'_\epsilon x$ for all $x \in \R$, we get that, for all $t\geq t_0$,
\[\int_{t_0}^t D_s d s \geq D_{\alpha,\epsilon} (t-t_0) - D_{\alpha,\epsilon} C'_\epsilon \int_{t_0}^t e^{-\frac{s c_{\alpha,\epsilon}}{2}} ds \geq D_{\alpha,\epsilon} (t-t_0) - \frac{ 2D_{\alpha,\epsilon} C'_\epsilon}{c_{\alpha,\epsilon}}=D_{\alpha,\epsilon}t + M_1   \]
for $M_1=- D_{\alpha,\epsilon} t_0 - \frac{ 2D_{\alpha,\epsilon} C'_\epsilon}{c_{\alpha,\epsilon}}$.
Thus, for all $t \ge t_0$, 
$$    f(t)\le f(0) e^{-\frac{2}{\beta}\int_{0}^t D_s d s } \le f(0)e^{-\frac{2}{\beta}\int_{0}^{t_0} D_s d s } e^{-\frac{2}{\beta}D_{\alpha,\epsilon}t} e^{-\frac{2}{\beta}M_1}\le f(0) e^{-c_{\alpha,\epsilon} t_0-\frac{2}{\beta}M_1} e^{-\frac{2}{\beta}D_{\alpha,\epsilon}t}.
$$
For $t\leq t_0$, we simply bound
\[f(t) \leq   f(0) \leq e^{-\frac{2}{\beta} D_{\alpha,\epsilon} t}  e^{\frac{2}{\beta} D_{\alpha,\epsilon} t_0} f(0)\,, \]
which concludes the exponential convergence stated in Theorem~\ref{th:CVoverdamped-sharp}.


\end{proof}

\begin{rem}\label{rem:sharpkinetic}

    Similarly to Theorem~\ref{thm:CVoverdamped} in the overdamped case, the convergence rate in Theorem~\ref{thm:CVkinetic} is {\em a priori} not bounded away from $0$ as $\epsilon \to 0$. One could expect a result similar to Theorem~\ref{th:CVoverdamped-sharp}, namely that in the long time the measure $\nu_t$ becomes smooth and close to $\nu_{\alpha,\epsilon}^*$, so that the exponential rate of convergence can be made independent of $\epsilon$, for sufficiently small $\epsilon$ (recall that $\alpha$ is fixed). As seen in the proof of Theorem~\ref{thm:CVkinetic}, the convergence rate  depends on three quantities, which should be  controlled uniformly in $\epsilon$ in order to get such a result:
    \begin{itemize}
        \item {\bf The log-Sobolev inequality constant of ${\Gamma}_{\alpha,\epsilon}(\nu^x_t)$}. We can use similar arguments as in the proof of Theorem~\ref{th:CVoverdamped-sharp} to show that for $t$ large enough, it is bounded from below independently of $\epsilon$.
        \item {\bf The Lipschitz constant $L_x$}. More precisely, by checking the proof of \cite{chen-lin-ren-wang-2024}, we see that what should be bounded is $L_x(t)$ where for all $t\geq 0$,
        $$L_x(t):= \left\|\nabla^2 \frac{\delta F_{\alpha,\epsilon}}{\delta\nu^x}(\cdot,\nu_t^x)\right\|_\infty.$$
        In our case, by the fact that ${\Gamma}_{\alpha,\epsilon}(\rho^*_{\alpha,\epsilon})=\rho^*_{\alpha,\epsilon},\tilde{\Gamma}_{\alpha,\epsilon}(\nu^*_{\alpha,\epsilon})=\nu^*_{\alpha,\epsilon},\nu^*_{\alpha,\epsilon}\propto\rho^*_{\alpha,\epsilon}\exp(-\beta\frac{|v|^2}{2}),$ we have that (again, $\na$ refers to derivatives
        with respect to position variable $x$) 
        $$\nabla \frac{\delta F_{\alpha,\epsilon}}{\delta\nu^x}=\beta\nabla U+\alpha \nabla K^m_\epsilon\star\ln (K^m_\epsilon\star\nu^{x,1})=\nabla\ln \rho^{*}_{\alpha,\epsilon}+\alpha \nabla K^m_\epsilon\star(\ln (K^m_\epsilon\star\nu^{x,1})-\ln (K^m_\epsilon\star\nu^{*,x,1}_{\alpha,\epsilon})).$$
        For fixed $\epsilon\in(0,1]$, the derivative of the second term converges to $0$ when $\mH(\nu^x|\nu^{*,x}_{\alpha,\epsilon})\to 0$ by a similar argument as in the proof of Lemma~\ref{lem:LSI_gamma}. Thus, $L_x(t)$ converges to $\|\nabla^2\ln \rho^{\star}_{\alpha,\epsilon} \|_\infty$ as $t\rightarrow\infty$. We would then conclude that $L_x(t)$ is independent of $\epsilon$ for large times by applying Theorem~\ref{thm:implicit} to higher order Sobolev space and by a similar argument as in the proof of Theorem~\ref{thm:cv_infty} to get that $\|\nabla^2\ln \rho^{\star}_{\alpha,\epsilon} \|_\infty$ converges to $\|\nabla^2\ln \rho^{\star}_{\alpha} \|_\infty$
        as $\epsilon$ vanishes.
        \item {\bf The Lipschitz constant $L_\nu$}. Again, by checking the proof of \cite{chen-lin-ren-wang-2024}, we see that what is precisely used is that for all $t\geq 0$ and all small $s\geq 0$, 
        $$\left\|\nabla \frac{\delta F_{\alpha,\epsilon}}{\delta\nu^x}(\cdot,\nu_t^x)-\nabla \frac{\delta F_{\alpha,\epsilon}}{\delta\nu^x}(\cdot,\nu_{t+s}^x)\right\|_\infty \leq  L_\nu(t) W_1(\nu_t^x,\nu_{t+s}^x)$$
        where $W_1$ is the Wasserstein 1-distance.
        Here, it is less clear how to exploit that $\nu_t \rightarrow \nu_{\alpha,\epsilon}^*$ for large times.  Let us recall that for $\epsilon=0$, $\frac{\delta F_{\alpha,0}}{\delta\nu^x}(\cdot,\nu^x)=\beta U + \alpha \ln\nu^{x,1}$. It is already not trivial to identify a subset of smooth probability measures on $\T^m$ over which $\mu \mapsto \nabla \ln \mu \in L^\infty(\T^m)$ is Lipschitz, $\mathcal P(\T^d)$ being endowed with the $W_1$ norm. 
     \end{itemize}
    This question of the behaviour of these quantities as $\varepsilon$ vanishes is related to the question of the well-posedness of the non-regularized dynamics from Section~\ref{sec:idealfreeenergy}. 
\end{rem}

\appendix

\section{Standard results on Sobolev spaces}\label{sec:ap2}
Recall the notation on Sobolev spaces introduced in Section~\ref{sec:notation}.

In all the following, when needed,   $\T$ is identified to the interval $I=\left[-\frac{1}{2},\frac{1}{2}\right)$, and thus $\T^m$ to $I^m=\left[-\frac{1}{2},\frac{1}{2}\right)^m$. We will also need the following definition:
\begin{defi}\label{def:extension_f}
Given a function $f: \T^m\rightarrow \R$, we define its  extension to $\R^m$ as, for $x=(x_1,x_2,...,x_m)\in \R^m,$
$$\tilde{f}(x_1,x_2,...,x_n)=f(\pi(x_1),\pi(x_2),...,\pi(x_n)),$$ where $\pi$ is the wrapping function defined by
\begin{equation}\label{eq:wrapping}
\pi:\left\{
\begin{aligned}
    \R &\to I\\
    x & \mapsto x \, {\rm mod} \, 1
\end{aligned}
\right.
\end{equation}
where $x \, {\rm mod} \, 1 = x+k$ for the integer $k \in \Z$ such that $x + k \in I$.
\end{defi}

\begin{lem}\label{lem:equivalent}
Suppose that $B_1$ is the unit ball of $\R^m$. For any $k\in \N, \, p\in [1,+\infty],\, f\in W^{k,p}(\T^m),$ let $F$ be the restriction of $\tilde{f}$ to $B_1$,  $\tilde{f}$ being the extension of $f$ to $\R^m$ as introduced in Definition~\ref{def:extension_f}. Then we have
$$2^{-m}\|F\|_{W^{k,p}(B_1)}\leq\|f\|_{W^{k,p}(\T^m)}\leq \|F\|_{W^{k,p}(B_1)}.$$
\end{lem}
\begin{proof}
Notice that $I^m=[-\frac{1}{2},\frac{1}{2})^m\subset B_1$, thus $$\|f\|_{W^{k,p}(\T^m)}=\|F\|_{W^{k,p}(I^m)}\leq \|F\|_{W^{k,p}(B_1)}.$$

Moreover, since $B_1 \subset (2I)^m=[-1,1)^m$, it holds
$$\|F\|_{W^{k,p}(B_1)}\leq \|\tilde{f}\|_{W^{k,p}([-1,1)^m]}=2^{\frac{m}{p}}\|\tilde{f}\|_{W^{k,p}(I^m)}=2^{\frac{m}{p}}\|f\|_{W^{k,p}(\T^m)}\leq 2^{m}\|f\|_{W^{k,p}(\T^m)}.$$
\end{proof}
\begin{lem}[Sobolev embedding theorem in $\T^m$]\label{lem:embedding}
Suppose that $m,k\in \N^+$ and $2k>m$. Then, there exists a constant $C_{\infty}$, such that for all $f\in H^k(\T^m),$
\begin{equation}\label{eq:embedding_infty}
\|f\|_{\infty}\leq C_{\infty}\|f\|_{H^k}.    
\end{equation}
Moreover, for $i\in \N \cap [k-\frac{m}{2},k-1],$
and $q\in [2,\infty)$ such that $\frac{1}{2}-\frac{k}{m}< \frac{1}{q}-\frac{i}{m}$, there exists a constant $C_{i,q}$ depending only on $k,i,q$, such that for all $f\in H^k(\T^m),\theta\in \N^m, |\theta|=i,$

\begin{equation}\label{eq:embedding_q}
\|D^\theta f\|_{q}\leq C_{i,q}\|f\|_{H^k}. \end{equation}

\end{lem}
\begin{proof} By Lemma~\ref{lem:equivalent}, we need only to consider the Sobolev space $W^{k,p}(B_1)$ in the unit ball $B_1$ of $\R^m$ instead of $\T^m$. Notice that $B_1$ is a bounded subdomain of $\R^m$ with $C^1$ boundary.

Equation~\eqref{eq:embedding_infty} is a direct consequence of~\cite[Theorem 6, p. 270]{evans2010partial}, and the fact that $2k>m$. And to prove~\eqref{eq:embedding_q}, notice that 

$$\|D^\theta f\|_{H^{k-i}}\leq \|f\|_{H^k},$$

We take $p$ such that $\frac{1}{q}=\frac{1}{p}-\frac{k-i}{m}.$ Notice that $p\in[1,2)$, since

$$\frac{1}{2}<\frac{1}{q}+\frac{k-i}{m}\leq \frac{1}{2}+\frac{1}{2}=1.$$

Thus by \cite[Theorem 6, p. 270]{evans2010partial}, and the fact that $k-i\leq \frac{m}{2}<\frac{m}{p},$ there exists a constant $C'_{i,q}$ depending only on $i,q$ (since $p$ depends on $i,q$), such that
$$\|D^\theta f\|_q\leq C'_{i,q}\|D^\theta f\|_{W^{k-i,p}}.$$

By Holder's inequality, there exists a constant $C''_{i,q}$ only depending on $i,q$ (since $p$ depends on $i,q$), such that 
$$\|D^\theta f\|_{W^{k-i,p}}\leq C''_{i,q} \|D^\theta f\|_{H^{k-i}}.$$
The result follows by gathering these inequalities.
\end{proof}
\begin{lem}\label{lem:holder}
Suppose that $j,m,k\in \mathbb{N^+}$ and $2k>m$. Then there exists a constant $C$, such that if $f_i\in H^k(\T^m),\theta_i\in \mathbb{N}^m$ for $i=1,2,...,j$, 
$\sum_{i=1}^j|\theta_{i}|\leq k$, $q_i=\frac{2\sum_{i=1}^j|\theta_{i}|}{|\theta_i|}$ (if $|\theta_i|=0$, then $q_i=\infty$ by convention), then 

$$\|\Pi_{i=1}^j D^{\theta_i}f_i\|_{2}\leq \Pi_{i=1}^j \|D^{\theta_i}f_i\|_{q_i}\leq C \Pi_{i=1}^j\|f_i\|_{H^k}.$$
\end{lem}
\begin{proof}
If $\sum_{i=1}^j|\theta_{i}|=0$, then we conclude directly by~\eqref{eq:embedding_infty} of
Lemma~\ref{lem:embedding}.

In the following, we assume that $\sum_{i=1}^j|\theta_{i}|\geq 1.$ Notice that $\sum_{i=1}^j \frac{2}{q_i}=1$, we deduce the first inequality by Holder's inequality, 

\begin{align*}
\int_{\T^m} \Pi_{i=1}^j (D^{\theta_i}f_i)^2\leq \Pi_{i=1}^j \left(\int_{\T^m}(D^{\theta_i}f_i)^{q_i}\right)^{\frac{2}{q_i}}
=\Pi_{i=1}^j \|D^{\theta_i}f_i\|_{q_i}^{2}.
\end{align*}

To obtain the second inequality, notice that $|\theta_i|\leq\sum_{i=1}^j|\theta_{i}|\leq k.$

If $0\leq|\theta_i|<k-\frac{m}{2}$ (when $|\theta_i|=0,q_i=\infty$), then $D^{\theta_i} f_i\in H^{k-|\theta_i|}(\T^m), 2(k-|\theta_i|)>m.$ Thus by~\eqref{eq:embedding_infty}, $$\|D^{\theta_i}f_i\|_{q_i}\leq\|D^{\theta_i}f_i\|_{\infty}\leq C_\infty \|D^{\theta_i} f_i\|_{H^{k-|\theta_i|}}\leq C_\infty \|f_i\|_{H^k}.$$

If $|\theta_i|=k$, then since $\sum_{i=1}^j|\theta_{i}|\leq k$, we have that $\sum_{i=1}^j|\theta_{i}|=k$, and thus $q_i=2$, and 
$$\|D^{\theta_i} f_i\|_{q_i}=\|D^{\theta_i} f_i\|_{2}\leq \|f_i\|_{H^k}.$$

If $|\theta_i|\in \N\cap [k-\frac{m}{2},k-1],$ then by the fact that $\frac{1}{2k}-\frac{1}{m}< 0,  \sum_{i=1}^j|\theta_{i}|\leq k$,
$$\frac{1}{2}-\frac{k}{m}=k\po \frac{1}{2k}-\frac{1}{m}\pf< |\theta_i|\po\frac{1}{2k}-\frac{1}{m}\pf \leq \frac{|\theta_i|}{2\sum_{i=1}^j|\theta_{i}|}-\frac{|\theta_i|}{m} =\frac{1}{q_i}-\frac{|\theta_i|}{m}.$$ By~\eqref{eq:embedding_q} of Lemma~\ref{lem:embedding}, this ensures that
$$\|D^{\theta_i}f_i\|_{q_i}\leq C_{|\theta_i|,q_i} \|f_i\|_{H^k}.$$
By convention, we pose $C_{k,2}=1,C_{|\theta_i|,q_i}=C_\infty$ if $0\leq|\theta_i|<k-\frac{m}{2}.$
The conclusion follows by taking $$C=\max\left\{\Pi_{i=1}^j C_{|\theta_i|,q_i} \Big|1\leq\sum_{i=1}^j|\theta_i|\leq k, q_i=\frac{2\sum_{i=1}^j|\theta_{i}|}{|\theta_i|} \right\}$$
which is finite since $\left\{(|\theta_i|,q_i)_{1 \le i \le j} \in (\N \times (\R \cup \{+\infty\}))^j \mid 1\leq\sum_{i=1}^j|\theta_i|\leq k, q_i=\frac{2\sum_{i=1}^j|\theta_{i}|}{|\theta_i|}\right\}$ is a finite set. \end{proof}

\begin{lem}\label{lem:product}Suppose that $2k>m$.
If $f,g\in H^k(\T^m)$, then $fg \in H^k(\T^m)$ and $\|fg\|_{H^k}\leq C\|f\|_{H^k}\|g\|_{H^k}$, where $C$ is a constant depending only on $k,m$. 
\end{lem}
\begin{proof}
It suffices to prove that for $\theta\in \N^m,$ $|\theta|\leq k$, we have that $\|D^\theta(fg)\|_2\leq C\|f\|_{H^k}\|g\|_{H^k}.$ Therefore, it suffices to prove that for
$\theta_1,\theta_2\in \N^m,$
$|\theta_1|+|\theta_2|\leq k$, it holds $\|D^{\theta_1}fD^{\theta_2}g\|_2\leq C\|f\|_{H^k}\|g\|_{H^k}$. This is a consequence of Lemma~\ref{lem:holder} for $j=2, f_1=f,f_2=g.$
\end{proof}

\begin{prop}\label{prop:phi(g)}
Suppose that $m,k\in \N^+, 2k>m,$ then there exists a constant $C>0$, such that if  
\begin{itemize}
    \item  $g\in H^{k}(\T^m)$, (thus $\inf g,\sup g\in \R$ by~\eqref{eq:embedding_infty}).
    \item $\phi\in C^k((a,\infty))$, where $a\in\R$ and $a<\inf g$,  and for $j=0,1,...,k,$ its $j$-th order derivative $\phi^{(j)}$ is bounded over $[\inf g,\sup g]$.
\end{itemize}
Then $\phi(g)\in H^k(\T^m),$ and one has the following estimate of its $H^k$ norm:
\begin{align}
\|\phi(g)\|_{H^k(\T^m)}\leq& C\sum_{j=0}^k\|\phi^{(j)}(g)\|_{L^\infty(\T^m)}\|g\|^j_{H^k(\T^m)}\label{eq:phi(g)1} \\ 
\leq&C\sum_{j=0}^k\|\phi^{(j)}\|_{L^\infty([\inf g,\sup g])}\|g\|^j_{H^k(\T^m)}.\label{eq:phi(g)}
\end{align}
\end{prop}
\begin{proof}

It suffices to bound $\|D^\theta \phi(g)\|_{L^2(\T^m)}$, for $\theta\in \N^m, |\theta|\leq k$. 

When $|\theta|=0$, the result holds since $\|\phi(g)\|_{L^2(\T^m)}\leq \|\phi(g)\|_{L^\infty(\T^m)}$.

When $|\theta|\geq 1$, it is easy to prove by induction the following claim: $D^\theta \phi(g)$ is a linear combination of functions of the form
$$\phi^{(j)}(g)\Pi_{i=1}^j D^{\theta_i}g$$
where $1\leq j\leq|\theta|,\theta_i\in \mathbb{N}^m$, $\sum_{i=1}^j |\theta_i|=|\theta|$. Moreover, one has $$\|\phi^{(j)}(g)\Pi_{i=1}^jD^{\theta_i}g\|_{L^2(\T^m)}\leq\|\phi^{(j)}(g)\|_{L^\infty(\T^m)}\|\Pi_{i=1}^jD^{\theta_i}g\|_{L^2(\T^m)}.$$
By applying Lemma~\ref{lem:holder} with $f_i=g$ for $i=1,2,...,j$, one gets $$ \|\Pi_{i=1}^j D^{\theta_i}g\|_{L^2(\T^m)}\leq C\|g\|_{H^k(\T^m)}^{j}.$$
This yields~\eqref{eq:phi(g)1}. The inequality~\eqref{eq:phi(g)} then follows since $$\|\phi^{(j)}(g)\|_{L^\infty(\T^m)}\leq \|\phi^{(j)}\|_{L^\infty ([\inf g,\sup g])}.$$ 

\end{proof}

\begin{prop}\label{prop:estimate_h}
Suppose that $m,k\in \N^+, 2k>m$.
Let $\phi\in C^k((-1,\infty))$ be such that $\phi(0)=\phi'(0)=0$. Then, there exists $\delta>0$, such that when $g\in H^k(\T^m)$ and $\|g\|_{H^k(\T^m)}\leq \delta$, $\phi(g)\in H^k(\T^m)$. Moreover, $\|\phi(g)\|_{H^k(\T^m)}=o(\|g\|_{H^k(\T^m)})$, i.e.
\begin{equation}\label{eq:o(g)}
\lim_{\|g\|_{H^k(\T^m)}\rightarrow 0} \frac{\|\phi(g)\|_{H^k(\T^m)}}{\|g\|_{H^k(\T^m)}}=0.    
\end{equation}

\end{prop}
\begin{proof}
By~\eqref{eq:embedding_infty} of Lemma~\ref{lem:embedding}, there exists $\delta>0$ such that when $\|g\|_{H^k}\leq \delta$ then $\|g\|_\infty\leq\frac{1}{2}$. By assumption, for $j=0,1,...,k,$ $\phi^{(j)}$ is continuous in $[-\frac{1}{2},\frac{1}{2}]$, thus also bounded in $[-\frac{1}{2},\frac{1}{2}]$. By
applying Proposition~\ref{prop:phi(g)} to $\phi(g)$ (with $a=-1$), one thus gets $\phi(g)\in H^k(\T^m).$ 

Let us now prove~\eqref{eq:o(g)}. Let us assume that $\|g\|_{H^k(\T^m)}\leq\delta$, so that $\|g\|_{L^\infty(\T^m)}\leq \frac{1}{2}.$ We will
estimate $\|\phi(g)\|_{H^k(\T^m)}$ using~\eqref{eq:phi(g)1}. Notice that, $[\inf g,\sup g]\subseteq [-\frac{1}{2},\frac{1}{2}],$ and by Lemma~\ref{lem:embedding}, $\lim_{\|g\|_{H^k(\T^m)}\rightarrow 0} \|g\|_{L^\infty(\T^m)} \rightarrow 0$. 

For the term $j=0$ in~\eqref{eq:phi(g)1}, since  $\phi(0)=\phi'(0)=0$, one has $\|\phi(g)\|_{L^\infty(\T^m)}\leq \frac{1}{2}\|\phi''\|_{L^\infty([-\frac{1}{2},\frac{1}{2}]) }\|g\|^2_{L^\infty(\T^m)} =o(\|g\|_{H^k(\T^m)})$.

For the term $j=1$ in~\eqref{eq:phi(g)1}, it holds
$$\lim_{\|g\|_{H^k(\T^m)}\rightarrow 0}\|\phi'(g)\|_{L^\infty(\T^m)}=\lim_{\|g\|_{L^\infty(\T^m)}\rightarrow 0}\|\phi'(g)\|_{L^\infty(\T^m)}=|\phi'(0)|=0.$$

Finally, for the terms $j\geq 2$ in~\eqref{eq:phi(g)1}, we notice that $$\lim_{\|g\|_{H^k(\T^m)}\rightarrow 0}\|\phi^{(j)}(g)\|_{L^\infty(\T^m)}=\lim_{\|g\|_{L^\infty(\T^m)}\rightarrow 0}\|\phi^{(j)}(g)\|_{L^\infty(\T^m)}=|\phi^{(j)}(0)|.$$

This concludes the proof.
\end{proof}

\section{Gaussian kernel on the torus $\T^m$}\label{sec:ap}
This appendix is devoted to the definition and properties the Gaussian kernel on the torus~$\T^m$, which is extensively used in this work.

\subsection{The Gaussian kernel, lower and upper bounds}

\begin{defi}
Let $\phi(x)=\frac{1}{\sqrt{2\pi}} \exp\left(-\frac{x^2}{2}\right)$ be the density of a standard Gaussian random variable. For any $\epsilon>0$, we define the Gaussian kernel  $K^1_\epsilon$ on $\T$  by: for all $x \in \T$,

\begin{equation}\label{eq:sum_gaussian_kernel}
K^1_\epsilon(x)=\frac{1}{\sqrt \epsilon}\sum_{n\in \mathbb{Z}} \phi\left(\frac{x+n}{\sqrt\epsilon}\right).    
\end{equation}

More generally, for $\epsilon>0$ and $m\geqslant 1$, we  define the Gaussian kernel  $K^m_\epsilon$ on $\T^m$  by: for all $x=(x_1,x_2,...,x_m)\in \T^m$:

\begin{equation}\label{eq:product_gaussian_kernel}
K^m_\epsilon(x)=\Pi_{i=1}^m K^1_\epsilon(x_i).    
\end{equation}

\end{defi}

Notice that by definition, $K^1_\epsilon$ is indeed a periodic function (with period $1$). As a periodic function, one can introduce its Fourier coefficients: for $k \in \Z$,
$$\hat{K^1_\epsilon}(k)=\int_\T \exp(-2i\pi kx) K^1_\epsilon(x) dx= \int_\R \frac{1}{\sqrt \epsilon} \phi\left( \frac{x}{\sqrt \epsilon}\right) \exp(-2i \pi k x) \, dx = \exp(-2\pi^2 k^2 \epsilon),$$
and $K^1_\epsilon(x) = \sum_{k \in \Z} \hat{K^1_\epsilon}(k) \exp (2i \pi kx) $.

The function $\epsilon \mapsto K^1_\epsilon$ can be identified as the heat kernel on $[-\frac{1}{2},\frac{1}{2})$ with periodic boundary conditions ($\epsilon \ge 0$ being the time). Besides, a probabilistic interpretation can be obtained using the wrapping function $\pi$ defined in~\eqref{eq:wrapping}. Indeed, $K^1_\epsilon|_I$ is the density over $I$ of $\pi(\sqrt{\epsilon} G)$ where $G$ is a standard Gaussian random variable. In particular $\int_I K^1_\epsilon=\int_\T K^1_\epsilon = 1$.

Let us now provide lower and upper bounds for $K^m_\epsilon$.

\begin{lem}\label{lem:estim_K}
There exists a constant $C>1$ such that, for all $\epsilon \in (0,1]$, for all $x \in  I$,
\begin{equation}\label{eq:estim_K}
  \frac{1}{\sqrt \epsilon}\phi\left(\frac{x}{\sqrt \epsilon}\right)\le K^1_\epsilon(x) \le \frac{C}{\sqrt \epsilon} \phi\left(\frac{x}{\sqrt \epsilon}\right).
\end{equation}    
As a consequence, for all $x\in I^m$, 
\begin{equation}\label{eq:Linf_K}
  \left(\frac{1}{\sqrt {2\pi\epsilon}}\right)^m\exp\left(-\frac{m}{8\epsilon}\right)\le K^m_\epsilon(x) \le \left(\frac{C}{\sqrt {2\pi\epsilon}}\right)^m.    
\end{equation}

\end{lem}\begin{proof}
The lower bound is obvious by simply considering the term $n=0$ in the sum defining~$K^1_\epsilon$. Notice that the upper bound is equivalent to: for all $x \in I$,
$$\sum_{n \in \Z} \exp\left(-\frac{nx}{ \epsilon} \right) \exp\left(-\frac{n^2}{ 2 \epsilon} \right) \le C$$
which is equivalent to
$$ \sum_{n \ge 1} \exp\left(-\frac{n^2}{ 2\epsilon}\right) \cosh\left(\frac{nx}{ \epsilon} \right) \le \frac{C-1}{2}.$$
Since $|x|\le 1/2$, one has (since $-n^2+n \le -(n-1)^2$ for all $n \ge 1$)
\begin{align*}
    \sum_{n \ge 1} \exp\left(-\frac{n^2}{ 2\epsilon}\right) \cosh\left(\frac{nx}{ \epsilon} \right)
&    \le
    \sum_{n \ge 1} \exp\left(-\frac{n^2}{ 2 \epsilon}\right) \cosh\left(\frac{n}{ 2 \epsilon} \right)\\
    & \le
    \sum_{n \ge 1} \exp\left(-\frac{n^2}{ 2 \epsilon}\right) \left( 1 + \exp\left(\frac{n}{ 2 \epsilon} \right)\right) \\
    & \le
    2 \sum_{n \ge 0} \exp\left(-\frac{n^2}{ 2 \epsilon}\right)\\
    & \le 2 \left( 1 + \int_0^\infty \exp\left(-\frac{x^2}{2 \epsilon}\right) \, dx\right) = 2 \left( 1 + \sqrt{\frac{\epsilon \pi}{2} } \right)
\end{align*}
which proves the upper bound in~\eqref{eq:estim_K} since $\epsilon \in (0,1]$, and thus yields the result.

Finally, the inequalities~\eqref{eq:Linf_K} are consequences of the fact that, for all $x \in I$, $\phi\left(\frac{1}{2\sqrt{\epsilon}}\right)\leq\phi\left(\frac{x}{\sqrt{\epsilon}}\right)\leq\phi(0)$.
\end{proof}
Notice that the previous lemma is useful to provide lower and upper bounds of probability densities convoluted with $K^m_\epsilon$ since, for any probability density $\rho$ on $\T^d$, one has (since $\rho \ge0$ and $\int_{\T^d} \rho =1$):
$$\inf_{\T^d} K^m_\epsilon \le K^m_\epsilon\star \rho\le \sup_{\T^d} K^m_\epsilon.$$
Besides, it is easy to check (since $K^m_\epsilon \ge0$ and $\int_{\T^d} K^m_\epsilon =1$) that, for any function $\rho:\T^d \to \R$
$$\inf_{\T^d} \rho \le K^m_\epsilon\star \rho\le \sup_{\T^d} \rho.$$

\begin{lem}\label{lem:e_var}
For all $\epsilon,\epsilon'$, $K^m_\epsilon\star K^m_{\epsilon'}=K^m_{\epsilon+\epsilon'}$. Moreover, there exists a constant $C>0$, such that for all $\epsilon\in (0,1]$,  $\int_{I^m} |x| K^m_\epsilon(x)dx\leq C m\sqrt{\epsilon}$ and
$\int_{I^m} |x|^2K^m_\epsilon(x)dx\leq C m\epsilon$. \end{lem}

\begin{proof}
Let us first consider the case  $m=1$.
The first property can be seen as a consequence of a semigroup property, since $K^1_\epsilon$ is the heat kernel on the torus. This can also be checked directly using the wrapping function $\pi$, as follows. Let $G$ (resp. $G'$) be a Gaussian random variable of variance $\epsilon$ (resp. $\epsilon'$). Moreover, $\pi(G)$ (respectively $\pi(G')$) has density $K^1_\epsilon$ (respectively $K^1_{\epsilon'}$) over $\left[-\frac 1 2, \frac 1 2\right)$. By the probabilistic interpretation of the convolution mentioned above, $\pi(\pi(G)+\pi(G'))$ then has a density $K^1_\epsilon \star K^1_{\epsilon'}$. On the other hand, $\pi(\pi(G)+\pi(G'))=\pi(G+G')$ and $G+G'$ is a Gaussian random variable with variance $\epsilon + \epsilon'$. Thus $\pi(\pi(G)+\pi(G'))$ has a density $K^1_{\epsilon + \epsilon'}$ over $\left[-\frac 1 2, \frac 1 2\right)$ and this concludes the proof.


For the inequalities, using Lemma~\ref{lem:estim_K}, one has (by a change of variable):
$$\int_I |x| K^1_\epsilon(x)dx\leq \frac{C}{\sqrt \epsilon} \int_I |x| \phi \left(\frac{x}{\sqrt{\epsilon}}\right) \, dx \leq  C'\sqrt{\epsilon}$$
with $C' = C\int_\R |x| \phi(x) \, dx$, and
$$\int_I x^2K^1_\epsilon(x)dx\leq \frac{C}{\sqrt \epsilon} \int_I x^2\phi \left(\frac{x}{\sqrt{\epsilon}}\right) \, dx \leq  C''\epsilon,$$
with $C'' = C\int_\R |x|^2 \phi(x) \, dx$. This concludes the proof for the case when $m=1$.

The results for $m>1$ are easily deduced from the definition of $K^m_\epsilon$, and the fact that $|x|\leq \sum_{i=1}^m |x_i|, |x|^2= \sum_{i=1}^m x_i^2.$
\end{proof}

\subsection{Regularizing using convolution with the Gaussian kernel $K^m_\epsilon$}

Let us define the convolution $\star$ on the torus as follows: for any two functions $f:\T \to \R$ and $g:\T \to \R$, for all $x \in \T$
$$f\star g(x) = \int_\T f(x-y) g(y) \, dy.$$

The probabilistic interpretation of this convolution is the following: if $f:\T\to \R$ and $g:\T \to \R$ are such that $f|_I$ and $g|_I$ are probability densities over $I$, and if $X \in I$ and $Y \in I$ are independent random variables with respective densities $f|_I$ and $g|_I$, then $\pi(X+Y)$ has density $f \star g$. For the sake of completeness, let us provide a proof of this.  First, $X+Y$ (seen as a random variable over $\R$) has density: for all $z \in \R$, $f_{X+Y}(z)=\int_\R f(z-y) 1_{z-y \in I} g(y) 1_{y \in I} dy$. Notice that $X+Y$ is with values in $[-1,1)$. Then $\pi(X+Y)$ has density: for all $z \in I$,
\begin{align*}
\sum_{n \in \Z} f_{X+Y}(z+n)&=
f_{X+Y}(z-1)+f_{X+Y}(z)+f_{X+Y}(z+1)\\
&=\int_\R \left( f(z-1-y) 1_{z-1-y \in I} + f(z-y) 1_{z-y \in I}  + f(z+1-y) 1_{z+1-y \in I}\right) g(y) 1_{y \in I} dy\\
&=\int_I f(z-y) ( 1_{z-1-y \in I}+ 1_{z-y \in I} + 1_{z+1-y \in I}) g(y) \, dy\\
&=\int_\T f(z-y) g(y) \, dy
\end{align*}
where we used the periodicity of $f$ and the fact that for $y \in I$ and $z \in I$,
$1_{z-1-y \in I}+ 1_{z-y \in I} + 1_{z+1-y \in I}=1$. This shows that $\pi(X+Y)$ has density $f \star g$.

Finally, let us recall that if $f(x)=\sum_{k \in \Z} \hat f(k) \exp (2i \pi kx)$ and $g(x)=\sum_{k \in \Z} \hat g(k) \exp (2i \pi kx)$, then $f \star g(x) = \sum_{k \in \Z} \widehat{f\star g} (k) \exp (2i \pi kx)$ where
\begin{equation}\label{eq:convol_fourier}
\widehat{f\star g} (k)=\hat f (k) \hat g(k).
\end{equation}

Recall that the centered Gaussian kernel on $\R^d$ with variance $\sigma^2 {\rm Id}$ converges when $\sigma \to 0$ to the Dirac measure in some sense (for example, in the sense of distribution). Let us give a lemma to state a similar property for the Gaussian kernel $K^m_\epsilon$ on $\T^m$ (this lemma will be used in 
Lemma~\ref{lem:limit_entropy} below).

\begin{lem}\label{lem:cv_Dirac}
Suppose that $f\in L^1(\T^m),$ then for almost every $x\in \T^m$, $K^m_\epsilon\star f(x)\rightarrow f(x),$ when $\epsilon\rightarrow 0^+.$   
\end{lem}
\begin{proof}
This is a standard result, which can be found in~\cite[p. 112]{stein2009real} for $f$ defined on $\R^m$. For the sake of completeness, let us give an adaptation of the proof on $\T^m$.

Let $\tilde{f}$ be the extension of $f$ to $\R^m$. For $x=(x_1,x_2,...,x_m)\in I^m$, 
\begin{align}
K^m_\epsilon\star f(x)=&\int_{I^m} K^m_\epsilon(y)f(x-y)dy\nonumber\\
=&\int_{\R^m} \left(\frac{1}{\sqrt{\epsilon}}\right)^m\prod_{i=1}^m\phi\left(\frac{y_i}{\sqrt{\epsilon}}\right)\tilde{f}(x_1-y_1,x_2-y_2,...,x_m-y_m)dy_1dy_2...dy_m\nonumber\\
=& \Phi_\epsilon\star \tilde{f}(x), \label{eq:phi_eps_f_tilde}
\end{align}
where $\Phi_\epsilon(y)=(\frac{1}{2\pi\epsilon})^{\frac{m}{2}}\exp(-\frac{|y|^2}{2\epsilon})$.

Since $\tilde{f}$ is locally integrable on $\R^m$, by \cite[Corollary 1.6, p. 107]{stein2009real},  almost every point $x\in \R^m$
is a Lebesgue point of $\tilde{f}$, i.e., a point $x$ such that
$$\lim_{r \rightarrow 0^+} \mathcal{A}(x,r) =0$$
where
$$\mathcal{A}(x,r)=\frac{\int_{|y|\leq r}|\tilde{f}(x-y)-\tilde{f}(x)|dy}{r^m}.$$
We will prove that for every Lebesgue point $x \in I^m$ of $\tilde{f}$, $\lim_{\epsilon\rightarrow 0^+}\Phi_\epsilon\star\tilde{f}(x)=\tilde{f}(x)=f(x)$, which, thanks to~\eqref{eq:phi_eps_f_tilde}, yields the result.

Let $x \in I$ be a Lebesgue point of $\tilde{f}$ so that $\lim_{r\rightarrow 0^+}\mathcal{A}(x,r)=0$. The function $r \mapsto \mathcal{A}(x,r)$ is continuous over $(0,+\infty)$, since $y \in \R^m \mapsto |\tilde{f}(x-y)-\tilde{f}(x)|$ is locally in $L^1$. Besides, we claim that the function $r\mapsto\mathcal{A}(x,r)$ is bounded, which is a consequence of the fact that 
\begin{equation}\label{eq:A_bdd}
\limsup_{r\rightarrow \infty} \mathcal{A}(x,r)<\infty.
\end{equation}
Indeed, notice that $$\frac{\int_{|y|\leq r}|\tilde{f}(x-y)|dy}{r^m}\leq \frac{\int_{|y|\leq r+|x|}|\tilde{f}(y)|dy}{r^m}.$$
Let us take $R\in \N$ such that $R\leq r+|x|<R+1$, then by the periodicity of $\tilde{f},$
$$\frac{\int_{|y|\leq r+|x|}|\tilde{f}(y)|dy}{r^m}\leq \frac{\int_{|y|\leq R+1}|\tilde{f}(y)|dy}{r^m}\leq \frac{\int_{[-R-1,R+1]^m}|\tilde{f}(y)|dy}{r^m}=\frac{4^m(R+1)^m}{r^m}\int_{I^m}|\tilde{f}(y)|dy,$$
and in particular $\frac{4^m(R+1)^m}{r^m}\leq \frac{4^m(r+|x|+1)^m}{r^m}$, which converges to $4^m$ when $r\rightarrow \infty.$
This yields~\eqref{eq:A_bdd} since $$\limsup_{r\rightarrow \infty} \mathcal{A}(x,r)\leq |\tilde{f}(x)|m(B_1)+4^m
\int_{I^m}|\tilde{f}(y)|dy$$ 
where $m(B_1)$ is the volume of unit ball in $\R^m$.

Let us now write
the integral 
$$|\Phi_\epsilon \star \tilde f(x) - \tilde f (x)| = \left| \int_{\R^m} (\tilde{f}(x - y) - \tilde{f}(x))\Phi_\epsilon(y) \, dy \right| \le \int_{\R^m} |\tilde{f}(x - y) - \tilde{f}(x)| \Phi_\epsilon(y) \, dy $$
over \( \R^m\) as a sum of integrals over annuli as follows:
\begin{align*}
\int_{\R^m} |\tilde{f}(x - y) - \tilde{f}(x)| \Phi_\epsilon(y) \, dy &= \int_{|y| \leq \sqrt{\epsilon}} |\tilde{f}(x - y) - \tilde{f}(x)| \Phi_\epsilon(y) \, dy 
\\ &\quad + \sum_{k=0}^{\infty} \int_{2^k \sqrt{\epsilon} < |y| \leq 2^{k+1} \sqrt{\epsilon}} |\tilde{f}(x - y) - \tilde{f}(x)| \Phi_\epsilon(y) \, dy.
\end{align*}
We pose $c=(\frac{1}{\sqrt{2\pi}})^m$. The first term is estimated by, 
\begin{align*}
\int_{|y| \leq \sqrt{\epsilon}} |\tilde{f}(x - y) - \tilde{f}(x)| \Phi_\epsilon(y) \, dy&\leq \left(\frac{1}{\sqrt{2\pi\epsilon}}\right)^m\int_{|y| \leq \sqrt{\epsilon}} |\tilde{f}(x - y) - \tilde{f}(x)|  \, dy =c \mathcal{A}(x,\sqrt{\epsilon}).
\end{align*}
And since $\lim_{k\rightarrow\infty} 2^{(k+1)m+k}\exp(-2^{2k-1})=0
$, we take $c'=c\sup_{k\in \N}  2^{(k+1)m+k}\exp(-2^{2k-1})<\infty
,$ and thus
\begin{align*}
\int_{2^k \sqrt{\epsilon} < |y| \leq 2^{k+1} \sqrt{\epsilon}}  |\tilde{f}(x - y) - \tilde{f}(x)| \Phi_\epsilon(y) \, dy 
   &\leq \frac{c}{( \sqrt{\epsilon})^{m}}\exp(-2^{2k-1}) \int_{|y| \leq 2^{k+1} \sqrt{\epsilon}} |\tilde{f}(x - y) - \tilde{f}(x)| \, dy \\
&= c2^{(k+1)m}\exp(-2^{2k-1})\mathcal{A}(x,2^{k+1}\sqrt\epsilon)\\ 
&\leq  c' 2^{-k}\mathcal{A}(x,2^{k+1}\sqrt{\epsilon}).
\end{align*}
Putting these estimates together, we find that
\[
| \tilde{f} \star \Phi_\epsilon(x) - \tilde{f}(x) | \leq c \mathcal{A}(x,\sqrt{\epsilon}) + c' \sum_{k=0}^{\infty} 2^{-k} \mathcal{A}(x,2^{k+1} \sqrt{\epsilon}).
\]

Given \( \delta > 0 \), we first choose \( N \) sufficiently large so that \( \sum_{k \geq N} 2^{-k} < \delta \). Then, since $\lim_{r\rightarrow 0+}\mathcal{A}(x,r)=0$, by making \( \epsilon \) sufficiently small, we have 
\[
\mathcal{A}(x,2^k \sqrt{\epsilon}) < \frac{\delta}{N}, \quad \text{whenever } k = 0, 1, \dots, N - 1.
\]
Hence, recalling that \( \mathcal{A}(x,r) \) is bounded, we find
\[
| \tilde{f} \star \Phi_\epsilon(x) - \tilde{f}(x) | \leq \max(c,c')(\sup_r \mathcal{A}(x,r)+1) \delta
\]
for all sufficiently small \( \epsilon\), and the theorem is proved, since $\delta$ is arbitrary.
\end{proof}

A consequence of the previous result is the following lemma.

\begin{lem}\label{lem:limit_entropy}
Suppose that $\rho$ is a probability density on $\T^d$ such that $\rho\ln\rho\in L^1(\T^d)$, then 
$$\lim_{\delta\rightarrow 0^+}\int_{\T^d}(K^d_\delta\star \rho)\ln(K^d_\delta\star \rho)=\int_{\T^d}\rho\ln\rho.$$    
\end{lem}
\begin{proof}
It suffices to prove the limit above for any sequence $\delta_n\rightarrow 0^+$.
On the one hand, since $\rho\in L^1(\T^d)$, by Lemma~\ref{lem:cv_Dirac}, $K^d_{\delta_n}\star\rho\rightarrow \rho$ a.e.
And since $\phi(x)=x\ln x$ is lower bounded, then by Fatou's lemma,
$$\liminf_{n\rightarrow \infty}\int_{\T^d}(K^d_{\delta_n}\star\rho)\ln (K^d_{\delta_n}\star\rho)\geq \int_{\T^d}\rho\ln \rho.$$

On the other hand, by the convexity of $\phi(x)$ and Jensen's inequality,
\begin{align*}
\int_{\T^d}(K^d_\delta\star\rho)\ln (K^d_\delta\star\rho)=& \int_{\T^d}\phi(K^d_\delta\star\rho)\\=&  \int_{\T^d}\phi\left(\int_{\T^d}K^d_\delta(x-y)\rho(y)dy\right)dx\\\leq&\int_{\T^d}\int_{\T^d}K^d_\delta(x-y)\phi(\rho(y))dydx\\=&\int_{\T^d} \phi(\rho(y))dy=\int_{\T^d}\rho\ln\rho,
\end{align*}
where $\rho\ln\rho\in L^1(\T^d)$ justifies the Fubini theorem for the interchange of integrations.
Thus it holds
$$\limsup_{n\rightarrow \infty}\int_{\T^d}(K^d_{\delta_n}\star\rho)\ln (K^d_{\delta_n}\star\rho)\leq \int_{\T^d}\rho\ln \rho.$$
This concludes the proof.
\end{proof}

For regular functions, the convergence of $K^m_\epsilon\star g$ to $g$ can be quantified in terms of $\epsilon$.
\begin{lem}\label{lem:error_convolution}
There exists $C>0$ such that for any $\epsilon\in(0,1]$ and any function $g\in C^1(\T^m)$, $\|g-K^m_\epsilon\star g\|_\infty\leq Cm\|
D^1 g\|_\infty \sqrt{\epsilon}$.
Likewise, there exists $C>0$ such that for any $\epsilon\in(0,1]$ and any  function $g\in C^2(\T^m)$, $\|g-K^m_\epsilon\star g\|_\infty\leq Cm^2\|D^2 g\|_\infty \epsilon$.
\end{lem}
\begin{proof}
For the first inequality, notice that \begin{align*}
    |g(x)-K^m_\epsilon\star g(x)|&=\left|\int_{\T^m} K^m_\epsilon(y)(g(x)-g(x-y))dy\right|\\
&=\left|\int_{I^m} K^m_\epsilon(y)(g(x)-g(x-y))dy\right|\\  &  \leq \|D^1 g\|_\infty\int_{I^m} K^m_\epsilon (y)|y|dy
\end{align*} which yields the conclusion using Lemma~\ref{lem:e_var}.

Likewise, for the second inequality, we consider the second order Taylor expansion
$$g(x-y)-g(x)=-\sum_{i=1}^m \frac{\partial g}{\partial x_i}(x)y_i+\sum_{i,j=1}^m R_{i,j}(x,y)y_iy_j, $$
where $R_{i,j}(x,y)=\int_0^1 (1-t)\frac{\partial^2 g}{\partial_i\partial_j}(x - ty) \, dt$, and thus $|R_{i,j}(x,y)|\leq \frac{1}{2}\|D^2 g\|_\infty.$
Thus, using the fact that $K^m_\epsilon$ is even on $I^m=(-\frac{1}{2},\frac{1}{2})^m$, it holds: for all $x \in \T^m$,
\begin{align*}
|g(x)-K^m_\epsilon\star g(x)|&=\left|\int_{\T^m} K^m_\epsilon(y)(g(x-y)-g(x))dy\right|\\&=\left|\int_{I^m} K^m_\epsilon(y)\left(-\sum_{i=1}^m \frac{\partial g}{\partial x_i}(x)y_i+\sum_{i,j=1}^m R_{i,j}(x,y)y_iy_j\right)dy\right|\\
&=\left|\int_{I^m} K^m_\epsilon(y)\sum_{i,j=1}^m R_{i,j}(x,y)y_iy_j dy\right|\\&\leq \frac{m}{2}\|D^2g\|_\infty\int_{I^m} K^m_\epsilon (y)|y|^2dy,
\end{align*}
which yields the conclusion using again Lemma~\ref{lem:e_var}.
\end{proof}

\subsection{Estimates involving convolutions with the Gaussian kernel $K^m_\epsilon$}\label{sec:ap3}

\begin{lem}\label{lem:convolution}
For any $f\in L^2(\T^m)$,
$\|f-K^m_{\epsilon}\star f\|_2\leq \|f\|_2$, $\|K^m_{\epsilon}\star f\|_2\leq \|f\|_2$. Thus, for any $f\in H^k(\T)$, $K^m_{\epsilon}\star f,f-K^m_{\epsilon}\star f\in H^k(\T^m),$ and
$\|f-K^m_{\epsilon}\star f\|_{H^k}\leq \|f\|_{H^k}$, $\|K^m_{\epsilon}\star f\|_{H^k}\leq \|f\|_{H^k}$.
\end{lem}
\begin{proof}
This is a consequence of the Parseval equality, the equality~\eqref{eq:convol_fourier} and the fact that the Fourier coefficients $\widehat{K^m_\epsilon}(k)=\exp(-2\pi^2\epsilon\sum_{i=1}^m k_i^2)$ of $K^m_\epsilon$ are in $[0,1]$. Indeed, for example, for $f:\T^m \to \R$ an $L^2$ function,
\begin{align*}
    \|f-K^m_{\epsilon}\star f\|_2^2&=\sum_{k \in \Z^m} |\widehat{(f-K^m_{\epsilon}\star f)}(k) |^2\\
&=    \sum_{k \in \Z^m} |1-\exp(-2\pi^2\epsilon\sum_{i=1}^m k_i^2)|^2 |\hat f(k) |^2\\
&\le  \sum_{k \in \Z^m}  |\hat f(k) |^2 =\|f\|_2^2.
\end{align*}
The other results are obtained in a similar way.
\end{proof}

\begin{lem}
    
\label{lem:inverse_Kf_1}
Suppose that $m,k\in \mathbb{N^+}$, $2k>m$,  $f_0\in H^k(\T^m)$, $\inf f_0>0$. Let $c=\frac{1}{2C_\infty}$ where $C_\infty$ has been defined in Lemma~\ref{lem:embedding}.
Then there exists a constant $C>0$, such that for all $\epsilon\in(0,1]$ and $f_1\in H^k(\T^m)$ such that $\|f_0-f_1\|_{H^k}\leq c\inf f_0$, if $f$ is one of the functions  $f_0$, $f_1$, $K^m_\epsilon\star f_0$, or $K^m_\epsilon\star f_1$, then $\frac{1}{f}\in H^k(\T^m)$ and
$$\left\|\frac{1}{f}\right\|_{H^k}\leq \frac{C}{\inf f_0}\sum_{j=0}^k\po \frac{\|f_0\|_{H^k}}{\inf f_0}\pf ^{j}.$$

As a consequence, if $f_0=\exp(-\frac{A}{\alpha+1})$ for some function $A \in C^{k+3}(\T^m)$ then there exists a constant $M>0$ independent of $\alpha,\epsilon$, such that if $f$ is one of the functions  $f_0$, $f_1$, $K^m_\epsilon\star f_0$, or $K^m_\epsilon\star f_1$,
$$\left\|\frac{1}{f}\right\|_{H^k} \leq M.$$  
\end{lem}
\begin{proof}
We only prove this result for $f=K^m_\epsilon\star f_1$. The other cases can be proved in the same way.

First of all, by Lemma~\ref{lem:embedding}, since $\|f_0-f_1\|_{H^k}\leq c\inf f_0$ with $c=\frac{1}{2C_\infty}$, $$\inf K^m_\epsilon\star f_1\geq\inf f_1\geq \inf f_0-\|f_0-f_1\|_\infty\geq \inf f_0 -C_\infty\|f_0-f_1\|_{H^k}\geq \frac{1}{2}\inf f_0>0.$$ 
By Lemma~\ref{lem:convolution}, $K^m_\epsilon\star f_1\in H^k(\T^m).$ Notice that the function $\phi(x)=\frac{1}{x}\in C^k((0,\infty))$ and satisfies that, for $j\in \N,$  $\phi^{(j)}(x)=(-1)^j\frac{j!}{x^{j+1}}$ is bounded by a constant times  $(\frac{1}{\inf f_0})^{j+1}$ over $[\frac{1}{2}\inf f_0,\infty)\supseteq [\inf K^m_\epsilon\star f_1,\sup  K^m_\epsilon\star f_1].$ By Proposition~\ref{prop:phi(g)}, $\phi(K^m_\epsilon\star f_1)=\frac{1}{K^m_\epsilon\star f_1}\in H^k(\T^m),$ and
$$\left\|\frac{1}{K^m_\epsilon\star f_1}\right\|_{H^k}\leq \frac{C}{\inf K^m_\epsilon\star f_1}\sum_{j=0}^k\left(\frac{\|K^m_\epsilon\star f_1\|_{H^k}}{\inf K^m_\epsilon\star f_1}\right)^{j}.$$
Then by Lemma~\ref{lem:convolution} and the fact that $\inf K^m_\epsilon\star f_1\geq \inf f_1$, we have that
$$\frac{C}{\inf K^m_\epsilon\star f_1}\sum_{j=0}^k \left(\frac{\|K^m_\epsilon\star f_1\|_{H^k}}{\inf K^m_\epsilon\star f_1}\right)^{j}\leq \frac{C}{\inf  f_1}\sum_{j=0}^k \left(\frac{\| f_1\|_{H^k}}{\inf f_1}\right)^{j}.$$
Finally, we have $\inf f_1\geq \frac{1}{2}\inf f_0$, and $\|f_1\|_{H^k}\leq \|f_0-f_1\|_{H^k}+\|f_0\|_{H^k}\leq c \inf f_0 +\|f_0\|_{H^k}\leq c\|f_0\|_{\infty}+\|f_0\|_{H^k}\leq \frac{3}{2}\|f_0\|_{H^k}$.
Therefore, upon changing the value of $C$, it holds 
$$\left\|\frac{1}{K^m_\epsilon\star f_1}\right\|_{H^k}\leq \frac{C}{\inf f_0}\sum_{j=0}^k\po\frac{\|f_0\|_{H^k}}{\inf f_0}\pf^{j} .$$

If in addition $f_0=\exp(-\frac{A}{\alpha+1}),$ then  $\inf f_0$ is bounded from below by a constant independent of $\alpha$, and $\|f_0\|_{H^k}$ is bounded from above by a constant independent of $\alpha$, as $0\leq \frac{1}{\alpha+1}\leq 1$. This yields the second statement of the lemma.
\end{proof}

\begin{lem}\label{lem:Kf-f}
Suppose $\epsilon\in(0,1]$, $k\in \N^+, 2k>m,$ $f_0=\exp(-\frac{A}{\alpha+1})$, where $A\in C^{k+3}(\T^m)$. There exists a constant $C>0$ independent of $(\alpha,\epsilon)$, such that, for $\theta\in \N^m, |\theta|=0,1$,
$$ \|K^m_\epsilon\star D^\theta f_0-D^\theta f_0\|_{H^k}\leq \frac{C\epsilon}{\alpha+1},$$

and 

$$ \|K^m_\epsilon\star D^\theta \ln f_0-D^\theta \ln f_0\|_{H^k}\leq \frac{C\epsilon}{\alpha+1}.$$
\end{lem}
\begin{proof}
In this proof,  $C$ denotes a generic constant independent of $(\alpha,\epsilon)$, whose value may vary from one occurrence to another.

By Lemma \ref{lem:error_convolution}, 
\begin{align*}
    \|K^m_\epsilon\star\ln f_0-\ln f_0\|_{H^k}&\le \sum_{\eta\in\N^m,|\eta|=0}^k\|D^\eta(K^m_\epsilon\star\ln f_0- \ln f_0)\|_{2}\\&\le \sum_{\eta\in\N^m,|\eta|=0}^k\|D^\eta(K^m_\epsilon\star\ln f_0- \ln f_0)\|_{\infty}\\
    &= \sum_{\eta\in\N^m,|\eta|=0}^k\|K^m_\epsilon\star(D^\eta\ln f_0)-D^\eta \ln f_0\|_{\infty}\\
&\le Cm^2\epsilon\sum_{\eta\in\N^m,|\eta|=2}^{k+2}\|D^\eta \ln f_0\|_{\infty}\\
&        \leq\frac{Cm^2\epsilon}{\alpha+1}\sum_{\eta\in\N^m,|\eta|=2}^{k+2}\|D^\eta A\|_{\infty}\leq \frac{C\epsilon}{\alpha+1}.
\end{align*}

The other estimates are proven in a similar way by noticing that for $1\leq |\eta|\leq k+3,$ $\|D^\eta \ln f_0\|_\infty$ and $\|D^\eta f_0\|_\infty$ are bounded by $\frac{C}{\alpha+1},$ where $C$ is a constant independent of $\alpha$.
\end{proof}

\begin{lem}\label{lem:KlnKf-lnf}
Suppose that $\epsilon\in(0,1],k\in\N^+,2k>m,$ $f_0=\exp(-\frac{A}{\alpha+1})$, where $A\in C^{k+3}(\T^m)$. There exists a constant $C>0$ independent of $(\alpha,\epsilon)$, such that,
$$ \|K^m_\epsilon\star\ln(K^m_\epsilon\star f_0)-\ln f_0\|_{H^k}\leq \frac{C\epsilon}{\alpha+1}.$$
\end{lem}
\begin{proof}
In the following,  $C$ denotes a generic constant independent of $(\alpha,\epsilon)$, whose value may vary from one occurrence to another.

By the triangle inequality,
$$\|K^m_\epsilon\star\ln(K^m_\epsilon\star f_0)-\ln f_0\|_{H^k}\le \|K^m_\epsilon\star\ln f_0-\ln f_0\|_{H^k} + \|K^m_\epsilon\star\ln(K^m_\epsilon\star f_0)-K^m_\epsilon\star\ln f_0\|_{H^k}.$$
By Lemma~\ref{lem:Kf-f},  $$\|K^m_\epsilon\star\ln f_0-\ln f_0\|_{H^k}\leq\frac{C\epsilon}{\alpha+1}.$$
Besides, by Lemma \ref{lem:convolution},
\begin{align*}
\|K^m_\epsilon\star\ln(K^m_\epsilon\star f_0)-K^m_\epsilon\star\ln f_0\|_{H^k}&\leq  \|\ln(K^m_\epsilon\star f_0)-\ln f_0\|_{H^k}\\&\leq \|\ln(K^m_\epsilon\star f_0)-\ln f_0\|_{\infty}+\sum_{|\theta|=1}\|D^\theta\ln(K^m_\epsilon\star f_0)-D^\theta\ln f_0\|_{H^{k-1}}.    
\end{align*}

We have that, by mean value theorem, Lemma \ref{lem:error_convolution}, and the explicit formula of $D^\eta f_0$, where $\eta\in\N^m,|\eta|=2$,
$$\|\ln(K^m_\epsilon\star f_0)-\ln f_0\|_{\infty}\leq \left\|\frac{1}{f_0}\right\|_\infty\|K^m_\epsilon\star f_0-f_0\|_\infty\leq C\epsilon m^2\left\|\frac{1}{f_0}\right\|_\infty\|D^2f_0\|_\infty\leq \frac{C\epsilon}{\alpha+1}.$$
And, again by triangle inequality, when $|\theta|=1$,
\begin{align*}
\|D^\theta\ln(K^m_\epsilon\star f_0)-D^\theta\ln f_0\|_{H^{k-1}} 
&= \left\|\frac{K^m_\epsilon\star D^{\theta}f_0}{K^m_\epsilon\star f_0}-\frac{D^{\theta}f_0}{f_0}\right\|_{H^{k-1}} \\
&\leq \left\|\frac{K^m_\epsilon\star D^{\theta}f_0-D^{\theta}f_0}{K^m_\epsilon\star f_0}\right\|_{H^{k-1}} + \left\|\frac{D^{\theta}f_0}{K^m_\epsilon\star f_0}-\frac{D^{\theta}f_0}{f_0}\right\|_{H^{k-1}}.
\end{align*}
By Lemma~\ref{lem:product}, Lemma~\ref{lem:inverse_Kf_1}, and Lemma~\ref{lem:Kf-f}, 
\begin{align*}
\left\|\frac{K^m_\epsilon\star D^{\theta}f_0-D^{\theta}f_0}{K^m_\epsilon\star f_0}\right\|_{H^{k-1}}&\leq \left\|\frac{K^m_\epsilon\star D^{\theta}f_0-D^{\theta}f_0}{K^m_\epsilon\star f_0}\right\|_{H^{k}}\\&\leq C\left\|\frac{1}{K^m_\epsilon\star f_0}\right\|_{H^{k}}\left\|{K^m_\epsilon\star D^{\theta}f_0-D^{\theta}f_0}\right\|_{H^{k}} \\& \leq \frac{C\epsilon}{\alpha+1}.   
\end{align*}
Similarly, by noticing that $\|D^\theta f_0\|_{H^k}$ is also bounded from above by a constant independent of $\alpha$,
\begin{align*}
\left\|\frac{D^{\theta}f_0}{K^m_\epsilon\star f_0}-\frac{D^{\theta}f_0}{f_0}\right\|_{H^{k-1}}&\leq \left\|\frac{D^{\theta}f_0}{K^m_\epsilon\star f_0}-\frac{D^{\theta}f_0}{f_0}\right\|_{H^{k}}\\&\leq C\left\|\frac{1}{K^m_\epsilon\star f_0}\right\|_{H^{k}}\left\|\frac{1}{f_0}\right\|_{H^{k}}\left\|D^{\theta}f_0\right\|_{H^{k}}\left\|{K^m_\epsilon\star f_0-f_0}\right\|_{H^{k}} \\& \leq \frac{C\epsilon}{\alpha+1}.   
\end{align*}
Thus $\|\ln(K^m_\epsilon\star f_0)-\ln f_0\|_{H^k} \le \frac{C\epsilon}{\alpha+1}$ and this concludes the proof.
\end{proof}
\begin{lem}\label{lem:h_1h_2}
Suppose that $\epsilon\in(0,1]$. There exists a constant $C > 0$, such that for any $ h_1\in C^1(\T^m),h_2\in L^2(\T^m)$, it holds: 
$$\|h_1K^m_{\epsilon}\star h_2- K^m_{\epsilon}\star(h_1h_2)\|_2\leq C\sqrt{\epsilon m}\|D^1h_1\|_\infty\|h_2\|_2 $$
where, we recall, $\|D^1 h\|_\infty =\max_{|\theta|=1}\|D^\theta h\|_\infty$.
\end{lem}
\begin{proof}

For $x \in \T^m$, using Cauchy-Schwarz inequality and Lemma~\ref{lem:e_var},
\begin{align*}
(h_1(x)K^m_{\epsilon}\star h_2(x)- K^m_{\epsilon}\star(h_1h_2)(x))^2 &=\left(\int_{\T^m} K^m_\epsilon(y)h_2(x-y)(h_1(x)-h_1(x-y))dy\right)^2\\ &\leq \|D^1 h_1\|_\infty^2 \left( \int_{\T^m} K^m_\epsilon(y)|h_2(x-y)\|y|dy\right)^2 \\  &\leq \|D^1 h_1\|_\infty^2  \int_{\T^m} K^m_\epsilon(y)|y|^2dy\int_{\T^m} K^m_\epsilon(y)h_2^2(x-y)dy
\\&\leq Cm\epsilon\|D^1 h_1\|_\infty^2   K^m_\epsilon\star(h_2^2)(x).
\end{align*}
Thus using the fact that $\int_{\T^m}K^m_\epsilon\star(h_2^2)(x)=\int_{\T^m}h_2^2(x)dx$, we obtain that
$$\|h_1K^m_{\epsilon}\star h_2- K^m_{\epsilon}\star(h_1h_2)\|_2\leq \sqrt{C \epsilon m}\|D^1h_1\|_\infty\|h_2\|_2. $$
which yields the result.
\end{proof}

\section{About the dynamics~\eqref{eq:McKean-Voverdamped}}\label{sec:McKean-Voverdamped}

In this appendix, we provide some results on the non-regularized dynamics~\eqref{eq:McKean-Voverdamped} using entropy techniques in the spirit of~\cite{LRS07}. As already mentioned in Section~\ref{sec:ABF}, we do not know how to generalize such results to the non-regularized kinetic Langevin dynamics~\eqref{eq:McKeanVkinetic}.

\subsection{Preliminaries}
Let us consider the non-regularized dynamics~\eqref{eq:McKean-Voverdamped} (taking $\beta=1$ for simplicity):
\begin{align*}
\partial_t \rho &= {\rm div}(\nabla (U +\alpha \ln \rho^1) \rho) + \Delta \rho\\
&={\rm div} \left( \Gamma_\alpha(\rho) \nabla \frac{\rho}{\Gamma_\alpha(\rho)} \right)\\
&={\rm div} \left( \rho \nabla \ln \frac{\rho}{\Gamma_\alpha(\rho)} \right)
\end{align*}
where, we recall
$$\Gamma_\alpha(\rho) \propto \exp(- U -\alpha \ln \rho^1).$$

Notice that $\rho^1$ satisfies the parabolic equation (this can again be seen as a manifestation of the fact that the temperature has been multiplied by a factor $1+\alpha$ along the collective variable)
$$\partial_t \rho^1 = \partial_1(A_t' \rho^1) + (\alpha+1) \partial_{1,1} \rho^1$$
where
$$A_t'(x_1)=\frac{\int \partial_{x_1} U(x_1,x_2) \rho(t,x_1,x_2) \, dx_2 }{\rho^1(t,x_1)}.$$
In particular, one has that for all $t>0$ $\rho^1(t,\cdot)>0$, since it can be written as
$$\rho^1(t,x) = K_t \star \rho^1(0,\cdot) + \int_0^t K_{t-s} \star \int \partial_{x_1} U (x_1,x_2) \rho(s,\cdot,x_2) \, dx_2 \, ds$$
where $K_t$ denotes the heat kernel.
In the following, we assume for simplicity that this is also true at time $t=0$.

We want to prove the long time convergence of $\rho$ towards
\begin{equation}\label{eq:rho_inf}
\rho_\infty \propto \exp\left(-U + \frac{\alpha}{\alpha +1} A\right).
\end{equation}

Recall that $\rho^1_\infty \propto \exp\left(- \frac{A}{\alpha +1} \right)$. In particular,
\begin{align*}
    A_\infty'(x_1)&=\frac{\int \partial_{x_1} U(x_1,x_2) \rho_\infty(x_1,x_2) \, dx_2 }{\rho_\infty^1(x_1)}\\
    &=\frac{\int \partial_{x_1} U(x_1,x_2) \exp(-U(x_1,x_2) \, dx_2 }{\int  \exp(-U(x_1,x_2) \, dx_2}\\
    &=A'(x_1)=-(\alpha+1) (\ln \rho^1_\infty)'.
\end{align*}
Thus, the equation on $\rho^1$ can be rewritten as:
\begin{align*}
    \partial_t \rho^1 &= \partial_1(A_t' \rho^1) + (\alpha+1) \partial_{1,1} \rho^1\\
    &=\partial_1((A_t'-A_\infty') \rho^1) + (\alpha+1) \partial_{1} \left( - \ln(\rho^1_\infty)' \rho^1 + \partial_1 \rho^1 \right)\\
    &=\partial_1((A_t'-A_\infty') \rho^1) + (\alpha+1) \partial_{1} \left(  \rho^1  \partial_1 \ln \frac{\rho^1}{\rho^1_\infty} \right).
\end{align*}

\subsection{Long-time behavior}

One has on the one hand
\begin{align}
\frac{d}{dt} \int \rho \ln \frac{\rho}{\rho_\infty}
&= \int {\rm div} \left( \rho \nabla \ln \frac{\rho}{\Gamma_\alpha(\rho)} \right) \ln \frac{\rho}{\rho_\infty}\nonumber \\
&= - \int  \rho \nabla \ln \frac{\rho}{\Gamma_\alpha(\rho)} \cdot \nabla \ln \frac{\rho}{\rho_\infty} \nonumber \\
&= - \int  \rho \left|  \nabla \ln \frac{\rho}{\rho_\infty} \right|^2 - \int  \rho \nabla \ln \frac{\rho_\infty}{\Gamma_\alpha(\rho)} \cdot \nabla \ln \frac{\rho}{\rho_\infty}. \nonumber
\end{align}
Besides
\begin{align*}
   \nabla \ln \frac{\rho_\infty}{\Gamma_\alpha(\rho)}
    &= \nabla \left( \frac{\alpha}{\alpha+1} A + \alpha \ln \rho^1 \right)=\alpha \nabla \ln \frac{\rho^1}{\rho_\infty^1}.
\end{align*}
Thus,
\begin{align}
\frac{d}{dt} \int \rho \ln \frac{\rho}{\rho_\infty}
&=- \int  \rho \left|  \nabla \ln \frac{\rho}{\rho_\infty} \right|^2 - \alpha \int  \rho \, \partial_1 \ln \frac{\rho^1}{\rho_\infty^1} \partial_1 \ln \frac{\rho}{\rho_\infty}. \label{eq:ent_total}
\end{align}

On the other hand, for the marginal in $x_1$, it holds
\begin{align*}
\frac{d}{dt} \int \rho^1 \ln \frac{\rho^1}{\rho_\infty^1}
&= \int \partial_1((A_t'-A_\infty') \rho^1)  \ln \frac{\rho^1}{\rho_\infty^1} +  (\alpha+1) \int  \partial_{1} \left(  \rho^1  \partial_1 \ln \frac{\rho^1}{\rho^1_\infty} \right) \ln \frac{\rho^1}{\rho_\infty^1}\\
&= - \int (A_t'-A_\infty') \rho^1 \partial_1 \ln \frac{\rho^1}{\rho_\infty^1} -  (\alpha+1) \int  
\rho^1  \left| \partial_1 \ln \frac{\rho^1}{\rho^1_\infty} \right|^2.
\end{align*}
Besides, from standard results on the convergence of ABF like dynamics (see~\cite[Lemmas 4, 5, and 6]{LRS07}), one has for some constant $c>0$
$$\int (A_t'-A_\infty')^2 \rho^1 \le c \frac{\|\partial_{1,2} U\|_\infty^2}{\delta} \int \rho \left| \partial_2 \ln \frac{\rho}{\rho_\infty} \right|^2$$
where $\delta$ is a (uniform in $x_1$) log-Sobolev constant of the conditional measures (indexed by $x_1$): $\rho_\infty/\rho_\infty^1=\exp(-U(x_1,x_2)) dx_2 / \exp(-A(x_1)$. Let us denote
\begin{equation}\label{eq:M}
M=c \frac{\|\partial_{1,2} U\|_\infty^2}{\delta}.
\end{equation}
The smaller is the coupling and the easier is the sampling of the conditional measures (i.e. the larger $\delta$), the smaller is $M$.
Thus, by Young's inequality
\begin{align}
\frac{d}{dt} \int \rho^1 \ln \frac{\rho^1}{\rho_\infty^1}
&= \frac{1}{2(\alpha+1)} \int (A_t'-A_\infty')^2 \rho^1 -  \frac{\alpha+1}{2} \int  
\rho^1  \left| \partial_1 \ln \frac{\rho^1}{\rho^1_\infty} \right|^2 \nonumber \\
&\le \frac{M}{2(\alpha+1)} \int \rho \left| \partial_2 \ln \frac{\rho}{\rho_\infty} \right|^2 -  \frac{\alpha+1}{2} \int  
\rho^1  \left| \partial_1 \ln \frac{\rho^1}{\rho^1_\infty} \right|^2  \label{eq:ent_marginal}
\end{align}

By combining a Young's inequality applied to~\eqref{eq:ent_total}, and~\eqref{eq:ent_marginal}, one gets (for $\epsilon,\gamma >0$ to be chosen):
\begin{align*}
&\frac{d}{dt} \int \rho \ln \frac{\rho}{\rho_\infty}
+\gamma \frac{d}{dt} \int \rho^1 \ln \frac{\rho^1}{\rho_\infty^1}\\
&\le  - \int  \rho \left|  \nabla \ln \frac{\rho}{\rho_\infty} \right|^2 + \epsilon \alpha \int  \rho_1 \left| \partial_1 \ln \frac{\rho^1}{\rho_\infty^1} \right|^2 + \frac{\alpha}{4 \epsilon} \rho \left| \partial_1 \ln \frac{\rho}{\rho_\infty}\right|^2\\
& \quad  + \frac{\gamma M}{2(\alpha+1)} \int \rho \left| \partial_2 \ln \frac{\rho}{\rho_\infty} \right|^2 -  \gamma \frac{\alpha+1}{2} \int  
\rho^1  \left| \partial_1 \ln \frac{\rho^1}{\rho^1_\infty} \right|^2.
\end{align*}
To get negative signs on the right-hand side, one needs
$$\epsilon \alpha < \gamma \frac{\alpha+1}{2}, \quad \frac{\alpha}{4 \epsilon} < 1, \quad \frac{\gamma M}{2(\alpha+1)} < 1.$$
We are typically interested in the large $\alpha$ regime. It is natural to choose $\epsilon=\frac{\alpha+1}{2}$ so that the second constraint disappears, and one needs
$$\alpha < \gamma, \quad \frac{\gamma M}{2(\alpha+1)} < 1$$
which rewrites
$$\alpha< \gamma < \frac{2}{M} (1+\alpha).$$
If $M > 2$ we thus requires
$$\alpha < \frac{2}{M-2},$$
otherwise, there are no constraints on $\alpha$.
Let us choose $\gamma=\frac{1}{M}+\alpha\left(\frac{1}{M}+\frac 1 2 \right)$ (the mean of the upper and lower bounds imposed on $\gamma$). One thus gets an exponential convergence of the entropy to zero with rate
\begin{equation}\label{eq:rate}
r=2\min\left(D,D_1\frac{\alpha+1}{2}  \left( 1 -\frac \alpha \gamma\right)\right)=\min\left(2 D,D_1(\alpha+1) \frac{1+\alpha-\alpha M /2}{1+\alpha+\alpha M /2}\right)
\end{equation}
where, $D$ (resp. $D_1$) is a log-Sobolev inequality constant of $\rho_\infty$ (resp. $\rho_\infty^1$), and $M$ is defined by~\eqref{eq:M}.

We thus have obtained the following result:
\begin{lem}
    Let us consider a smooth solution to the non-regularized dynamics~\eqref{eq:PDEoverdamped}-\eqref{eq:McKean-Voverdamped}. Let us denote by $D$ (resp. $D_1$ and $\delta$) a log-Sobolev constant for $\rho_\infty$ (resp. $\rho_\infty^1$ and the conditional measures $\rho_\infty(x_1,\cdot)/\rho_\infty^1(x_1)$). Let us assume that $\alpha >0$ and $\alpha <\frac{2}{M-2}$ if $M>2$, where $M$ is defined by~\eqref{eq:M}. Then, for any $t \ge 0$,
    $$\mathcal H(\rho_t|\rho_\infty) \le \mathcal H(\rho_0|\rho_\infty) 
    \exp(- r t)$$
    where $r$ is defined by~\eqref{eq:rate}.
\end{lem}

For example, if $M \le 2$, by choosing $\alpha$ sufficiently large, one gets exponential convergence with rate $2D$. If $M>2$, by choosing $\alpha = \frac{1}{M-2}$, one gets exponential convergence with rate $\min\left(2D,D_1 \frac{M-1}{3M-2}\right)$.

\subsection*{Acknowledgements}

P.M. would like to thank Anna Korba for fruitfull discussions. His research is supported by the projects SWIDIMS (ANR-20-CE40-0022) and CONVIVIALITY (ANR-23-CE40-0003) of the French National Research Agency. The works of P.M. and T.L. are supported by project EMC2 from the European Union’s Horizon 2020 research and innovation program (grant agreement No 810367). The works of T.L. are also partially
funded by the Agence Nationale de la Recherche through
the grants ANR-19-CE40-0010-01 (QuAMProcs) and ANR-21-CE40-0006 (SINEQ). The research of X.L. was supported by an internship from CERMICS (Ecole Nationale des Ponts et Chaussées).

\bibliographystyle{plain}  
\bibliography{biblio}

@book{loomis-sternberg-1968,
  title={Advanced calculus},
  author={Loomis, Lynn Harold and Sternberg, Shlomo},
  year={1968},
  publisher={World Scientific}
}

@article{burger2023porous,
  title={Porous medium equation and cross-diffusion systems as limit of nonlocal interaction},
  author={Burger, Martin and Esposito, Antonio},
  journal={Nonlinear Analysis},
  volume={235},
  pages={113347},
  year={2023},
  publisher={Elsevier}
}

@article{deuschel-stroock-1990,
  title={Hypercontractivity and spectral gap of symmetric diffusions with applications to the stochastic {I}sing models},
  author={Deuschel, Jean-Dominique and Stroock, Daniel W.},
  journal={Journal of Functional Analysis},
  volume={92},
  pages={30--48},
  year={1990},
  publisher={Elsevier}
}

@ARTICLE{M33,
       author = {{Guillin}, Arnaud and {Le Bris}, Pierre and {Monmarch{\'e}}, Pierre},
        title = "{Uniform in time propagation of chaos for the 2D vortex model and other singular stochastic systems}",
  journal={Journal of the European Mathematical Society},
  pages={1--28},
  year={2024}
}

@article{10.1214/24-EJP1217,
author = {Pierre Monmarch{\'e} and Zhenjie Ren and Songbo Wang},
title = {{Time-uniform log-Sobolev inequalities and applications to propagation of chaos}},
volume = {29},
journal = {Electronic Journal of Probability},
number = {none},
publisher = {Institute of Mathematical Statistics and Bernoulli Society},
pages = {1 -- 38},
year = {2024},
doi = {10.1214/24-EJP1217},
URL = {https://doi.org/10.1214/24-EJP1217}
}

@book{ane-2000,
  title={Sur les in{\'e}galit{\'e}s de Sobolev logarithmiques},
  author={An{\'e}, C{\'e}cile and Blach{\`e}re, S{\'e}bastien and Chafa{\"\i}, Djalil and Foug{\`e}res, Pierre and Gentil, Ivan and Malrieu, Florent and Roberto, Cyril and Scheffer, Gr{\'e}gory},
  volume={10},
  year={2000},
  publisher={Soci{\'e}t{\'e} math{\'e}matique de France Paris}
}

@book{stein2009real,
  title={Real analysis: measure theory, integration, and Hilbert spaces},
  author={Stein, Elias M. and Shakarchi, Rami},
  year={2009},
  publisher={Princeton University Press}
}

@article{chen-lin-ren-wang-2024,
  title={Uniform-in-time propagation of chaos for kinetic mean field {L}angevin dynamics},
  author={Chen, Fan and Lin, Yiqing and Ren, Zhenjie and Wang, Songbo},
  journal={Electronic Journal of Probability},
  volume={29},
  pages={1--43},
  year={2024},
  publisher={The Institute of Mathematical Statistics and the Bernoulli Society}
}

@article{chizat-2022,
  title={Mean-field {L}angevin dynamics: Exponential convergence and annealing},
  author={Chizat, L{\'e}na{\"\i}c},
  journal={Transactions on Machine Learning Research},
  year={2022}
}

@article{crandall1992user,
  title={User’s guide to viscosity solutions of second order partial differential equations},
  author={Crandall, Michael G. and Ishii, Hitoshi and Lions, Pierre-Louis},
  journal={Bulletin of the American mathematical society},
  volume={27},
  number={1},
  pages={1--67},
  year={1992}
}

@article{nesterov,
author = {Yi-An Ma and Niladri S. Chatterji and Xiang Cheng and Nicolas Flammarion and Peter L. Bartlett and Michael I. Jordan},
title = {{Is there an analog of Nesterov acceleration for gradient-based MCMC?}},
volume = {27},
journal = {Bernoulli},
number = {3},
publisher = {Bernoulli Society for Mathematical Statistics and Probability},
pages = {1942 -- 1992},
year = {2021},
doi = {10.3150/20-BEJ1297},
URL = {https://doi.org/10.3150/20-BEJ1297}
}

@book{evans2010partial,
  title={Partial Differential Equations},
  author={Evans, Lawrence C},
  volume={19},
  year={2010},
  publisher={American Mathematical Soc.}
}

@article{chen2022uniform,
author = {Fan Chen and Zhenjie Ren and Songbo Wang},
title = {{Uniform-in-time propagation of chaos for mean field Langevin dynamics}},
volume = {61},
journal = {Annales de l'Institut Henri Poincaré, Probabilités et Statistiques},
number = {4},
publisher = {Institut Henri Poincaré},
pages = {2357 -- 2404},
year = {2025},
doi = {10.1214/24-AIHP1499},
URL = {https://doi.org/10.1214/24-AIHP1499}
}

@article{suzuki2024mean,
  title={Mean-field {L}angevin dynamics: Time-space discretization, stochastic gradient, and variance reduction},
  author={Suzuki, Taiji and Wu, Denny and Nitanda, Atsushi},
  journal={Advances in Neural Information Processing Systems},
  volume={36},
  year={2024}
}

@book{dupuis1997weak,
  title={A Weak Convergence Approach to the Theory of Large Deviations},
  author={Dupuis, Paul and Ellis, Richard S.},
  year={1997},
  publisher={Wiley}
}

@article{Hu2021MeanField,
  title={Mean-field {L}angevin dynamics and energy landscape of neural networks},
  author={Hu, Kaitong and Ren, Zhenjie and {\v{S}}i{\v{s}}ka, David and Szpruch, {\L}ukasz},
  journal={Annales de l’Institut Henri Poincar{\'e}, Probabilit{\'e}s et Statistiques},
  volume={57},
  number={4},
  pages={2043--2065},
  year={2021},
  doi={10.1214/20-AIHP1140},
  url={https://doi.org/10.1214/20-AIHP1140}
}

@ARTICLE{M31,
       author = {{Guillin}, Arnaud and {Le Bris}, Pierre and {Monmarch{\'e}}, Pierre},
title = {{Convergence rates for the {V}lasov-{F}okker-{P}lanck equation and uniform in time propagation of chaos in non convex cases}},
volume = {27},
journal = {Electronic Journal of Probability},
number = {none},
publisher = {Institute of Mathematical Statistics and Bernoulli Society},
pages = {1 -- 44},
year = {2022},
doi = {10.1214/22-EJP853},
URL = {https://doi.org/10.1214/22-EJP853}
}

@article{monmarche2023note,
     author = {Pierre Monmarch\'e},
     title = {A note on a {Vlasov{\textendash}Fokker{\textendash}Planck} equation with non-symmetric interaction},
     journal = {Annales de la Facult\'e des sciences de Toulouse : Math\'ematiques},
     pages = {243--255},
     year = {2025},
     publisher = {Universit\'e de Toulouse, Toulouse},
     volume = {Ser. 6, 34},
     number = {2},
     doi = {10.5802/afst.1812},
     language = {en},
     url = {https://afst.centre-mersenne.org/articles/10.5802/afst.1812/}
}

@article{Schuh,
author = {Katharina Schuh},
title = {{Global contractivity for Langevin dynamics with distribution-dependent forces and uniform in time propagation of chaos}},
volume = {60},
journal = {Annales de l'Institut Henri Poincaré, Probabilités et Statistiques},
number = {2},
publisher = {Institut Henri Poincaré},
pages = {753 -- 789},
year = {2024},
doi = {10.1214/22-AIHP1337},
URL = {https://doi.org/10.1214/22-AIHP1337}
}

@ARTICLE{M35,
       author = {{Monmarch{\'e}}, Pierre and {Ramil}, Mouad},
title = {{Overdamped limit at stationarity for non-equilibrium Langevin diffusions}},
volume = {27},
journal = {Electronic Communications in Probability},
number = {none},
publisher = {Institute of Mathematical Statistics and Bernoulli Society},
pages = {1 -- 8},
year = {2022},
doi = {10.1214/22-ECP447},
URL = {https://doi.org/10.1214/22-ECP447}
}

@ARTICLE{GuillinWuZhang,
author = {Arnaud Guillin and Wei Liu and Liming Wu and Chaoen Zhang},
title = {{Uniform {P}oincaré and logarithmic {S}obolev inequalities for mean field particle systems}},
volume = {32},
journal = {The Annals of Applied Probability},
number = {3},
publisher = {Institute of Mathematical Statistics},
pages = {1590 -- 1614},
year = {2022},
doi = {10.1214/21-AAP1707},
URL = {https://doi.org/10.1214/21-AAP1707}
}

@article{MonmarcheReygner,
  title={Local convergence rates for Wasserstein gradient flows and McKean-Vlasov equations with multiple stationary solutions},
  author={Monmarch{\'e}, Pierre and Reygner, Julien},
  journal={Probability Theory and Related Fields},
  pages={1--59},
  year={2025},
  publisher={Springer}
}

@article {MenzSchlichting,
    AUTHOR = {Menz, Georg and Schlichting, Andr\'e},
     TITLE = {Poincar\'e and logarithmic {S}obolev inequalities by
              decomposition of the energy landscape},
   JOURNAL = {Ann. Probab.},
  FJOURNAL = {The Annals of Probability},
    VOLUME = {42},
      YEAR = {2014},
    NUMBER = {5},
     PAGES = {1809--1884},
      ISSN = {0091-1798},
   MRCLASS = {60J60 (35A23 35J15 35P15)},
  MRNUMBER = {3262493},
       DOI = {10.1214/14-AOP908},
       URL = {http://dx.doi.org/10.1214/14-AOP908},
}

@inproceedings{lelievre2012two,
  title={Two mathematical tools to analyze metastable stochastic processes},
  author={Leli{\`e}vre, Tony},
  booktitle={Numerical Mathematics and Advanced Applications 2011: Proceedings of ENUMATH 2011, the 9th European Conference on Numerical Mathematics and Advanced Applications, Leicester, September 2011},
  pages={791--810},
  year={2012},
  organization={Springer}
}

@book{ambrosio2005gradient,
  title={Gradient flows: in metric spaces and in the space of probability measures},
  author={Ambrosio, Luigi and Gigli, Nicola and Savar{\'e}, Giuseppe},
  year={2005},
  publisher={Springer Science \& Business Media}
}

@article{Lelievre_twoscale,
title = {A general two-scale criteria for logarithmic {S}obolev inequalities},
journal = {Journal of Functional Analysis},
volume = {256},
number = {7},
pages = {2211-2221},
year = {2009},
issn = {0022-1236},
doi = {https://doi.org/10.1016/j.jfa.2008.09.019},
url = {https://www.sciencedirect.com/science/article/pii/S0022123608004084},
author = {Tony Lelièvre}
}

@article {Holley,
    AUTHOR = {Holley, Richard A. and Kusuoka, Shigeo and Stroock, Daniel W.},
     TITLE = {Asymptotics of the spectral gap with applications to the
              theory of simulated annealing},
   JOURNAL = {J. Funct. Anal.},
  FJOURNAL = {Journal of Functional Analysis},
    VOLUME = {83},
      YEAR = {1989},
    NUMBER = {2},
     PAGES = {333--347},
      ISSN = {0022-1236},
     CODEN = {JFUAAW},
   MRCLASS = {60J60 (47A10 47F05 60K99)},
  MRNUMBER = {995752 (92d:60081)},
MRREVIEWER = {Shuenn Jyi Sheu},
       DOI = {10.1016/0022-1236(89)90023-2},
       URL = {http://dx.doi.org/10.1016/0022-1236(89)90023-2},
}

@article {Villani2009,
    AUTHOR = {Villani, Cédric},
     TITLE = {Hypocoercivity},
   JOURNAL = {Mem. Amer. Math. Soc.},
  FJOURNAL = {Memoirs of the American Mathematical Society},
    VOLUME = {202},
      YEAR = {2009},
    NUMBER = {950},
     PAGES = {iv+141},
      ISSN = {0065-9266},
     CODEN = {MAMCAU},
      ISBN = {978-0-8218-4498-4},
   MRCLASS = {35Q84 (35H10 76N10 76P05 82C70)},
  MRNUMBER = {2562709 (2011e:35381)},
MRREVIEWER = {Andr{\'a}s Domokos},
       DOI = {10.1090/S0065-9266-09-00567-5},
       URL = {http://dx.doi.org/10.1090/S0065-9266-09-00567-5},
}

@ARTICLE{2024arXiv240917901M,
       author = {{Monmarch{\'e}}, Pierre},
        title = "{Uniform log-Sobolev inequalities for mean field particles beyond flat-convexity}",
      journal = {to appear in Annales de la Facult\'e des sciences de Toulouse : Math\'ematiques},
        year = {2026}
}

@article{jourdain2010existence,
  title={Existence, uniqueness and convergence of a particle approximation for the adaptive biasing force process},
  author={Jourdain, Benjamin and Leli{\`e}vre, Tony and Roux, Rapha{\"e}l},
  journal={ESAIM: Mathematical Modelling and Numerical Analysis},
  volume={44},
  number={5},
  pages={831--865},
  year={2010},
  publisher={EDP Sciences}
}

@article{carrillo2024fisher,
  title={Fisher-Rao Gradient Flow: Geodesic Convexity and Functional Inequalities},
  author={Carrillo, Jos{\'e} A and Chen, Yifan and Huang, Daniel Zhengyu and Huang, Jiaoyang and Wei, Dongyi},
  journal={arXiv preprint arXiv:2407.15693},
  year={2024}
}

@article{hyvarinen2005estimation,
  title={Estimation of non-normalized statistical models by score matching.},
  author={Hyv{\"a}rinen, Aapo and Dayan, Peter},
  journal={Journal of Machine Learning Research},
  volume={6},
  number={4},
  year={2005}
}

@article{gianazza2009wasserstein,
  title={The {W}asserstein gradient flow of the {F}isher information and the quantum drift-diffusion equation},
  author={Gianazza, Ugo and Savar{\'e}, Giuseppe and Toscani, Giuseppe},
  journal={Archive for rational mechanics and analysis},
  volume={194},
  number={1},
  pages={133--220},
  year={2009},
  publisher={Springer}
}

@ARTICLE{GuillinMonmarche,
       author = {{Guillin}, Arnaud and {Monmarch{\'e}}, Pierre},
        title = "{Uniform long-time and propagation of chaos estimates for mean field kinetic particles in non-convex landscapes}",
      journal = {J Stat Phys},
      volume={185},
      number={15},
      year={2021},
      doi={https://doi.org/10.1007/s10955-021-02839-6}
}

@ARTICLE{Delgadino,
       author = {{Delgadino}, Mat{\'\i}as G. and {Gvalani}, Rishabh S. and {Pavliotis}, Grigorios A. and {Smith}, Scott A.},
        title = "{Phase Transitions, Logarithmic {S}obolev Inequalities, and Uniform-in-Time Propagation of Chaos for Weakly Interacting Diffusions}",
      journal = {Communications in Mathematical Physics},
         year = 2023,
        month = jul,
       volume = {401},
       number = {1},
        pages = {275-323},
          doi = {10.1007/s00220-023-04659-z},
archivePrefix = {arXiv},
       eprint = {2112.06304},
 primaryClass = {math.PR},
       adsurl = {https://ui.adsabs.harvard.edu/abs/2023CMaPh.401..275D},
}

@ARTICLE{Songbo,
       author = {{Wang}, Songbo},
        title = "{Uniform log-Sobolev inequalities for mean field particles with flat-convex energy}",
      journal = {arXiv e-prints},
         year = 2024,
        month = aug,
          eid = {arXiv:2408.03283},
        pages = {arXiv:2408.03283},
          doi = {10.48550/arXiv.2408.03283},
archivePrefix = {arXiv},
       eprint = {2408.03283},
 primaryClass = {math.PR},
       adsurl = {https://ui.adsabs.harvard.edu/abs/2024arXiv240803283W}
}

@article{M15,
author = {Monmarch{\'e}, Pierre},
     TITLE = {Long-time behaviour and propagation of chaos for mean field
              kinetic particles},
   JOURNAL = {Stochastic Process. Appl.},
  FJOURNAL = {Stochastic Processes and their Applications},
    VOLUME = {127},
      YEAR = {2017},
    NUMBER = {6},
     PAGES = {1721--1737},
      ISSN = {0304-4149},
   MRCLASS = {60J60 (35B40 35K58 35Q83 35R09 82B40)},
  MRNUMBER = {3646428},
       DOI = {10.1016/j.spa.2016.10.003},
       URL = {https://doi.org/10.1016/j.spa.2016.10.003},
}

@Article{ActaNumerica,
  author  = {Tony Lelièvre and Gabriel Stoltz},
  title   = {Partial differential equations and stochastic methods in molecular dynamics},
  journal = {Acta Numerica},
  year    = {2016},
  volume  = {25},
}

@article{wang2001efficient,
  title={Efficient, multiple-range random walk algorithm to calculate the density of states},
  author={Wang, Fugao and Landau, David P.},
  journal={Physical review letters},
  volume={86},
  number={10},
  pages={2050},
  year={2001},
  publisher={APS}
}

@article{chevallier2020wang,
  title={Wang-{L}andau Algorithm: an adapted random walk to boost convergence},
  author={Chevallier, Augustin and Cazals, Fr{\'e}d{\'e}ric},
  journal={Journal of Computational Physics},
  volume={410},
  pages={109366},
  year={2020},
  publisher={Elsevier}
}

@article{liu2016stein,
  title={Stein variational gradient descent: A general purpose bayesian inference algorithm},
  author={Liu, Qiang and Wang, Dilin},
  journal={Advances in neural information processing systems},
  volume={29},
  year={2016}
}

@article{BBM,
author = {Michel Bena{\"i}m and Charles-Edouard Br{\'e}hier and Pierre Monmarch{\'e}},
title = {{Analysis of an Adaptive Biasing Force method based on self-interacting dynamics}},
volume = {25},
journal = {Electronic Journal of Probability},
number = {none},
publisher = {Institute of Mathematical Statistics and Bernoulli Society},
pages = {1 -- 28},
year = {2020},
doi = {10.1214/20-EJP490},
URL = {https://doi.org/10.1214/20-EJP490}
}

@article{fort2017self,
  title={Self-healing umbrella sampling: convergence and efficiency},
  author={Fort, Gersende and Jourdain, Benjamin and Leli{\`e}vre, Tony and Stoltz, Gabriel},
  journal={Statistics and Computing},
  volume={27},
  pages={147--168},
  year={2017},
  publisher={Springer}
}

@article{marsili2006self,
  title={Self-healing umbrella sampling: a non-equilibrium approach for quantitative free energy calculations},
  author={Marsili, Simone and Barducci, Alessandro and Chelli, Riccardo and Procacci, Piero and Schettino, Vincenzo},
  journal={The Journal of Physical Chemistry B},
  volume={110},
  number={29},
  pages={14011--14013},
  year={2006},
  publisher={ACS Publications}
}

@article{Metadynamics,
author = {Alessandro Laio  and Michele Parrinello },
title = {Escaping free-energy minima},
journal = {Proceedings of the National Academy of Sciences},
volume = {99},
number = {20},
pages = {12562-12566},
year = {2002},
doi = {10.1073/pnas.202427399},
URL = {https://www.pnas.org/doi/abs/10.1073/pnas.202427399},
eprint = {https://www.pnas.org/doi/pdf/10.1073/pnas.202427399}}

@article{welltempered,
  title = {Well-Tempered Metadynamics: A Smoothly Converging and Tunable Free-Energy Method},
  author = {Barducci, Alessandro and Bussi, Giovanni and Parrinello, Michele},
  journal = {Phys. Rev. Lett.},
  volume = {100},
  issue = {2},
  pages = {020603},
  numpages = {4},
  year = {2008},
  month = {Jan},
  publisher = {American Physical Society},
  doi = {10.1103/PhysRevLett.100.020603},
  url = {https://link.aps.org/doi/10.1103/PhysRevLett.100.020603}
}

@Article{Comer,
  author    = {Comer, Jeffrey and Gumbart, James C. and H{\'e}nin, J{\'e}r{\^o}me and Leli{\`e}vre, Tony and Pohorille, Andrew and Chipot, Christophe},
  title     = {{The Adaptive Biasing Force Method: everything you always wanted to know but were afraid to ask}},
  journal   = {{Journal of Physical Chemistry B}},
  year      = {2015},
  volume    = {119},
  number    = {3},
  pages     = {1129--1151},
  month     = {January},
  publisher = {{American Chemical Society}},
}

@Book{FreeEnergy,
  title     = {Free Energy Computations},
  publisher = {Imperial College Press},
  year      = {2010},
  author    = {Lelièvre, Tony and Rousset, Mathias and Stoltz, Gabriel},
}

@article{bolley2010trend,
  title={Trend to equilibrium and particle approximation for a weakly selfconsistent {V}lasov-{F}okker-{P}lanck equation},
  author={Bolley, Fran{\c{c}}ois and Guillin, Arnaud and Malrieu, Florent},
  journal={ESAIM: Mathematical Modelling and Numerical Analysis},
  volume={44},
  number={5},
  pages={867--884},
  year={2010},
  publisher={EDP Sciences}
}

@article{herau2007short,
  title={Short and long time behavior of the {F}okker--{P}lanck equation in a confining potential and applications},
  author={H{\'e}rau, Fr{\'e}d{\'e}ric},
  journal={Journal of Functional Analysis},
  volume={244},
  number={1},
  pages={95--118},
  year={2007},
  publisher={Elsevier}
}

@article{herau2016global,
  title={On global existence and trend to the equilibrium for the {V}lasov--{P}oisson--{F}okker--{P}lanck system with exterior confining potential},
  author={H{\'e}rau, Fr{\'e}d{\'e}ric and Thomann, Laurent},
  journal={Journal of Functional Analysis},
  volume={271},
  number={5},
  pages={1301--1340},
  year={2016},
  publisher={Elsevier}
}

@article{M40,
     author = {Monmarch\'e, Pierre},
     title = {Wasserstein contraction and {Poincar\'e} inequalities for elliptic diffusions with high diffusivity},
     journal = {Annales Henri Lebesgue},
     pages = {941--973},
     publisher = {\'ENS Rennes},
     volume = {6},
     year = {2023},
     doi = {10.5802/ahl.182},
     language = {en},
     url = {https://ahl.centre-mersenne.org/articles/10.5802/ahl.182/}
}

@ARTICLE{Elementary,
       author = {{Monmarch{\'e}}, Pierre},
	title = {Elementary coupling approach for non-linear perturbation of {M}arkov processes with mean-field jump mechanisms and related problems},
	DOI= "10.1051/ps/2023002",
	url= "https://doi.org/10.1051/ps/2023002",
	journal = {ESAIM: PS},
	year = 2023,
	volume = 27,
	pages = "278-323",
}

@ARTICLE{JournelMonmarcheFV,
  title={Uniform convergence of the Fleming--Viot process in a hard killing metastable case},
  author={Journel, Lucas and Monmarch{\'e}, Pierre},
  journal={The Annals of Applied Probability},
  volume={35},
  number={2},
  pages={1019--1048},
  year={2025},
  publisher={Institute of Mathematical Statistics}
}

@article{mattingly2002ergodicity,
  title={Ergodicity for SDEs and approximations: locally {L}ipschitz vector fields and degenerate noise},
  author={Mattingly, Jonathan C. and Stuart, Andrew M. and Higham, Desmond J.},
  journal={Stochastic processes and their applications},
  volume={101},
  number={2},
  pages={185--232},
  year={2002},
  publisher={Elsevier}
}

@article {Talay,
    AUTHOR = {Talay, Denis},
     TITLE = {Stochastic {H}amiltonian systems: exponential convergence to
              the invariant measure, and discretization by the implicit
              {E}uler scheme},
   JOURNAL = {Markov Process. Related Fields},
  FJOURNAL = {Markov Processes and Related Fields},
    VOLUME = {8},
      YEAR = {2002},
    NUMBER = {2},
     PAGES = {163--198},
      ISSN = {1024-2953},
   MRCLASS = {60H10 (37J99 60H35)},
  MRNUMBER = {1924934 (2003e:60129)},
MRREVIEWER = {Jan I. Seidler},
}

@book {BakryGentilLedoux,
    AUTHOR = {Bakry, Dominique and Gentil, Ivan and Ledoux, Michel},
     TITLE = {Analysis and geometry of {M}arkov diffusion operators},
    SERIES = {Grundlehren der Mathematischen Wissenschaften [Fundamental
              Principles of Mathematical Sciences]},
    VOLUME = {348},
 PUBLISHER = {Springer, Cham},
      YEAR = {2014},
     PAGES = {xx+552},
      ISBN = {978-3-319-00226-2; 978-3-319-00227-9},
   MRCLASS = {60J25 (58J65 60J35 60J60)},
  MRNUMBER = {3155209},
MRREVIEWER = {Ming Liao},
       DOI = {10.1007/978-3-319-00227-9},
       URL = {http://dx.doi.org/10.1007/978-3-319-00227-9},
}

@ARTICLE{LRS07,
       author = {{Leli\`evre}, Tony and {Rousset}, Mathias and
         {Stoltz}, Gabriel},
        title = "{Long-time convergence of an {A}daptive {B}iasing {F}orce method}",
      journal = {Nonlinearity},
      volume = {21},
      publisher = {IOP Publishing},
      year={2008}
}

@ARTICLE{Chipot2011,
       author = {{Chipot}, Chris and {Leli{\`e}vre}, Tony},
        title = "{Enhanced sampling of multidimensional free-energy landscapes using adaptive biasing forces}",
      journal = {SIAM Journal on Applied Mathematics},
         year = {2011},
        volume = {71},
        number = {5},
        pages = {1673--1695},
          doi = {10.1137/10080600X},
      hal_id = {hal-00510981}
}

@Article{Henin-Chipot,
  author  = {Hénin,Jérôme and Chipot,Christophe},
  title   = {Overcoming free energy barriers using unconstrained molecular dynamics simulations},
  journal = {The Journal of Chemical Physics},
  year    = {2004},
  volume  = {121},
  number  = {7},
  pages   = {2904-2914},
  doi     = {10.1063/1.1773132}
}

@Article{Darve-Pohorille,
  author  = {Darve,Eric and Pohorille,Andrew},
  title   = {Calculating free energies using average force},
  journal = {The Journal of Chemical Physics},
  year    = {2001},
  volume  = {115},
  number  = {20},
  pages   = {9169-9183},
  doi     = {10.1063/1.1410978}
}

@book{robert-casella-2004,
  title={Monte Carlo Statistical Methods},
  author={Robert, Christian P. and Casella, George},
  year={2004},
  publisher={Springer-Verlag}
}

@book{bach-2024,
  title={Learning theory from first principles},
  author={Bach, Francis},
  year={2024},
  publisher={MIT press}
}

@article{henin-lelievre-shirts-valsson-delemotte-22,
title={Enhanced Sampling Methods for Molecular Dynamics Simulations},
volume={4},
number={1},
journal={Living Journal of Computational Molecular Science},
author={Hénin, Jérôme and Lelièvre, Tony and Shirts, Michael R. and Valsson, Omar and Delemotte, Lucie},
year={2022},
pages={1583}
}

@Book{chipot-pohorille-07,
  editor =		 {Chipot, Chris and Pohorille, Andrew},
  title = 		 {Free Energy Calculations},
  publisher = 	 {Springer},
  year = 		 2007,
  volume =		 86,
  series =		 {Springer Series in Chemical Physics}
}

@Article{lelievre-rousset-stoltz-07-b,
  author = 		 {Leli\`evre, Tony and Rousset, Mathias and Stoltz, Gabriel},
  title = 		 {Computation of free energy profiles with adaptive parallel dynamics},
  journal = 	 {J. Chem. Phys.},
  year = 		 2007,
  volume =		 126,
  pages =		 134111
}

@Article{fort-jourdain-lelievre-stoltz-18,
  author = 	 {Fort, Gersende and Jourdain, Benjamin and Leli\`evre, Tony and Stoltz, Gabriel},
  title = 	 {Convergence and Efficiency of Adaptive Importance Sampling Techniques with Partial Biasing},
  journal = 	 {J. Stat. Phys},
  year = 	 2018,
  volume =		 171,
  pages =		 {220-268}
}

@Article{fort-jourdain-kuhn-lelievre-stoltz-15,
  author = 	 {Fort, Gersende and Jourdain, Benjamin and Kuhn, Esetelle and Leli\`evre, Tony and Stoltz, Gabriel},
  title = 	 {Convergence  of the {W}ang-{L}andau algorithm},
  journal = 	 {Math. Comput.},
  year = 	 2015,
  volume =		 84,
  number =		 295,
  pages =		 {2297-2327}
}

@article{benaim-brehier-monmarche-20,
  title={Analysis of an Adaptive Biasing Force method based on self-interacting dynamics},
  author={Bena{\"\i}m, Michel and Br{\'e}hier, Charles-Edouard and Monmarch{\'e}, Pierre},
  year={2020},
  journal = 	 {Electron. J. Probab.},
  volume = 	 25,
  pages = 	 {1-28}
}

@article{benaim-brehier-16,
  title={Convergence of adaptive biasing potential methods for diffusions},
  author={Bena{\"\i}m, Michel and Br{\'e}hier, Charles-{\'E}douard},
  journal={Comptes Rendus. Math{\'e}matique},
  volume={354},
  number={8},
  pages={842--846},
  year={2016}
}

@article{ehrlacher-lelievre-monmarche-22,
  title={Adaptive force biasing algorithms: new convergence results and tensor approximations of the bias},
  author={Ehrlacher, Virginie and Leli{\`e}vre, Tony and Monmarch{\'e}, Pierre},
  journal={The Annals of Applied Probability},
  volume={32},
  number={5},
  pages={3850--3888},
  year={2022},
  publisher={Institute of Mathematical Statistics}
}

@article{lelievre-maurin-monmarche-22,
  title={The adaptive biasing force algorithm with non-conservative forces and related topics},
  author={Leli{\`e}vre, Tony and Maurin, Lise and Monmarch{\'e}, Pierre},
  journal={ESAIM: Mathematical Modelling and Numerical Analysis},
  volume={56},
  number={2},
  pages={529--564},
  year={2022},
  publisher={EDP Sciences}
}

@Article{lelievre-minoukadeh-11,
  author = 	 {Leli{\`e}vre, Tony and Minoukadeh, Kimiya},
  title = 	 {Long-time convergence of an Adaptive Biasing Force method: 
the bi-channel case},
  journal = 	 {Archive for Rational Mechanics and Analysis},
  year =         2011,
  volume =		 202,
  number =		 1,
  pages =		 {1-34}
}

@Article{chopin-lelievre-stoltz-12,
  author = 	 {Chopin, Nicolas and Leli\`{e}vre, Tony and Stoltz, Gabriel},
  title = 	 {Free energy methods for {B}ayesian inference: efficient exploration of univariate {G}aussian mixture posteriors},
  journal =      {Stat. Comput.},
  year = 	 2012,
  volume =       22,
  number =       4,
  pages =        {897-916}
}

@article{valsson2016enhancing,
  title={Enhancing important fluctuations: Rare events and metadynamics from a conceptual viewpoint},
  author={Valsson, Omar and Tiwary, Pratyush and Parrinello, Michele},
  journal={Annual review of physical chemistry},
  volume={67},
  number={1},
  pages={159--184},
  year={2016},
  publisher={Annual Reviews}
}

\end{document}